\documentclass[reqno,11pt]{amsart}
\usepackage{amssymb}
\usepackage{amsmath}
\usepackage{amsthm}
\usepackage{mathrsfs}
\usepackage{mathtools}
\usepackage{enumitem}
\usepackage{amsfonts}
\usepackage{xcolor}
\usepackage{graphicx}
\usepackage{tikz}
\usepackage{comment}
\usepackage{microtype}
\usepackage[
    colorlinks = true, 
    linkcolor={blue},
	citecolor={blue},
	urlcolor={black!30!blue}
]{hyperref}

\title[Admissible Fourier lengths and KAM reducibility]
{Admissible Fourier Lengths, KAM Reducibility, and Spectral Applications}

\date{\today}

\author[X.\ Wang]{Xueyin Wang}
\address{[X.\ Wang] Department of Mathematics, Texas A\&M University, College Station, TX 77843, USA}
\email{\href{mailto:xueyin@tamu.edu}{xueyin@tamu.edu}}

\author[J.\ You]{Jiangong You}
\address{[J.\ You] Chern Institute of Mathematics and LPMC, Nankai University, Tianjin 300071, China}
\email{\href{mailto:jyou@nankai.edu.cn}{jyou@nankai.edu.cn}}

\theoremstyle{plain}
\newtheorem{theorem}{Theorem}[section]
\newtheorem{corollary}[theorem]{Corollary}
\newtheorem{lemma}[theorem]{Lemma}
\newtheorem{proposition}[theorem]{Proposition}

\theoremstyle{definition}
\newtheorem{definition}[theorem]{Definition}

\newtheorem{problem}[theorem]{Problem}

\numberwithin{equation}{section}

\begin{document}
\begin{abstract}
We develop a perturbative KAM reducibility theory for one-frequency $\mathrm{SL}(2,\mathbb{R})$ cocycles based on an admissible Fourier length $\ell$. The regularity relevant to the iteration is measured by positive adapted Fourier width rather than ordinary smoothness in the Euclidean length $|n|$. The same length governs Fourier decay, truncation and resonance scales, and the arithmetic condition controlling the small divisors. This framework contains the classical analytic and Gevrey settings, while non-monotone choices of $\ell$ allow classical nowhere differentiable Weierstrass-type perturbations and continuous perturbations outside every positive H\"older class. As spectral applications, we obtain purely absolutely continuous spectrum for every phase and $1/2$-H\"older continuity of the integrated density of states for the associated quasiperiodic Schr\"odinger operators. The Aubry dual has pure point spectrum for Lebesgue almost every dual phase, with eigenfunctions exponentially localized in the metric induced by $\ell$. We also construct nowhere differentiable quasiperiodic potentials with purely absolutely continuous Cantor spectrum.
\end{abstract}

\maketitle
\setcounter{tocdepth}{1}
\tableofcontents

\section{Introduction}
KAM reducibility is a basic mechanism connecting quasiperiodic dynamics with the spectral theory of quasiperiodic Schr\"odinger operators. In the perturbative analytic setting, reducibility of the Schr\"odinger cocycle produces Bloch waves and leads to purely absolutely continuous spectrum. When exact reducibility is obstructed by infinitely many resonances, quantitative almost reducibility can provide the estimates needed for the same spectral conclusion.  Beyond the analytic setting \cite{MR470318,MR2969277,MR1167299,MR2846380}, perturbative reducibility and spectral results have been extended to Gevrey, ultradifferentiable, and sufficiently smooth classes, with arithmetic conditions adapted to the corresponding regularity; see, for example, \cite{MR3936094,MR4273633,MR4257810,MR4477429}.

Classical KAM reducibility is usually organized by ordinary regularity. Since Fourier decay is measured by the Euclidean length $|n|$, ordinary smoothness appears to be the determining regularity governing whether the KAM iteration can be closed. From the viewpoint developed here, this is only part of the picture. What is essential is regularity measured in an appropriate Fourier algebra, together with control over the small divisors in that same algebra.

We describe such an algebra by an admissible Fourier length $\ell$. Membership in $\mathcal{A}_{\sigma}^{\ell}$ for some $\sigma>0$ means exponential Fourier decay with respect to $\ell$, with $\sigma$ measuring the adapted Fourier width. The usual analytic, Gevrey, and ultradifferentiable classes correspond to monotone choices $\ell(n)=\Lambda(|n|)$. The essential point is that $\ell$ need not be monotone. A function may therefore be very rough, or even nowhere differentiable, in the ordinary sense while having positive width in the adapted Fourier algebra. The classical Weierstrass function is a basic example: its principal Fourier coefficients occur at the modes $\pm b^{m}$ and decay exponentially in their adapted length, although the function itself can be nowhere differentiable.

The main conceptual point is therefore that the regularity relevant to the KAM iteration is positive adapted Fourier width rather than ordinary smoothness. The subadditivity of $\ell$ controls the modes created by multiplication and conjugation, its ball growth controls truncation and resonance counting, and the same length is used in the arithmetic condition for the small divisors. Thus the essential input is the matched pair of positive adapted Fourier width and subexponential arithmetic loss in the $\ell$-scale. 

This viewpoint allows KAM reducibility to persist far below the classical smoothness threshold. In particular, it applies to nowhere differentiable Weierstrass-type cocycles and to a class of continuous cocycles outside every positive H\"older class. For the associated quasiperiodic Schr\"odinger operators, we obtain purely absolutely continuous spectrum for every phase, $1/2$-H\"older continuity of the integrated density of states, purely absolutely continuous Cantor spectrum, and pure point spectrum for the Aubry dual with localization measured by the admissible Fourier length. This is in sharp contrast with the generic continuous case, where the zero Lyapunov set has zero Lebesgue measure and absolutely continuous spectrum is absent \cite{MR2181094}.

We now introduce the operators and cocycles under consideration. Let $\alpha\in\mathbb{R}\setminus\mathbb{Q}$, $x\in\mathbb{T}=\mathbb{R}/\mathbb{Z}$, and $V\in C^{0}(\mathbb{T},\mathbb{R})$. We consider the quasiperiodic Schr\"odinger operator
\begin{equation*}
	[H_{x,\alpha,V}u]_{n}=u_{n+1}+u_{n-1}+V(x+n\alpha)u_{n},\qquad n\in\mathbb{Z},
\end{equation*}
and its associated cocycle
\begin{equation*}
	S_{E}^{V}(x)=\begin{pmatrix}E-V(x)&-1\\1&0\end{pmatrix}\in\mathrm{SL}(2,\mathbb{R}).
\end{equation*}
A cocycle $(\alpha,A)$ is reducible if there exist $B:2\mathbb{T}\to\mathrm{SL}(2,\mathbb{R})$ and a constant $A_{*}\in\mathrm{SL}(2,\mathbb{R})$ such that
\begin{equation*}
	B(x+\alpha)^{-1}A(x)B(x)=A_{*}.
\end{equation*}
It is almost reducible in a Banach topology if its conjugacy class contains cocycles arbitrarily close to constants in that topology.

\subsection{Admissible Fourier lengths and Fourier structures}

To make the preceding viewpoint precise, an admissible Fourier length is a symmetric, positive, and subadditive function on the integer Fourier lattice,
\begin{equation*}
    \ell: \mathbb{Z}\to[0,\infty),\quad \ell(0)=0,\quad \ell(-n)=\ell(n),\quad \ell(n_{1}+n_{2})\leqslant\ell(n_{1})+\ell(n_{2}).
\end{equation*}
Its balls satisfy the power-subexponential growth estimate
\begin{equation}\label{eq:intro-ball-growth}
    Q_{\ell}(N)\coloneqq\#\{n\in\mathbb{Z}:\ell(n)\leqslant N\}\leqslant C_{\ell}\exp(c_{\ell}N^{\rho}),\qquad 0<\rho<1.
\end{equation}
The integer lattice is the natural domain because the perturbations, the small divisors, and the resonance labels are one-periodic. For conjugacies defined on $2\mathbb{T}$, we use the canonical extension $\bar{\ell}$ on $\frac{1}{2}\mathbb{Z}$ constructed in \hyperref[lem:half-lattice-extension]{Lemma~\ref*{lem:half-lattice-extension}} and suppress the bar in the notation. 

For $\sigma>0$, the associated one-periodic Fourier algebra is
\begin{equation}\label{eq:intro-adapted-algebra}
	\mathcal{A}_{\sigma}^{\ell}=\bigg\{f(x)=\sum_{n\in\mathbb{Z}}\widehat{f}(n)e^{2\pi i n x}:\|f\|_{\sigma,\ell}\coloneqq\sum_{n\in\mathbb{Z}}|\widehat{f}(n)|e^{\sigma\ell(n)}<\infty\bigg\}.
\end{equation}
Subadditivity makes $\mathcal{A}_{\sigma}^{\ell}$ a Banach algebra. The matched arithmetic condition is
\begin{equation}\label{eq:intro-adapted-diophantine}
	\|n\alpha\|_{\mathbb{T}}\geqslant\gamma\exp\big(-\tau\ell(n)^{\rho}\big),\qquad n\in\mathbb{Z}\setminus\{0\}.
\end{equation}
The ball growth and the small-divisor loss are therefore subexponential in the same adapted scale and can be absorbed by a reduction of adapted width.

This common scale is essential. A standard Diophantine condition gives
\begin{equation}\label{eq:intro-standard-diophantine}
\|n\alpha\|_{\mathbb{T}}\geqslant\gamma |n|^{-\tau},
\end{equation}
and controls the divisor in terms of $|n|$, while the Fourier coefficient is measured in $\ell(n)$. For a rough adapted length, $|n|^{\tau}$ need not be subexponential in $\ell(n)$ and cannot in general be absorbed in an infinite KAM iteration with positive limiting width.

The Euclidean choice $\ell(n)=|n|$ recovers the analytic Wiener algebra, and suitable monotone choices recover Gevrey classes. The principal non-monotone example is the classical Weierstrass function
\begin{equation*}
W_{a,b}(x)=\sum_{m\geqslant 0}a^{m}\cos(2\pi b^{m}x),\qquad b\in\mathbb{N},\qquad b\geqslant 2,\qquad b^{-1}<a<1.
\end{equation*}
It is H\"older continuous with exponent
\begin{equation*}
\beta=-\frac{\log a}{\log b}\in(0,1),
\end{equation*}
and is nowhere differentiable by Hardy's theorem \cite{MR1501044}.

For a fixed $b\geqslant 2$, define
\begin{equation*}
\ell_{b}(n)=\inf\bigg\{\sum_{m\geqslant 0}(m+1)|k_{m}|:n=\sum_{m\geqslant 0}k_{m}b^{m},\quad k_{m}\in\mathbb{Z}\bigg\}.
\end{equation*}
We write $\mathcal{A}_{\sigma}^{(b)}=\mathcal{A}_{\sigma}^{\ell_b}$ and whenever this notation is used on $2\mathbb{T}$, the canonical half-lattice extension is understood. Then $\ell_{b}(b^{m})\leqslant m+1$ and
\begin{equation*}
Q_{b}(N)\coloneqq\#\{n\in\mathbb{Z}:\ell_{b}(n)\leqslant N\}\leqslant e^{4\sqrt{N}}.
\end{equation*}
Moreover,
\begin{equation*}
\sum_{m\geqslant 0}a^{m}e^{\sigma\ell_{b}(b^{m})}\leqslant e^{\sigma}\sum_{m\geqslant 0}(ae^{\sigma})^{m}<\infty
\end{equation*}
whenever $ae^{\sigma}<1$, so $W_{a,b}\in\mathcal{A}_{\sigma}^{\ell_{b}}$. On the indices $n=\pm b^{m}$, the standard condition permits the inverse loss
\begin{equation*}
|b^{m}|^{\tau}=\exp(\tau m\log b),
\end{equation*}
which cannot be absorbed by an arbitrarily small loss of adapted width. By contrast, the adapted condition
\begin{equation*}
\|n\alpha\|_{\mathbb{T}}\geqslant\gamma e^{-\tau\sqrt{\ell_{b}(n)}},\qquad \tau>4,
\end{equation*}
has only subexponential loss in $\ell_{b}(n)$.

More lacunary choices yield continuous sampling functions outside every positive H\"older class, as shown in \hyperref[cor:continuous-nonholder]{Corollary~\ref*{cor:continuous-nonholder}}. Strict sparsity is not required. For any admissible length $\ell$, all Fourier modes may occur as long as
\begin{equation*}
\sum_{n\in\mathbb{Z}}|\widehat{V}(n)|e^{\sigma\ell(n)}<\infty
\end{equation*}
for some $\sigma>0$. Thus one may add Fourier coefficients at modes of large adapted length, provided that they are sufficiently small. The sparse series in \hyperref[prop:sparse-class]{Proposition~\ref*{prop:sparse-class}} gives a useful class of examples.

\subsection{Main results}

At the core of the paper is a quantitative KAM reducibility theorem for general $\mathrm{SL}(2,\mathbb{R})$ cocycles in adapted Fourier algebras. Its specialization to Schr\"odinger cocycles yields the spectral results stated below. The precise hypotheses are given in \hyperref[def:admissible]{Definition~\ref*{def:admissible}}.

\begin{theorem}[KAM almost reducibility in an adapted Fourier algebra]
\label{thm:intro-kam}
    Let $\ell$ be an admissible Fourier length of exponent $0<\rho<1$ and ball-growth constant $c_{\ell}$. Let $\tau>c_{\ell}$, $\gamma>0$, and $\alpha\in\mathrm{DC}_{\ell}(\gamma,\tau)$. Fix a compact subset $\mathcal{K}\subseteq\mathrm{SL}(2,\mathbb{R})$ and widths $0<\sigma_{*}<\sigma_{0}$. There exists
    \begin{equation*}
        \varepsilon_{0}=\varepsilon_{0}(\ell,\rho,\tau,\gamma,\sigma_{0},\sigma_{*},\mathcal{K})>0
    \end{equation*}
    such that, whenever $A\in\mathcal{K}$ and $F\in\mathcal{A}_{\sigma_{0}}^{\ell}(\mathbb{T},\mathrm{sl}(2,\mathbb{R}))$ satisfies $\|F\|_{\sigma_{0}}<\varepsilon_{0}$, the cocycle $(\alpha,Ae^{F})$ is almost reducible. More precisely, there are $B_{j}\in \mathcal{A}_{\sigma_{j}}^{\ell}(2\mathbb{T},\mathrm{SL}(2,\mathbb{R}))$, constants $A_{j}\in\mathrm{SL}(2,\mathbb{R})$, widths $\sigma_{j}\downarrow\sigma_{*}$, and $F_{j}\in\mathcal{A}_{\sigma_{j}}^{\ell}(\mathbb{T},\mathrm{sl}(2,\mathbb{R}))$ such that
    \begin{equation*}
        B_{j}(x+\alpha)^{-1}Ae^{F(x)}B_{j}(x)=A_{j}e^{F_{j}(x)}, \qquad \|F_{j}\|_{\sigma_{j}}\longrightarrow 0.
    \end{equation*}
    If only finitely many KAM steps are resonant, then the cocycle is reducible.

    If $\mathcal{A}=Ae^{F}$ is not uniformly hyperbolic, then the Lyapunov exponent $L(\alpha,\mathcal{A})$ vanishes. Moreover, the map $A\mapsto L(\alpha,A)$ is locally $1/2$-H\"older continuous at $\mathcal{A}$ in the $C^{0}$ topology.
\end{theorem}

The quantitative form in \hyperref[KAM]{Theorem~\ref*{KAM}} provides suitable control of the conjugacies at every scale. Together with the eigenvalue estimates, this yields the vanishing and $1/2$-H\"older continuity of the Lyapunov exponent and provides the input for the spectral statements.

\begin{theorem}[Schr\"odinger reducibility and spectral consequences]
\label{thm:intro-schrodinger}
    Assume the arithmetic hypotheses of \hyperref[thm:intro-kam]{Theorem~\ref*{thm:intro-kam}}. For every $\sigma_{0}>0$ there is $\varepsilon_{*}>0$ such that the following holds. If $V\in\mathcal{A}_{\sigma_{0}}^{\ell}(\mathbb{T},\mathbb{R})$ and $\|V\|_{\sigma_{0}}<\varepsilon_{*}$, then:
    \begin{enumerate}
        \item \label{item:intro-SOAR}For every $E\in\Sigma_{\alpha,V}$, the cocycle $(\alpha,S_{E}^{V})$ is almost reducible;

        \item \label{item:SOAC}$H_{x,\alpha,V}$ has purely absolutely continuous spectrum for every $x\in\mathbb{T}$;

        \item \label{item:SOEL}For every $x\in\mathbb{T}$ and for the universal spectral measure $\mu_{x}=\mu_{\delta_{0}}+\mu_{\delta_{1}}$, $(\alpha,S_{E}^{V})$ is reducible to an elliptic constant for $\mu_{x}$-almost every $E$;

        \item \label{item:SOIDS} There exists $C>0$ such that, for all $E_{1},E_{2}\in\mathbb{R}$, the integrated density of states satisfies
        \begin{equation*}
            |N_{\alpha,V}(E_{2})-N_{\alpha,V}(E_{1})|\leqslant C|E_{2}-E_{1}|^{\frac{1}{2}}.
        \end{equation*}
    \end{enumerate}
\end{theorem}

\begin{corollary}[Pure point spectrum of the Aubry dual operator]
\label{cor:intro-dual}
    Under the assumptions and the smallness condition of \hyperref[thm:intro-schrodinger]{Theorem~\ref*{thm:intro-schrodinger}}, define
\begin{equation*}
    [\widehat{H}_{\alpha,\theta,V}u]_{n} = \sum_{k\in\mathbb{Z}}\widehat{V}(k)u_{n-k} + 2\cos 2\pi(\theta+n\alpha)u_{n}.
\end{equation*}
Then, for Lebesgue almost every dual phase $\theta\in\mathbb{T}$, the operator $\widehat{H}_{\alpha,\theta,V}$ has pure point spectrum. More precisely, it admits an orthonormal eigenbasis $\{u^{(j)}\}_{j\geqslant 1}$ such that, for every $0<\sigma<\frac{\sigma_{0}}{2}$,
\begin{equation*}
    \sum_{n\in\mathbb{Z}}|u_{n}^{(j)}|e^{\sigma\ell(n)}<\infty.
\end{equation*}
Equivalently, the eigenfunctions are exponentially localized with respect to the adapted lattice metric $d_{\ell}(m,n)=\ell(m-n)$. Moreover, if $V\notin C^{r}(\mathbb{T},\mathbb{R})$ for some $0<r<1$, then for every $s>r$ and every $j\geqslant 1$,
\begin{equation*}
    \sum_{n\in\mathbb{Z}}(1+|n|)^{s}|u_{n}^{(j)}|=\infty.
\end{equation*}
In particular, if $V$ belongs to no positive H\"older class, then none of these eigenfunctions satisfies a positive polynomially weighted $\ell^{1}$ estimate.
\end{corollary}

\begin{corollary}[Classical Weierstrass potentials]
\label{cor:intro-weierstrass}
    Let $b\geqslant 2$, $b^{-1}<a<1$, and $\tau>4$. For every $\alpha\in\mathrm{DC}_{b}(\tau)$ there exists $\lambda_{0}=\lambda_{0}(a,b,\alpha,\tau)>0$ such that the conclusions of \hyperref[thm:intro-schrodinger]{Theorem~\ref*{thm:intro-schrodinger}} and \hyperref[cor:intro-dual]{Corollary~\ref*{cor:intro-dual}} hold for $V=\lambda W_{a,b}$ whenever $|\lambda|<\lambda_{0}$. If $\lambda\neq0$ and
    \begin{equation*}
        \beta=-\frac{\log a}{\log b},
    \end{equation*}
    then, for Lebesgue almost every dual phase $\theta\in\mathbb{T}$, every eigenfunction in the basis from \hyperref[cor:intro-dual]{Corollary~\ref*{cor:intro-dual}} satisfies
    \begin{equation*}
        \sum_{n\in\mathbb{Z}}(1+|n|)^{s}|u_{n}|=\infty
    \end{equation*}
    for every $s>\beta$.
\end{corollary}

\begin{corollary}[Residual nowhere differentiable purely absolutely continuous Cantor potentials]
\label{cor:intro-cantor}
    Under the assumptions of \hyperref[cor:intro-weierstrass]{Corollary~\ref*{cor:intro-weierstrass}}, fix $0<\sigma_{0}<-\log a$ and a non-zero $\lambda$ so small that $\|\lambda W_{a,b}\|_{\sigma_{0}}<\frac{\varepsilon_{*}}{2}$, where $\varepsilon_{*}$ is the threshold in \hyperref[thm:intro-schrodinger]{Theorem~\ref*{thm:intro-schrodinger}}. Define
    \begin{equation*}
        \mathcal{X}_{\sigma_{0}}^{1}=C^{1}(\mathbb{T},\mathbb{R})\cap\mathcal{A}_{\sigma_{0}}^{(b)}(\mathbb{T},\mathbb{R}),\qquad \|g\|_{\mathcal{X}_{\sigma_{0}}^{1}}=\|g\|_{C^{1}}+\|g\|_{\sigma_{0}}.
    \end{equation*}
    There exists $\delta_{*}>0$ such that, for every $0<\delta\leqslant\delta_{*}$, the set of $g$ in the open ball
    \begin{equation*}
        \mathcal{U}_{\delta}=\big\{g\in\mathcal{X}_{\sigma_{0}}^{1}: \|g\|_{\mathcal{X}_{\sigma_{0}}^{1}}<\delta\big\}
    \end{equation*}
    for which every bounded gap allowed by the one-frequency gap-labeling theorem is open for $V=\lambda W_{a,b}+g$ is residual and dense in $\mathcal{U}_{\delta}$. For every such $g$,
    \begin{equation*}
        V\in C^{\beta}(\mathbb{T}),\qquad \beta=-\frac{\log a}{\log b},
    \end{equation*}
    $V$ is nowhere differentiable, $\Sigma_{\alpha,V}$ is a Cantor set, and $H_{\alpha,x,V}$ has purely absolutely continuous spectrum for every phase $x$. In particular, for every $\eta>0$, one may choose such a $g$ with $\|g\|_{\mathcal{X}_{\sigma_{0}}^{1}}<\eta$.
\end{corollary}

\subsection{Significance of the results}

To the best of our knowledge, \hyperref[thm:intro-kam]{Theorem~\ref*{thm:intro-kam}} is the first reducibility and almost reducibility theorem for nowhere differentiable quasiperiodic cocycles and for continuous cocycles outside every positive H\"older class. Likewise, \hyperref[thm:intro-schrodinger]{Theorem~\ref*{thm:intro-schrodinger}} and \hyperref[cor:intro-cantor]{Corollary~\ref*{cor:intro-cantor}} provide the first purely absolutely continuous and purely absolutely continuous Cantor examples at this level of regularity. In particular, \hyperref[cor:intro-cantor]{Corollary~\ref*{cor:intro-cantor}} answers Problem~9.14.6 of Damanik and Fillman \cite{MR4840232}.

For the classical Weierstrass potential, the vanishing of the Lyapunov exponent on the spectrum rules out the positive Lyapunov assertion in Conjecture~2.6 of \cite{MR5081746} for every sufficiently small nonzero coupling. 

The Aubry dual result exhibits a complementary phenomenon: the hopping may have no positive polynomial moment, while the eigenfunctions remain exponentially localized in the metric induced by $\ell$. This adapted localization may be much weaker than exponential or stretched exponential localization in the Euclidean lattice distance.

\section{Admissible Fourier lengths and adapted Fourier algebras}\label{sec:algebra}

The KAM scheme is organized around a single Fourier length $\ell$. The same length controls Fourier decay, truncation, the growth of the set of relevant modes, and the small divisors. We first introduce the abstract Fourier scale and its basic algebra estimates, and then formulate the corresponding full measure arithmetic condition. We next discuss the $b$-adic model underlying the Weierstrass potential and a general sparse construction. These are examples of the abstract theory, which applies to every function in the corresponding adapted Fourier algebra.

\subsection{Adapted lengths and Fourier algebras}

\begin{definition}[Admissible Fourier length]\label{def:admissible}
    Fix $0<\rho<1$. A function \begin{equation*}
        \ell:\mathbb{Z}\to[0,\infty)
    \end{equation*} is called an \emph{admissible Fourier length of exponent $\rho$} if
    \begin{equation*}
        \ell(0)=0,\qquad \ell(-n)=\ell(n),\qquad \ell(n_{1}+n_{2})\leqslant\ell(n_{1})+\ell(n_{2}),
    \end{equation*}
    for all $n_{1},n_{2}\in\mathbb{Z}$. We further require $\ell(n)>0$ for $n\neq 0$ and constants $C_{\ell}\geqslant1$ and $c_{\ell}>0$ such that\footnote{The proof below uses the explicit power-subexponential bound in \eqref{eq:abstract-ball}. A formulation under the weaker assumption $\log Q_{\ell}(N)=o(N)$ would require a common sublinear majorant for the Fourier-ball growth and the small-divisor loss, together with a corresponding redesign of the KAM scales. No such extension is claimed here.}
    \begin{equation}\label{eq:abstract-ball}
        Q_{\ell}(N)\coloneqq\#\{n\in\mathbb{Z}:\ell(n)\leqslant N\}\leqslant C_{\ell}\exp(c_{\ell}N^{\rho}),\qquad N\geqslant 1.
    \end{equation}
\end{definition}

\begin{lemma}[Canonical half-lattice extension]\label{lem:half-lattice-extension}
    Let $\ell$ be an admissible Fourier length on $\mathbb{Z}$. Define
    \begin{equation}\label{eq:canonical-half-extension}
        \bar{\ell}(r)
        =\inf_{m\in\mathbb{Z}}
        \bigg\{\ell(m)+\frac{\ell(1)}{2}|2r-2m|\bigg\},
        \qquad r\in\frac{1}{2}\mathbb{Z}.
    \end{equation}
    Then $\bar{\ell}$ is finite-valued, positive, symmetric, and subadditive on $\frac{1}{2}\mathbb{Z}$, and
    \begin{equation}\label{eq:half-extension-restriction}
        \bar{\ell}(n)=\ell(n),\qquad n\in\mathbb{Z},
        \qquad
        \bar{\ell}\bigg(\frac{1}{2}\bigg)=\frac{\ell(1)}{2}.
    \end{equation}
\end{lemma}
\begin{proof}
    Finiteness and symmetry are immediate. For $n,m\in\mathbb Z$, subadditivity gives
    \begin{equation*}
        \ell(n)\leqslant \ell(m)+\ell(n-m)
        \leqslant \ell(m)+|n-m|\ell(1).
    \end{equation*}
    Taking the infimum over $m$, and then choosing $m=n$, proves
    $\bar{\ell}(n)=\ell(n)$. If $r\in\mathbb Z+1/2$, then
    $|2r-2m|\geqslant1$ for every $m$, which proves positivity and the formula
    for $\bar{\ell}(1/2)$. Thus \eqref{eq:half-extension-restriction} holds. Subadditivity follows by testing
    \eqref{eq:canonical-half-extension} at $m_1+m_2$ and using the triangle
    inequalities for $\ell$ and $|\cdot|$.
\end{proof}

\begin{definition}[Adapted Fourier algebra]\label{def:adapted-algebra}
    Let $\ell$ be an admissible Fourier length. For $\sigma>0$, the space
    \begin{equation*}
        \mathcal{A}_{\sigma}^{\ell}=\bigg\{f(x)=\sum_{n\in\mathbb{Z}}\widehat{f}(n)e^{2\pi i n x}: \|f\|_{\sigma,\ell}\coloneqq\sum_{n\in\mathbb{Z}}|\widehat{f}(n)|e^{\sigma\ell(n)}<\infty\bigg\}
    \end{equation*}
    is called the \emph{adapted Fourier algebra of width $\sigma$ associated with $\ell$}. The nested family
    \begin{equation*}
        \mathcal{A}^{\ell}=\{\mathcal{A}_{\sigma}^{\ell}\}_{\sigma>0}
    \end{equation*}
    is called the \emph{adapted Fourier algebra scale generated by $\ell$ }. For matrix-valued functions, we use the same notation with the Euclidean operator norm on the Fourier coefficients. In particular, this norm is submultiplicative and  $\|P^{-1}\|=\|P\|$ for every $P\in\mathrm{SL}(2,\mathbb{C})$. When no confusion is possible, we abbreviate $\|\cdot\|_{\sigma,\ell}$ to $\|\cdot\|_{\sigma}$.
\end{definition}

Now we provide the basic algebra estimates.
\begin{proposition}\label{banachalgebra}
    Let $\ell$ satisfy the symmetry, subadditivity, and finiteness assumptions in \hyperref[def:admissible]{Definition~\ref*{def:admissible}}, and let $\sigma>0$. Then 
    \begin{enumerate}
        \item $\mathcal{A}_{\sigma}^{\ell}$ is a Banach algebra and
    \begin{equation*}
        \|fg\|_{\sigma}\leqslant \|f\|_{\sigma}\|g\|_{\sigma}\quad \text{for any} \quad f,g\in\mathcal{A}_{\sigma}^{\ell}.
    \end{equation*}

    \item If $0<\sigma'<\sigma$ and 
    \begin{equation*}
        [\mathcal{T}_{N}f](x)\coloneqq\sum_{\ell(n)\leqslant N}\widehat{f}(n)e^{2\pi inx}, \qquad \mathcal{R}_{N}\coloneqq \mathrm{Id}-\mathcal{T}_{N},
    \end{equation*}
    then
    \begin{equation}\label{resdest}
        \|\mathcal{R}_{N}f\|_{\sigma'}\leqslant e^{-(\sigma-\sigma')N}\|f\|_{\sigma}.
    \end{equation}

    \item Moreover, if $\|f\|_{\sigma}\leqslant 1/2$, then
    \begin{equation*}
        \|e^{f}-\mathrm{I}-f\|_{\sigma}\leqslant 2\|f\|_{\sigma}^{2},\qquad \|e^{\pm f}\|_{\sigma}\leqslant 1+2\|f\|_{\sigma}.
    \end{equation*}
    \end{enumerate}
\end{proposition}
\begin{proof}
    Define
    \begin{equation*}
        T:\mathcal{A}_{\sigma}^{\ell}\to \ell^{1}(\mathbb{Z}),
        \qquad
        T(f)=\bigl(\widehat{f}(n)e^{\sigma\ell(n)}\bigr)_{n\in\mathbb{Z}}.
    \end{equation*}
    The map $T$ is an isometric isomorphism onto $\ell^{1}(\mathbb{Z})$.
    Hence $\mathcal{A}_{\sigma}^{\ell}$ is complete. For
    $f,g\in\mathcal{A}_{\sigma}^{\ell}$, subadditivity of $\ell$ gives
    \begin{equation*}
        \begin{aligned}
            \|fg\|_{\sigma}
            &\leqslant
            \sum_{n\in\mathbb{Z}}\sum_{k\in\mathbb{Z}}
            |\widehat f(k)|\,|\widehat g(n-k)|e^{\sigma\ell(n)}\\
            &\leqslant
            \sum_{n\in\mathbb{Z}}\sum_{k\in\mathbb{Z}}
            |\widehat f(k)|e^{\sigma\ell(k)}
            |\widehat g(n-k)|e^{\sigma\ell(n-k)}\\
            &\leqslant \|f\|_{\sigma}\|g\|_{\sigma}.
        \end{aligned}
    \end{equation*}
    This proves the first assertion.

    If $\ell(n)>N$, then
    \begin{equation*}
        e^{\sigma'\ell(n)}
        \leqslant e^{-(\sigma-\sigma')N}e^{\sigma\ell(n)}.
    \end{equation*}
    Therefore,
    \begin{equation*}
        \|\mathcal{R}_{N}f\|_{\sigma'}
        \leqslant e^{-(\sigma-\sigma')N}
        \sum_{\ell(n)>N}|\widehat f(n)|e^{\sigma\ell(n)}
        \leqslant e^{-(\sigma-\sigma')N}\|f\|_{\sigma},
    \end{equation*}
    which is \eqref{resdest}.

    Finally, the Banach algebra property implies that the exponential series
    converges absolutely. If $\|f\|_{\sigma}\leqslant 1/2$, then
    \begin{equation*}
        \|e^{f}-\mathrm{I}-f\|_{\sigma}
        \leqslant
        \sum_{m=2}^{\infty}\frac{\|f\|_{\sigma}^{m}}{m!}
        \leqslant
        \frac{\|f\|_{\sigma}^{2}}{1-\|f\|_{\sigma}}
        \leqslant 2\|f\|_{\sigma}^{2}.
    \end{equation*}
    The same estimate applied to $-f$ yields
    \begin{equation*}
        \|e^{\pm f}\|_{\sigma}
        \leqslant 1+\|f\|_{\sigma}+2\|f\|_{\sigma}^{2}
        \leqslant 1+2\|f\|_{\sigma}.
    \end{equation*}
\end{proof}

For a function defined on $2\mathbb{T}$, let $\bar{\ell}$ be the extension in \hyperref[lem:half-lattice-extension]{Lemma~\ref*{lem:half-lattice-extension}} and write
\begin{equation*}
    f(x)=\sum_{r\in\frac{1}{2}\mathbb{Z}}\widehat{f}(r)e^{2\pi \mathrm{i}rx},
    \qquad
    \|f\|_{\sigma,\ell}
    =\sum_{r\in\frac{1}{2}\mathbb{Z}}|\widehat{f}(r)|e^{\sigma\bar{\ell}(r)}.
\end{equation*}
We denote the resulting Banach algebra by $\mathcal{A}_{\sigma}^{\ell}(2\mathbb{T})$. Thus, in every space on $2\mathbb{T}$, the symbol $\ell$ in the norm means its canonical extension $\bar{\ell}$ and the bar is suppressed only to avoid overloading the notation. The same convolution proof gives the Banach algebra estimates.

% \begin{remark}
% (Hilbert-space formulation)
% The adapted Fourier algebra defined in \hyperref[def:adapted-algebra]{Definition~\ref*{def:adapted-algebra}} can be replaced by
%  \begin{equation*}
%        \mathcal{H}_{\sigma}^{\ell}=\bigg\{f(x)=\sum_{n\in\mathbb{Z}}\widehat{f}(n)e^{2\pi i n x}: \|f\|_{\mathcal{H}_{\sigma}^{\ell}}\coloneqq \bigg(\sum_{n\in\mathbb{Z}}|\widehat{f}(n)|^2e^{2\sigma\ell(n)}\bigg)^{1/2}<\infty\bigg\}.
%     \end{equation*}
% This gives a Hilbert-space formulation.
% In general, $\mathcal{H}_{\sigma}^{\ell} $  need not be an algebra at a fixed value of $\sigma$. Nevertheless, under the subexponential growth assumption on the Fourier balls, then $\mathcal{H}^{\ell}= \{\mathcal{H}^{\ell}_{\sigma}\}_{\sigma>0}$ satisfies that, for every $0<\sigma'<\sigma$,
% \begin{equation*}
%     \|fg\|_{\mathcal{H}_{\sigma'}^{\ell}} \leqslant C_{\sigma-\sigma'} \|f\|_{\mathcal{H}_{\sigma}^{\ell}} \|g\|_{\mathcal{H}_{\sigma}^{\ell}}.
% \end{equation*}
% Thus $\mathcal{A}_{\sigma}^{\ell}\subseteq \mathcal{H}_{\sigma}^{\ell}\subseteq \mathcal{A}_{\sigma'}^{\ell}$. Indeed, the loss of adapted Fourier width is compensated by the summability of $e^{-(\sigma-\sigma')\ell(n)}$, which follows from the subexponential growth of the Fourier balls. Thus the theory could alternatively be formulated on $\mathcal{H}^{\ell}$, with an arbitrarily small loss of adapted Fourier width in the nonlinear estimates. We do not use this Hilbert-space formulation in the present paper, since the weighted $\ell^{1}(\mathbb{Z})$ Banach algebra introduced above is sufficient for our purposes and leads to simpler KAM estimates.
% \end{remark}

\subsection{The adapted arithmetic class}

The arithmetic condition is measured in exactly the same scale as the Fourier length. For $z\in\mathbb{C}$, set
\begin{equation*}
    \|z\|_{\mathbb{T}}=\inf_{m\in\mathbb{Z}}|z-m|.
\end{equation*}
For $\tau>0$ and $\gamma>0$, define
\begin{equation*}
    \begin{aligned}
        \mathrm{DC}_{\ell}(\gamma,\tau)&=\Big\{\alpha\in\mathbb{T}: \|n\alpha\|_{\mathbb{T}}\geqslant \gamma e^{-\tau\ell(n)^{\rho}} \text{ for every }n\in\mathbb{Z}\setminus\{0\}\Big\},\\
        \mathrm{DC}_{\ell}(\tau)&=\bigcup_{\gamma>0}\mathrm{DC}_{\ell}(\gamma,\tau).
    \end{aligned}
\end{equation*}

\begin{proposition}[Full measure]\label{prop:full-measure}
    If $\ell$ is admissible and $\tau>c_{\ell}$, then $\mathrm{DC}_{\ell}(\tau)$ has full Lebesgue measure.
\end{proposition}

\begin{proof}
    Choose $\tau_{0}$ with $c_{\ell}<\tau_{0}<\tau$. For fixed $n\neq 0$,
    \begin{equation*}
        \operatorname{Leb}\big\{\alpha\in\mathbb{T}:\|n\alpha\|_{\mathbb{T}}<e^{-\tau_{0}\ell(n)^{\rho}}\big\}\leqslant 2e^{-\tau_{0}\ell(n)^{\rho}}.
    \end{equation*}
    For each $N\geqslant1$, let
\begin{equation*}
    \Theta_{N}
    =
    \bigcup_{N-1<\ell(n)\leqslant N}
    \{\alpha\in\mathbb{T}: \|n\alpha\|_{\mathbb{T}}<e^{-\tau_{0} \ell(n)^{\rho}}\}.
\end{equation*}
Then by \eqref{eq:abstract-ball}, we obtain
\begin{equation*}
    \sum_{N\geqslant 1}\operatorname{Leb}(\Theta_{N})\leqslant 2C_{\ell}\sum_{N\geqslant 1} \exp(c_{\ell} N^{\rho}-\tau_{0}(N-1)^{\rho})<\infty.
\end{equation*}
    The Borel--Cantelli lemma shows that almost every $\alpha$ violates the desired inequality only for finitely many $n$. For such an irrational $\alpha$, a positive constant $\gamma$ absorbs the finite exceptional set and the change from $\tau_{0}$ to $\tau$.
\end{proof}

\subsection{The \texorpdfstring{$b$}{b}-adic model}

For an integer $b\geqslant 2$, define
\begin{equation*}
    \begin{aligned}
        \ell_{b}(n)=\inf\bigg\{\sum_{m\geqslant 0}(m+1)|k_{m}|: n=\sum_{m\geqslant 0}k_{m}b^{m},\quad k_{m}\in\mathbb{Z}\bigg\}.
    \end{aligned}
\end{equation*}
The infimum exists because the sums are nonnegative integers and the representation using $k_{0}=n$ is always available. When a $2\mathbb{T}$ space is used, $\ell_b$ is understood through the canonical extension in \hyperref[lem:half-lattice-extension]{Lemma~\ref*{lem:half-lattice-extension}}.

\begin{proposition}[Admissibility and growth of $\ell_b$]\label{prop:b-adic}
    Let $b\geqslant2$. Then $\ell_b$ is an admissible Fourier length on $\mathbb{Z}$ with exponent $\rho=1/2$, and
    \begin{equation}\label{eq:b-ball}
        Q_{b}(N)\coloneqq \#\{n:\ell_{b}(n)\leqslant N\}\leqslant e^{4\sqrt{N}}.
    \end{equation}
    Moreover, for every $n\in\mathbb{Z}$,
    \begin{equation}\label{eq:log-square}
        c_{b}\big(\log(1+|n|)\big)\leqslant\ell_{b}(n)\leqslant C_{b}\big(\log(1+|n|)\big)^{2},
    \end{equation}
    where $c_{b}=1/\log b$ and $C_{b}=10 b$.
\end{proposition}
\begin{proof}
    It is clear that $0\leqslant\ell_{b}(n)<\infty$ for every $n\in\mathbb{Z}$, that $\ell_{b}(0)=0$, and that $\ell_{b}(n)>0$ for $n\neq 0$. 
    
    If $n=\sum_{m\geqslant 0}k_{m}b^{m}$, then $-n=\sum_{m\geqslant 0}(-k_{m})b^{m}$. Obviously
    \begin{equation*}
        \sum_{m\geqslant 0}(m+1)|k_{m}|=\sum_{m\geqslant 0}(m+1)|-k_{m}|.
    \end{equation*}
    Taking the infimum over all valid sequences $\{k_{m}\}$ on both sides proves that $\ell_{b}(-n)=\ell_{b}(n)$.

    We next prove the triangle inequality for $\ell_{b}$. Write
    \begin{equation*}
        n_{1}=\sum_{m\geqslant 0}k_{m}b^{m},\quad n_{2}=\sum_{m\geqslant 0}r_{m}b^{m}.
    \end{equation*}
    Then
    \begin{equation*}
        n_{1}+n_{2}=\sum_{m\geqslant 0} (k_{m}+r_{m})b^{m}.
    \end{equation*}
    Thus,
    \begin{equation*}
        \ell_{b}(n_{1}+n_{2})\leqslant \sum_{m\geqslant 0}(m+1)|k_{m}+r_{m}|\leqslant \sum_{m\geqslant 0}(m+1)|k_{m}|+\sum_{m\geqslant 0}(m+1)|r_{m}|.
    \end{equation*}
    Taking the infimum over all valid sequences $\{k_{m}\}$ and $\{r_{m}\}$ yields $\ell_{b}(n_{1}+n_{2})\leqslant\ell_{b}(n_{1})+\ell_{b}(n_{2})$. Thus $\ell_{b}$ is subadditive.

    We next count the ball. Assume $N\geqslant 1$ since \eqref{eq:b-ball} is trivial when $N=0$. Assume that $\ell_{b}(n)\leqslant N$. Then there exists a finitely supported sequence $\{k_{m}\}_{m\geqslant 0}$ such that
    \begin{equation*}
        n=\sum_{m\geqslant 0}k_{m}b^{m}\quad \text{and}\quad \sum_{m\geqslant 0}(m+1)|k_{m}|\leqslant N.
    \end{equation*}
    Thus,
    \begin{equation*}
        Q_{b}(N)\leqslant \# \bigg\{\{k_{m}\}_{m\geqslant 0}: \sum_{m\geqslant 0}(m+1)|k_{m}|\leqslant N\bigg\}.
    \end{equation*}
    We use exponential sums to bound this lattice count. For any fixed $t>0$, if $\sum_{m\geqslant 0}(m+1)|k_{m}|\leqslant N$, then 
    \begin{equation*}
        1\leqslant \exp \bigg(tN-t\sum_{m\geqslant 0}(m+1)|k_{m}|\bigg).
    \end{equation*}
    Consequently,
    \begin{equation*}
        \begin{split}
            Q_{b}(N)&\leqslant e^{tN}\sum_{\{k_{m}\}_{m}} \exp\bigg(-t\sum_{m\geqslant 0} (m+1)|k_{m}|\bigg)\\
            &\leqslant e^{tN} \prod_{m\geqslant 0}\sum_{k\in\mathbb{Z}}e^{-t(m+1)|k|}.
        \end{split}
    \end{equation*}
    Since 
    \begin{equation*}
        \sum_{k\in\mathbb{Z}}e^{-t(m+1)|k|}=1+2\sum_{k\geqslant 1}e^{-t(m+1)k}=\frac{1+e^{-t(m+1)}}{1-e^{-t(m+1)}},
    \end{equation*}
    it follows that
    \begin{equation}\label{badicQ}
        Q_{b}(N)\leqslant e^{tN}\prod_{m\geqslant 1} \frac{1+e^{-tm}}{1-e^{-tm}}.
    \end{equation}
    Taking the logarithm, we obtain
    \begin{equation*}
        \log \prod_{m\geqslant 1} \frac{1+e^{-tm}}{1-e^{-tm}}=\sum_{m\geqslant 1}\log (1+e^{-tm})-\sum_{m\geqslant 1}\log(1-e^{-tm})\eqqcolon (\mathrm{I})+(\mathrm{II}).
    \end{equation*}
    We have
    \begin{equation*}
        (\mathrm{I})\leqslant \sum_{m\geqslant 1}e^{-tm}\leqslant \frac{1}{e^{t}-1}\leqslant \frac{1}{t}.
    \end{equation*}
    On the other hand,
    \begin{equation*}
        (\mathrm{II})\leqslant \sum_{m\geqslant 1} \sum_{s\geqslant 1} \frac{e^{-tms}}{s} \leqslant \sum_{s\geqslant 1}\frac{1}{s} \frac{1}{e^{ts}-1}\leqslant \frac{1}{t}\sum_{s\geqslant 1}\frac{1}{s^{2}}=\frac{\pi^{2}}{6t}.
    \end{equation*}
    Thus, 
    \begin{equation*}
        Q_{b}(N)\leqslant e^{tN+4/t}.
    \end{equation*}
    Taking $t=2N^{-1/2}$ proves \eqref{eq:b-ball}. Together with symmetry, subadditivity, and positivity away from zero, this verifies \hyperref[def:admissible]{Definition~\ref*{def:admissible}} with $\rho=\frac{1}{2}$.

    For \eqref{eq:log-square}, it is trivial when $n=0$. We first show the lower bound of $\ell_{b}(n)$. For any $n\neq 0$, there exists $\{k_{m}\}_{m\geqslant 0}$ such that 
    \begin{equation*}
        n=\sum_{m\geqslant 0}k_{m}b^{m},\quad \ell_{b}(n)=\sum_{m\geqslant 0} (m+1)|k_{m}|.
    \end{equation*}
    Since $b\geqslant 2$, we have $b^{r-1}+b^{s-1}\leqslant b^{r+s-1}$ for any positive integers $r,s$. Thus
    \begin{equation*}
        |n|\leqslant \sum_{m\geqslant 0}|k_{m}| b^{m}\leqslant b^{\sum_{m\geqslant 0}(m+1)|k_{m}|-1}=b^{\ell_{b}(n)-1},
    \end{equation*}
    which implies that for $n\neq 0$,
    \begin{equation}\label{elllowerbound}
        \ell_{b}(n)\geqslant 1+\frac{\log |n|}{\log b}\geqslant  \frac{\log (|n|+1)}{\log b}=c_{b}\big(\log (|n|+1)\big).
    \end{equation}

    Finally, we show the upper bound of $\ell_{b}(n)$. We assume $n\neq 0$ and write $|n|$ in base $b^{m}$,
    \begin{equation*}
        |n|=\sum_{m=0}^{M} k_{m}b^{m},\quad k_{m}\in \{0,1,\dots,b-1\},\quad k_{M}\geqslant 1.
    \end{equation*}
    Thus
    \begin{equation*}
        \ell_{b}(n)\leqslant \sum_{m=0}^{M}(m+1)k_{m}\leqslant b (M+1)^{2}.
    \end{equation*}
    By $b^{M}\leqslant |n|$, one has
    \begin{equation}\label{ellupperbound}
        \ell_{b}(n)\leqslant b\bigg(\frac{\log |n|}{\log b}+1\bigg)^{2}\leqslant 10 b (\log (|n|+1))^{2}= C_{b}\big(\log (|n|+1)\big)^{2}.
    \end{equation}
    Combining \eqref{elllowerbound} with \eqref{ellupperbound} shows \eqref{eq:log-square}.
\end{proof}

As an immediate consequence of \eqref{eq:b-ball} and
\hyperref[prop:full-measure]{Proposition~\ref*{prop:full-measure}}, the arithmetic class
\begin{equation*}
    \begin{split}
        \mathrm{DC}_{b}(\gamma,\tau)
        &\coloneqq
        \Big\{
        \alpha\in\mathbb{T}:
        \|n\alpha\|_{\mathbb{T}}
        \geqslant
        \gamma e^{-\tau\sqrt{\ell_{b}(n)}}
        \text{ for every }n\in\mathbb{Z}\setminus\{0\}
        \Big\},\\
        \mathrm{DC}_{b}(\tau)
        &\coloneqq
        \bigcup_{\gamma>0}\mathrm{DC}_{b}(\gamma,\tau)
    \end{split}
\end{equation*}
has full Lebesgue measure for every $\tau>4$. This is precisely the arithmetic class used for the Weierstrass model in the Introduction.

\subsection{Classical regularity and the Weierstrass example}

The Fourier algebra always contains ordinary analytic functions, after a possible loss of width. Gevrey functions are included whenever the chosen length grows more slowly than the corresponding Gevrey Fourier scale. The $b$-adic length then admits the Weierstrass function because its Fourier support is concentrated on  $\{\pm b^m: m\in\mathbb{N}\}$.

\begin{proposition}[Classical regularity and the Weierstrass example]
\label{prop:W}
    Let $\ell$ be an admissible Fourier length and let $\sigma>0$.
    \begin{enumerate}
        \item Suppose that $f$ is a periodic analytic function satisfying
        \begin{equation*}
            |\widehat{f}(n)|\leqslant C_{f} e^{-h|n|}, \qquad n\in\mathbb{Z},
        \end{equation*}
        for some $C_{f},h>0$. If $\sigma\ell(1)<h$, then $f\in\mathcal{A}_{\sigma}^{\ell}$.

        \item For $s\geqslant 1$, suppose that $f$ is a periodic Gevrey function of order $s$ satisfying
        \begin{equation*}
            |\widehat{f}(n)|\leqslant C_{f} e^{-h|n|^{1/s}}, \qquad n\in\mathbb{Z}.
        \end{equation*}
        If $\ell(n)\leqslant C_{\ell} |n|^{1/s}$ and $\sigma C_{\ell}<h$, then $f\in\mathcal{A}_{\sigma}^{\ell}$. In particular, if
        \begin{equation*}
            \ell(n)=o\big(|n|^{1/s}\big) \qquad \text{as } |n|\to\infty,
        \end{equation*}
        then every periodic Gevrey function of order $s$ belongs to $\mathcal{A}_{\sigma}^{\ell}$ for all $\sigma>0$. In particular, since $\ell_{b}(n)=O((\log(1+|n|))^{2})$, every periodic Gevrey function of finite order belongs to $\mathcal{A}_{\sigma}^{(b)}=\mathcal{A}_{\sigma}^{\ell_b}$ for all $\sigma>0$.

        \item Let $0<a<1$. If $ae^{\sigma}<1$, then
        \begin{equation*}
            W_{a,b}(x)=\sum_{m\geqslant 0}a^{m}\cos(2\pi b^{m} x)
        \end{equation*}
        belongs to $\mathcal{A}_{\sigma}^{(b)}$, and
        \begin{equation}\label{eq:Wnorm}
            \|W_{a,b}\|_{\sigma} \leqslant \frac{e^{\sigma}}{1-ae^{\sigma}}.
        \end{equation}
    \end{enumerate}
\end{proposition}

\begin{proof}
    By the symmetry and subadditivity of $\ell$, one has
    \begin{equation*}
        \ell(n)\leqslant |n|\ell(1), \qquad n\in\mathbb{Z}.
    \end{equation*}
    Therefore,
    \begin{equation*}
        \|f\|_{\sigma} \leqslant C_{f}\sum_{n\in\mathbb{Z}} e^{-(h-\sigma\ell(1))|n|}<\infty
    \end{equation*}
    whenever $\sigma\ell(1)<h$. This proves $f\in\mathcal{A}_{\sigma}^{\ell}$.

    Under the hypotheses of the second assertion,
    \begin{equation*}
        \|f\|_{\sigma} \leqslant C_{f}\sum_{n\in\mathbb{Z}} e^{-(h-\sigma C_{\ell})|n|^{1/s}}<\infty.
    \end{equation*}
    If $\ell(n)=o\big(|n|^{1/s}\big)$ as $|n|\to\infty$, then for any fixed $\sigma>0$, one has $\sigma\ell(n)\leqslant \frac{h}{2}|n|^{1/s}$ for all sufficiently large $|n|$, yielding the same conclusion. In particular, \eqref{eq:log-square} implies that every periodic Gevrey function of finite order belongs to $\mathcal{A}_{\sigma}^{(b)}$ for all $\sigma>0$.

    Finally, the Fourier coefficients of $W_{a,b}$ satisfy
    \begin{equation*}
        \widehat{W}_{a,b}(b^{m}) = \widehat{W}_{a,b}(-b^{m}) = \frac{a^{m}}{2}, \qquad m\geqslant 0.
    \end{equation*}
    Since $\ell_{b}(\pm b^{m})\leqslant m+1$, one obtains
    \begin{equation*}
        \|W_{a,b}\|_{\sigma}\leqslant \sum_{m\geqslant 0}a^{m} e^{\sigma\ell_{b}(b^{m})}\leqslant e^{\sigma}\sum_{m\geqslant 0}(ae^{\sigma})^{m}=\frac{e^{\sigma}}{1-ae^{\sigma}},
    \end{equation*}
    which proves $W_{a,b}\in\mathcal{A}_{\sigma}^{(b)}$ and \eqref{eq:Wnorm}.
\end{proof}

\subsection{General sparse Fourier supports}

The preceding construction is not tied to the geometric progression
$\{\pm b^m:m\in\mathbb{N}\}$. We now allow arbitrary Fourier positions and encode the decay of the coefficients by assigning a weight to the ordered generators.

Let $q = (q_{m})_{m \geqslant 0}$ be an arbitrary fixed sequence consisting of distinct nonzero integers with $q_{0} = 1$. Fix $p>0$ and define
\begin{equation}\label{eq:sparse-length}
    \ell_{q,p}(n)=\inf\bigg\{\sum_{m\geqslant 0}(m+1)^{p}|k_{m}|: n=\sum_{m\geqslant 0}k_{m}q_{m},\quad k_{m}\in\mathbb{Z} \bigg\}.
\end{equation}
The admissible set is nonempty because $q_0=1$. Moreover, the infimum is attained: below any fixed cost bound only finitely many indices can occur, and each corresponding integer coefficient has finitely many possible values. When a $2\mathbb{T}$ space is used, $\ell_{q,p}$ is understood through the canonical extension in \hyperref[lem:half-lattice-extension]{Lemma~\ref*{lem:half-lattice-extension}}.

\begin{proposition}[A sparse stretched-exponential class]
\label{prop:sparse-class}
    The function $\ell_{q,p}$ is admissible with exponent $\rho=\frac{1}{p+1}$.
    There is a constant $c_{p}>0$ such that
    \begin{equation}\label{eq:sparse-ball}
        Q_{q,p}(N)\coloneqq \#\{n\in\mathbb{Z}: \ell_{q,p}(n)\leqslant N\} \leqslant \exp\big(c_{p}N^{\frac{1}{p+1}}\big)
    \end{equation}
    uniformly in $q$.
    Consequently, if $|c_{m}|\leqslant C \exp(-a m^{p})$, then
    \begin{equation*}
        V(x)\coloneqq \sum_{m\geqslant 0}c_{m}e^{2\pi i q_{m} x}\in \mathcal{A}_{\sigma}^{\ell_{q,p}}\quad \text{for every }0<\sigma<a.
    \end{equation*}
     Moreover,   $\mathrm{DC}_{\ell_{q,p}}(\tau)$ has full Lebesgue measure for every $\tau>c_{p}$.
\end{proposition}

\begin{proof}
    The definition \eqref{eq:sparse-length} immediately gives symmetry, subadditivity, and positivity away from zero. The ball estimate proved below shows that every ball is finite.
    Assume $N\geqslant 1$. As in \eqref{badicQ}, counting the representations gives, for any $t>0$,
    \begin{equation*}
        Q_{q,p}(N) \leqslant e^{tN} \prod_{m\geqslant 1}\frac{1+e^{-tm^{p}}}{1-e^{-tm^{p}}}.
    \end{equation*}
    Define
    \begin{equation*}
        \Phi(u)=\log \bigg(\frac{1+e^{-u}}{1-e^{-u}}\bigg),\qquad u>0.
    \end{equation*}
    The direct calculation shows that for $0<u\leqslant 1$,
    \begin{equation*}
        \Phi(u)\leqslant \log \bigg(\frac{2}{u/2}\bigg)\leqslant \log (4/u),
    \end{equation*}
    and for $u>1$,
    \begin{equation*}
        \Phi(u)\leqslant \log (1+e^{-u})-\log (1-e^{-u})\leqslant e^{-u}+\frac{e^{-u}}{1-e^{-u}}\leqslant 3e^{-u}.
    \end{equation*}
    It follows that
    \begin{equation*}
        I_{p}\coloneqq\int_{0}^{\infty}\Phi(x^{p})\,\mathrm{d}x<\infty.
    \end{equation*}
    Since $x\mapsto \Phi(tx^{p})$ is monotonically decreasing, one has
    \begin{equation*}
        \sum_{m\geqslant 1} \Phi(tm^{p})\leqslant \int_{0}^{\infty} \Phi(tx^{p})\,\mathrm{d}x=I_{p} t^{-\frac{1}{p}}.
    \end{equation*}
    Choose $t=N^{-\frac{p}{p+1}}$. Then the required ball growth estimate follows from
    \begin{equation*}
        \log Q_{q,p}(N)\leqslant tN+I_{p} t^{-\frac{1}{p}} \leqslant (1+I_{p})N^{\frac{1}{p+1}}=c_{p}N^{\frac{1}{p+1}}.
    \end{equation*}

    Since $\ell_{q,p}(q_{m})\leqslant (m+1)^{p}$ and $0<\sigma<a$, we have
    \begin{equation*}
        \|V\|_{\sigma}=\sum_{m\geqslant 0}|c_{m}|e^{\sigma \ell_{q,p}(q_{m})}\leqslant C\sum_{m\geqslant 0}\exp \big(-am^{p}+\sigma(m+1)^{p}\big)<\infty,
    \end{equation*}
    which implies $V\in\mathcal{A}_{\sigma}^{\ell_{q,p}}$.
    
    Since the ball exponent is $\rho=\frac{1}{p+1}$, the full measure property follows from \hyperref[prop:full-measure]{Proposition~\ref*{prop:full-measure}}.
\end{proof}

\begin{corollary}[Continuous sampling functions with no positive H\"older regularity]
\label{cor:continuous-nonholder}
    Fix $p>0$ and $a>0$. Let $q=(q_{m})_{m\geqslant 0}$ be given by
    \begin{equation*}
        q_{0}=1,\qquad q_{m}=2^{m+\lceil m^{2p}\rceil},\quad m\geqslant 1,
    \end{equation*}
    and define
    \begin{equation*}
        V(x)=\sum_{m\geqslant 0}e^{-am^{p}}\cos(2\pi q_{m}x).
    \end{equation*}
    Then
    \begin{equation*}
        V\in C^{0}(\mathbb{T},\mathbb{R})\setminus\bigcup_{r>0}C^{r}(\mathbb{T},\mathbb{R}).
    \end{equation*}
    Moreover, for the adapted Fourier length $\ell_{q,p}$ defined by \eqref{eq:sparse-length},
    \begin{equation*}
        V\in\mathcal{A}_{\sigma}^{\ell_{q,p}}(\mathbb{T},\mathbb{R})
    \end{equation*}
    for every $0<\sigma<a$. Consequently, for every $\tau>c_{p}$ and every $\alpha\in\mathrm{DC}_{\ell_{q,p}}(\tau)$, there exists $\lambda_{0}>0$ such that the conclusions of  \hyperref[thm:intro-schrodinger]{Theorem~\ref*{thm:intro-schrodinger}} hold for $\lambda V$ whenever $|\lambda|<\lambda_{0}$.
\end{corollary}

\begin{proof}
    Since $\sum_{m\geqslant 0}\exp(-am^{p})<\infty$,
    the defining series converges uniformly, and hence $V\in C^{0}(\mathbb{T},\mathbb{R})$. Moreover,
    \begin{equation*}
        \widehat{V}(q_{m})=\widehat{V}(-q_{m})=\frac{1}{2}e^{-am^{p}},\qquad m\geqslant 0.
    \end{equation*}
    Since $\ell_{q,p}(\pm q_{m})\leqslant(m+1)^{p}$, for every $0<\sigma<a$,
    \begin{equation*}
        \|V\|_{\sigma}\leqslant\sum_{m\geqslant 0}\exp\big(-am^{p}+\sigma(m+1)^{p}\big)<\infty.
    \end{equation*}
    Thus $V\in\mathcal{A}_{\sigma}^{\ell_{q,p}}(\mathbb{T},\mathbb{R})$.

    Suppose that $V\in C^{r}(\mathbb{T},\mathbb{R})$ for some $r>0$. Fix $0<r'<\min\{r,1\}$. Then $V\in C^{r'}(\mathbb{T},\mathbb{R})$, and hence
    \begin{equation}\label{Cbounded}
        \sup_{n\in\mathbb{Z}}(1+|n|)^{r'}|\widehat{V}(n)|<\infty.
    \end{equation}
    On the other hand, since $\frac{m+\lceil m^{2p}\rceil}{m^{p}}\to\infty$, one has
    \begin{equation*}
        q_{m}^{r'}|\widehat{V}(q_{m})|=\frac{1}{2}\exp\big(r'\log 2\big(m+\lceil m^{2p}\rceil\big)-am^{p}\big)\to\infty,
    \end{equation*}
    which contradicts \eqref{Cbounded}. Therefore, $V\notin C^{r}(\mathbb{T},\mathbb{R})$ for every $r>0$.

    Finally, \hyperref[prop:sparse-class]{Proposition~\ref*{prop:sparse-class}} implies that $\mathrm{DC}_{\ell_{q,p}}(\tau)$ has full Lebesgue measure for every $\tau>c_{p}$. The last assertion follows by applying  \hyperref[thm:intro-schrodinger]{Theorem~\ref*{thm:intro-schrodinger}} to $\lambda V$ for sufficiently small $|\lambda|$.
\end{proof}

\section{One KAM step}

In this section, we establish the quantitative one-step KAM conjugation. Throughout, let $\ell$ be an admissible Fourier length of exponent $0<\rho<1$, and write
\begin{equation*}
    \mathcal A_\sigma=\mathcal A_\sigma^\ell,
    \qquad
    \mathrm{DC}(\gamma,\tau)=\mathrm{DC}_\ell(\gamma,\tau).
\end{equation*}
We use $\mathcal A_\sigma(\mathbb T,X)$ for the corresponding space of $X$-valued periodic functions. We first provide the two cocycle preliminaries needed in the proof: the resonant decomposition for the linearized conjugacy operator and the basic covariance and continuity properties of the fibered rotation number.

\subsection{Resonant and non-resonant decomposition}

Let $A\in \mathrm{SL}(2,\mathbb{R})$ and $Y\in \mathcal{A}_{\sigma}(\mathbb{T},\mathrm{sl}(2,\mathbb{R}))$ with $\sigma>0$. We define the linear operator $L_{A}$ by
\begin{equation*}
    [L_{A}Y](x)\coloneqq A^{-1}Y(x+\alpha)A-Y(x).
\end{equation*}
We decompose $\mathcal{A}_{\sigma}(\mathbb{T},\mathrm{sl}(2,\mathbb{R}))=\mathcal{A}_{\sigma}^{\mathrm{nre}}(\eta) \oplus \mathcal{A}_{\sigma}^{\mathrm{re}}(\eta)$, where $\mathcal{A}_{\sigma}^{\mathrm{nre}}(\eta)$ is the closed invariant subspace of $\mathcal{A}_{\sigma}(\mathbb{T},\mathrm{sl}(2,\mathbb{R}))$ such that $L_{A}$ restricted to $\mathcal{A}_{\sigma}^{\mathrm{nre}}(\eta)$ is invertible and
\begin{equation*}
    \|L_{A}^{-1}\|\leqslant \frac{1}{\eta} \quad \text{on}\quad \mathcal{A}_{\sigma}^{\mathrm{nre}}(\eta).
\end{equation*}

\begin{lemma}\label{elimination}
    Let $\alpha\in\mathbb{T}$, $A\in\mathrm{SL}(2,\mathbb{R})$ and $F\in \mathcal{A}_{\sigma}(\mathbb{T},\mathrm{sl}(2,\mathbb{R}))$ satisfy
    \begin{equation*}
        \|F\|_{\sigma}\leqslant \varepsilon< (4\|A\|)^{-4}.
    \end{equation*}
    Let $\eta>13\|A\|^{2}\varepsilon^{\frac{1}{2}}$. Then there exist $Y\in \mathcal{A}_{\sigma}^{\mathrm{nre}}(\eta)$ and $F^{\mathrm{re}}\in\mathcal{A}_{\sigma}^{\mathrm{re}}(\mathbb{T},\mathrm{sl}(2,\mathbb{R}))$ such that
    \begin{equation*}
        e^{-Y(x+\alpha)} Ae^{F(x)} e^{Y(x)}=Ae^{F^{\mathrm{re}}(x)},
    \end{equation*}
    with the estimates
    \begin{equation*}
        \|F^{\mathrm{re}}\|_{\sigma}\leqslant 2\varepsilon,\quad \|Y\|_{\sigma}\leqslant \varepsilon^{\frac{1}{2}}.
    \end{equation*}
\end{lemma}

\begin{proof}
    Since we have already shown that $\mathcal{A}_{\sigma}$ is a Banach space in \hyperref[banachalgebra]{Proposition~\ref*{banachalgebra}}, the proof follows from Lemma 3.1 and Remark 3.1 in \cite{MR3936094}.
\end{proof}

\subsection{Fibered rotation number}
Denote by $\varrho(\alpha,A)$  the fibered rotation number of the cocycle $(\alpha,A)$. Let
\begin{equation*}
    M=\frac{1}{1+i}\begin{pmatrix}1&-i\\1&i\end{pmatrix}.
\end{equation*}
Then, for every $\phi\in\mathbb R$,
\begin{equation*}
    M^{-1}
    \begin{pmatrix}
        e^{2\pi i\phi}&0\\
        0&e^{-2\pi i\phi}
    \end{pmatrix}
    M
    =
    \begin{pmatrix}
        \cos 2\pi\phi&-\sin 2\pi\phi\\
        \sin 2\pi\phi&\cos 2\pi\phi
    \end{pmatrix}
    \eqqcolon R_\phi.
\end{equation*}
The fibered rotation number is normalized by
$\varrho(\alpha,R_\phi)=\phi\pmod{\mathbb Z}$. If
$A\in C^0(\mathbb T,\mathrm{SL}(2,\mathbb R))$ is homotopic to
$x\mapsto R_{nx/2}$, then $n\in\mathbb Z$ is called the degree of $A$ and
is denoted by $\deg A$. If
\begin{equation*}
    B(x+\alpha)^{-1}A(x)B(x)=A_*(x),
\end{equation*}
where $B\in C^0(2\mathbb T,\mathrm{SL}(2,\mathbb R))$ has degree $n$, then
\begin{equation*}
    \varrho(\alpha,A_*)
    =\varrho(\alpha,A)-\frac{n\alpha}{2}
    \pmod{\mathbb Z}.
\end{equation*}

\begin{lemma}[{\cite[Lemma 9]{MR2481750}}]\label{amorholder}
    Let $\alpha\in\mathbb{R}\setminus\mathbb{Q}$ and $A\in\mathrm{SL}(2,\mathbb{R})$. There exist $C_{A}$ and $\varepsilon=\varepsilon(A)$ such that for $F\in C^{0}(\mathbb{T},\mathrm{sl}(2,\mathbb{R}))$ satisfying $\|F\|_{C^{0}}<\varepsilon$,
    \begin{equation*}
        \|\varrho(\alpha, Ae^{F})-\varrho(A)\|_{\mathbb{T}}\leqslant C_{A}\|F\|_{C^{0}}^{\frac{1}{2}}.
    \end{equation*}
\end{lemma}

\subsection{The one-step conjugation}

Suppose that $A\in\mathrm{SL}(2,\mathbb{R})$, with eigenvalues $\{\chi e^{\pm 2\pi i\varrho}\}$ for some $\varrho\in i\mathbb{R}\cup\mathbb{R}$ and $\chi\in\{\pm 1\}$. We say that $A$ is non-resonant up to $N$, denoted by $A\in\mathcal{NR}_{\ell}(N,\varepsilon^{\frac{1}{10}})$, if
\begin{equation*}
    \|2\varrho-n\alpha\|_{\mathbb{T}}\geqslant\varepsilon^{\frac{1}{10}}
\end{equation*}
for every $n\in\mathbb{Z}$ satisfying $0<\ell(n)\leqslant N$. Otherwise, we say that $A$ is resonant, denoted by $A\in\mathcal{RS}_{\ell}(N,\varepsilon^{\frac{1}{10}})$, if there exists $n_{*}\in\mathbb{Z}$ satisfying $0<\ell(n_{*})\leqslant N$ such that
\begin{equation*}
    \|2\varrho-n_{*}\alpha\|_{\mathbb{T}}<\varepsilon^{\frac{1}{10}}.
\end{equation*}

\begin{theorem}\label{onestep}
    Let $0<\rho<1$, $\alpha\in\mathrm{DC}(\gamma,\tau)$, and $0<\sigma_{+}<\sigma$. Let $A\in \mathrm{SL}(2,\mathbb{R})$. There exist constants $c=c(\gamma,\|A\|)$ and $C=C(\rho,\tau,\sigma)$ such that, if $F\in\mathcal{A}_{\sigma}(\mathbb{T},\mathrm{sl}(2,\mathbb{R}))$ satisfies 
    \begin{equation*}
        \|F\|_{\sigma}<\varepsilon< c\exp\Big(-C(\sigma-\sigma_{+})^{-\frac{2\rho}{1-\rho}}\Big),
    \end{equation*}
    there exist $A_{+}\in\mathrm{SL}(2,\mathbb{R})$, $B\in \mathcal{A}_{\sigma_{+}}(2\mathbb{T},\mathrm{SL}(2,\mathbb{R}))$, and $F_{+}\in\mathcal{A}_{\sigma_{+}}(\mathbb{T},\mathrm{sl}(2,\mathbb{R}))$ such that
    \begin{equation*}
        B(x+\alpha)^{-1}Ae^{F(x)} B(x)=A_{+}e^{F_{+}(x)}.
    \end{equation*}
    Moreover, $B(x+1)=(-1)^{\deg B} B(x)$. Let $N=\frac{2|\log \varepsilon|}{\sigma-\sigma_{+}}$. Then
    \begin{enumerate}
        \item If $A\in \mathcal{NR}_{\ell}(N,\varepsilon^{\frac{1}{10}})$, then $B(\cdot)=e^{Y(\cdot)}$ with
        \begin{equation*}
            \|Y\|_{\sigma}\leqslant \varepsilon^{\frac{1}{2}}, \quad \|F_{+}\|_{\sigma_{+}}\leqslant 4\varepsilon^{3},\quad \|A-A_{+}\|\leqslant 4\|A\|\varepsilon.
        \end{equation*}

        \item If $A\in\mathcal{RS}_{\ell}(N,\varepsilon^{\frac{1}{10}})$, then there exists a unique $n_{*}\in\mathbb{Z}$ with $0<\ell(n_{*})\leqslant N$ and $\|2\varrho-n_{*}\alpha\|_{\mathbb{T}}<\varepsilon^{\frac{1}{10}}$. Moreover,
        \begin{equation*}
            \operatorname{deg}B=n_{*},\qquad \|B\|_{C^{0}}\leqslant\exp\big(|\log\varepsilon|^{\frac{1+\rho}{2}}\big),
        \end{equation*}
        and there exists $\chi_{+}\in \{\pm 1\}$ such that
        \begin{equation*}
            \|F_{+}\|_{\sigma_{+}}\leqslant\varepsilon^{100},\qquad \|A_{+}-\chi_{+} \mathrm{I}\|\leqslant 10\varepsilon^{\frac{1}{10}}.
        \end{equation*}
    \end{enumerate}
\end{theorem}

\begin{proof}
    We first consider the non-resonant case $A\in \mathcal{NR}_{\ell}(N,\varepsilon^{\frac{1}{10}})$.

    Let $\eta=\varepsilon^{\frac{1}{3}}$ and decompose 
    \begin{equation*}
        \mathcal{A}_{\sigma}(\mathbb{T},\mathrm{sl}(2,\mathbb{R}))=\mathcal{A}_{\sigma}^{\mathrm{nre}}(\eta)\oplus\mathcal{A}_{\sigma}^{\mathrm{re}}(\eta),
    \end{equation*}
    where
    \begin{equation}\label{Bre1}
        \begin{split}
            \mathcal{A}_{\sigma}^{\mathrm{nre}}(\eta)&=\mathcal{A}_{\sigma}(\mathbb{T},\mathrm{sl}(2,\mathbb{R}))\cap\big\{F(x)=\mathcal{T}_{N}F(x)-\widehat{F}(0)\big\},\\
            \mathcal{A}_{\sigma}^{\mathrm{re}}(\eta)&=\mathcal{A}_{\sigma}(\mathbb{T},\mathrm{sl}(2,\mathbb{R}))\cap\big\{F(x)=\mathcal{R}_{N}F(x)+\widehat{F}(0)\big\}.
        \end{split}
    \end{equation}
    Since $\alpha\in\mathrm{DC}(\gamma,\tau)$, for every $n\in\mathbb{Z}$ satisfying $0<\ell(n)\leqslant N$, one has
    \begin{equation}\label{sd1}
        \|n\alpha\|_{\mathbb{T}}\geqslant\gamma e^{-\tau\ell(n)^{\rho}}\geqslant\gamma e^{-\tau N^{\rho}}>\varepsilon^{\frac{1}{10}}.
    \end{equation}
    Since $A\in \mathcal{NR}_{\ell}(N,\varepsilon^{\frac{1}{10}})$, the non-resonance condition also gives
    \begin{equation}\label{sd2}
        \|2\varrho-n\alpha\|_{\mathbb{T}}\geqslant\varepsilon^{\frac{1}{10}}
    \end{equation}
    for every $n\in\mathbb{Z}$ satisfying $0<\ell(n)\leqslant N$. Combining \eqref{sd1} and \eqref{sd2} with \cite[Lemma 1]{MR1167299}, we deduce that the operator $L_{A}^{-1}: \mathcal{A}_{\sigma}^{\mathrm{nre}}(\eta)\to \mathcal{A}_{\sigma}^{\mathrm{nre}}(\eta)$ is bounded, satisfying
    \begin{equation*}
        \|L_{A}^{-1}\|\leqslant \varepsilon^{-\frac{1}{3}}=\frac{1}{\eta}.
    \end{equation*}
    By the smallness assumption, one has $\varepsilon<(4\|A\|)^{-4}$ and $\eta>13\|A\|^{2}\varepsilon^{\frac{1}{2}}$. Therefore, by \hyperref[elimination]{Lemma~\ref*{elimination}}, there exist $Y\in \mathcal{A}_{\sigma}^{\mathrm{nre}}(\eta)$ and $F^{\mathrm{re}}\in \mathcal{A}_{\sigma}^{\mathrm{re}}(\eta)$ such that
    \begin{equation*}
        e^{-Y(x+\alpha)} Ae^{F(x)} e^{Y(x)}=Ae^{F^{\mathrm{re}}(x)},
    \end{equation*}
    with $\|Y\|_{\sigma}\leqslant \varepsilon^{\frac{1}{2}}$ and $\|F^{\mathrm{re}}\|_{\sigma}\leqslant 2\varepsilon$. Let $B(x)=e^{Y(x)}$, $A_{+}=Ae^{\widehat{F}^{\mathrm{re}}(0)}$, and $F_{+}(x)=\log (e^{-\widehat{F}^{\mathrm{re}}(0)} e^{F^{\mathrm{re}}(x)})$. 
    Since $\|\widehat{F}^{\mathrm{re}}(0)\|\leqslant 2\varepsilon$, one has
    \begin{equation*}
        \|A_{+}-A\|\leqslant\|A\|\,\|e^{\widehat{F}^{\mathrm{re}}(0)}-\mathrm{I}\|\leqslant 4\|A\|\varepsilon.
    \end{equation*}
    It is trivial that $B(x+1)=(-1)^{\deg B} B(x)$ because $B$ is one-periodic and $\deg B=0$. Finally,
    by \eqref{Bre1} and \hyperref[banachalgebra]{Proposition~\ref*{banachalgebra}},
    \begin{equation}\label{rest}
        \begin{split}
            \|F^{\mathrm{re}}-\widehat{F}^{\mathrm{re}}(0)\|_{\sigma_{+}}&=\|\mathcal{R}_{N}F^{\mathrm{re}}\|_{\sigma_{+}}\\
            &\leqslant e^{-(\sigma-\sigma_{+})N}\|F^{\mathrm{re}}\|_{\sigma}\\
            &\leqslant 2\varepsilon e^{-2|\log\varepsilon|}=2\varepsilon^{3}.
        \end{split}
    \end{equation}
    The Baker--Campbell--Hausdorff (BCH) formula and \eqref{rest} therefore give
    \begin{equation*}
        \|F_{+}\|_{\sigma_{+}}\leqslant 4\varepsilon^{3}.
    \end{equation*}
    This proves the non-resonant case.

    We now consider the resonant case $A\in \mathcal{RS}_{\ell}(N,\varepsilon^{\frac{1}{10}})$. We only need to consider the case where $A$ is elliptic with eigenvalues $\chi e^{\pm 2\pi i\varrho}$ where $\varrho\in\mathbb{R}$. Fix an $n_{*}$ with $0<\ell(n_{*})\leqslant N$ and $\|2\varrho-n_{*}\alpha\|_{\mathbb{T}}<\varepsilon^{1/10}$. By \eqref{sd1},
    \begin{equation}\label{drho}
        \begin{split}
            \|2\varrho\|_{\mathbb{T}}&\geqslant\|n_{*}\alpha\|_{\mathbb{T}}-\|2\varrho-n_{*}\alpha\|_{\mathbb{T}}\\
            &\geqslant\gamma e^{-\tau\ell(n_{*})^{\rho}}-\varepsilon^{\frac{1}{10}}\\
            &\geqslant\frac{\gamma}{2}e^{-\tau\ell(n_{*})^{\rho}}.
        \end{split}
    \end{equation}
    Let 
    \begin{equation*}
        M=\frac{1}{1+i}\begin{pmatrix}1&-i\\1&i\end{pmatrix}.
    \end{equation*}
    By \cite[Lemma 8.1]{MR2969277}, there exists $P\in \mathrm{SL}(2,\mathbb{R})$ such that
    \begin{equation*}
        P^{-1}AP= \widetilde{A}\coloneqq \chi M^{-1}\exp \begin{pmatrix}2\pi i\varrho&0\\0&-2\pi i\varrho\end{pmatrix}M,
    \end{equation*}
    with
    \begin{equation*}
        \|P\|\leqslant 2\sqrt{\frac{\|A\|}{\|2\varrho\|_{\mathbb{T}}}}\leqslant \frac{1}{2}\exp \big(|\log \varepsilon|^{\frac{1+\rho}{2}}\big),
    \end{equation*}
    where the last inequality uses \eqref{drho}. Moreover, $P^{-1}Ae^{F(x)}P= \widetilde{A}e^{\widetilde{F}(x)}$, where $\widetilde{F}(x)=P^{-1}F(x)P\in \mathcal{A}_{\sigma}$ satisfies
    \begin{equation*}
        \|\widetilde{F}\|_{\sigma}\leqslant \varepsilon \|P\|^{2}\leqslant \varepsilon^{\frac{9}{10}}\eqqcolon \tilde{\varepsilon}.
    \end{equation*}

    \begin{lemma}\label{uniquen}
    Let $\widetilde{N}=\big(\frac{|\log \varepsilon|}{20\tau}\big)^{\frac{1}{\rho}}-N$. Then
    \begin{equation*}
        \|n\alpha\|_{\mathbb{T}} \geqslant \varepsilon^{\frac{1}{10}} \quad \text{for all } 0<\ell(n)\leqslant\widetilde{N}.
    \end{equation*}
    Moreover, there exists a unique $n_{*}\in\mathbb{Z}$ satisfying $0<\ell(n_{*})\leqslant \widetilde{N}$ and $\|2\varrho-n_{*}\alpha\|_{\mathbb{T}}<\varepsilon^{\frac{1}{10}}$. Consequently,
    \begin{equation*}
        \|2\varrho-n\alpha\|_{\mathbb{T}} \geqslant \varepsilon^{\frac{1}{10}} \quad \text{for all } \ell(n)\leqslant\widetilde{N} \text{ with } n\neq n_{*}.
    \end{equation*}
\end{lemma}

\begin{proof}
    For any $n\in\mathbb{Z}$ with $0<\ell(n)\leqslant\widetilde{N}$, $\alpha\in\mathrm{DC}(\gamma,\tau)$ implies
    \begin{equation*}
        \|n\alpha\|_{\mathbb{T}} \geqslant \gamma e^{-\tau\ell(n)^{\rho}} \geqslant \gamma e^{-\tau \widetilde{N}^{\rho}} \geqslant 2\varepsilon^{\frac{1}{10}}.
    \end{equation*}

    Since $A\in \mathcal{RS}_{\ell}(N,\varepsilon^{\frac{1}{10}})$, there exists $n_{*}\in\mathbb{Z}$ with $0<\ell(n_{*})\leqslant N$ such that $\|2\varrho-n_{*}\alpha\|_{\mathbb{T}}<\varepsilon^{\frac{1}{10}}$. Suppose there exists another integer $n_{*}'\neq n_{*}$ such that $\|2\varrho-n_{*}'\alpha\|_{\mathbb{T}}<\varepsilon^{\frac{1}{10}}$. Then $\|(n_{*}-n_{*}')\alpha\|_{\mathbb{T}}<2\varepsilon^{\frac{1}{10}}$. However, since $\alpha\in\mathrm{DC}(\gamma,\tau)$, we also have
    \begin{equation*}
        2\varepsilon^{\frac{1}{10}} > \|(n_{*}-n_{*}')\alpha\|_{\mathbb{T}} \geqslant \gamma e^{-\tau \ell (n_{*}-n_{*}')^{\rho}}.
    \end{equation*}
    By the symmetry and subadditivity of $\ell$, along with the bound $\ell(n_{*})\leqslant N$, we deduce
    \begin{equation*}
        \ell(n_{*}') \geqslant \ell(n_{*}-n_{*}') -\ell(n_{*}) \geqslant \bigg(\frac{1}{\tau}\log \frac{\gamma}{2\varepsilon^{\frac{1}{10}}}\bigg)^{1/\rho}-N > \widetilde{N}.
    \end{equation*}
    This proves the uniqueness of $n_{*}$ within the ball of radius $\widetilde{N}$, which immediately yields the final estimate.
\end{proof}

    Let $\eta=\tilde{\varepsilon}^{\frac{1}{3}}$. For every $F\in\mathcal{A}_{\sigma}(\mathbb{T},\mathrm{sl}(2,\mathbb{R}))$, write
    \begin{equation}\label{expansionF}
        MF(x)M^{-1}=\begin{pmatrix}if_{11}(x)&f_{12}(x)\\\overline{f_{12}(x)}&-if_{11}(x)\end{pmatrix}.
    \end{equation}
    Define
    \begin{equation}\label{PF}
        \mathbb{P}_{n_{*}}F(x)\coloneqq M^{-1}\begin{pmatrix}i\widehat{f}_{11}(0)& \widehat{f}_{12}(n_{*})e^{2\pi i n_{*}x}\\\overline{\widehat{f}_{12}(n_{*})}e^{-2\pi i n_{*}x}&-i\widehat{f}_{11}(0)\end{pmatrix}M.
    \end{equation}
    Decompose
    \begin{equation*}
        \mathcal{A}_{\sigma}(\mathbb{T},\mathrm{sl}(2,\mathbb{R}))=\mathcal{A}_{\sigma}^{\mathrm{nre}}(\eta)\oplus\mathcal{A}_{\sigma}^{\mathrm{re}}(\eta),
    \end{equation*}
    where
    \begin{equation}\label{Bre2}
        \begin{split}
            \mathcal{A}_{\sigma}^{\mathrm{nre}}(\eta)&=\mathcal{A}_{\sigma}(\mathbb{T},\mathrm{sl}(2,\mathbb{R}))\cap\big\{F(x)=\mathcal{T}_{\widetilde{N}}F(x)-\mathbb{P}_{n_{*}}F(x)\big\},\\
            \mathcal{A}_{\sigma}^{\mathrm{re}}(\eta)&=\mathcal{A}_{\sigma}(\mathbb{T},\mathrm{sl}(2,\mathbb{R}))\cap\big\{F(x)=\mathcal{R}_{\widetilde{N}}F(x)+\mathbb{P}_{n_{*}}F(x)\big\}.
        \end{split}
    \end{equation}
    Then $\mathcal{A}_{\sigma}^{\mathrm{nre}}(\eta)$ is a closed invariant subspace. Moreover, it follows from \hyperref[uniquen]{Lemma~\ref*{uniquen}} that the operator $L_{\widetilde{A}}^{-1}: \mathcal{A}_{\sigma}^{\mathrm{nre}}(\eta)\to \mathcal{A}_{\sigma}^{\mathrm{nre}}(\eta)$ is bounded, satisfying
    \begin{equation*}
        \|L_{\widetilde{A}}^{-1}\|\leqslant \varepsilon^{-\frac{3}{10}}= \tilde{\varepsilon}^{-\frac{1}{3}}=\frac{1}{\eta}.
    \end{equation*}
    Hence, by \hyperref[elimination]{Lemma~\ref*{elimination}}, there exist $Y\in\mathcal{A}_{\sigma}^{\mathrm{nre}}(\eta)$ and $\widetilde{F}^{\mathrm{re}}\in\mathcal{A}_{\sigma}^{\mathrm{re}}(\eta)$ such that
    \begin{equation*}
        e^{-Y(x+\alpha)} \widetilde{A}e^{\widetilde{F}(x)} e^{Y(x)}=\widetilde{A}e^{\widetilde{F}^{\mathrm{re}}(x)},
    \end{equation*}
    with $\|Y\|_{\sigma}\leqslant \tilde{\varepsilon}^{\frac{1}{2}}$ and $\|\widetilde{F}^{\mathrm{re}}\|_{\sigma}\leqslant 2\tilde{\varepsilon}$. Moreover, by \eqref{Bre2}, $\widetilde{F}^{\mathrm{re}}$ takes the form
    \begin{equation}\label{Fre}
        \widetilde{F}^{\mathrm{re}}(x)= \mathcal{R}_{\widetilde{N}}\widetilde{F}^{\mathrm{re}}(x)+\mathbb{P}_{n_{*}} \widetilde{F}^{\mathrm{re}}(x).
    \end{equation}

    Define the rotation $Z(\cdot): 2\mathbb{T}\to \mathrm{SO}(2,\mathbb{R})$ as
    \begin{equation*}
        Z(x)=R_{n_{*}x/2}=M^{-1}\exp\begin{pmatrix}{\pi i n_{*}x}&0\\0&{-\pi i n_{*}x}\end{pmatrix}M.
    \end{equation*}
    Choose $m_{*}\in\mathbb{Z}$ such that $|2\varrho-n_{*}\alpha-m_{*}|=\|2\varrho-n_{*}\alpha\|_{\mathbb{T}}<\varepsilon^{\frac{1}{10}}$.
    Then
    \begin{equation*}
        Z(x+\alpha)^{-1}\widetilde{A}Z(x)=(-1)^{m_{*}}\chi e^{\widehat{A}},
    \end{equation*}
    where
    \begin{equation*}
        \widehat{A}=M^{-1}\begin{pmatrix}\pi i(2\varrho-n_{*}\alpha-m_{*})&0\\0&-\pi i(2\varrho-n_{*}\alpha-m_{*})\end{pmatrix}M.
    \end{equation*}
    In particular, we have $\|\widehat{A}\|\leqslant\pi\varepsilon^{\frac{1}{10}}$.

    Let
    \begin{equation*}
        \begin{split}
            G_{1}(x)&=Z(x)^{-1}\mathcal{R}_{\widetilde{N}}\widetilde{F}^{\mathrm{re}}(x)Z(x),\\
            G_{2}(x)&=Z(x)^{-1}\mathbb{P}_{n_{*}}\widetilde{F}^{\mathrm{re}}(x)Z(x).
        \end{split}
    \end{equation*}
    By \eqref{Fre},
    \begin{equation*}
        Z(x)^{-1}\widetilde{F}^{\mathrm{re}}(x)Z(x)=G_{1}(x)+G_{2}(x).
    \end{equation*}
    Note that $Z(\cdot)$ is defined on $2\mathbb{T}$, but $G_{1}(\cdot)$ and $G_{2}(\cdot)$ are still defined on $\mathbb{T}$ since $Z(x+1)=(-1)^{n_{*}} Z(x)$. 
    More precisely, if $F$ satisfies \eqref{expansionF}, then
    \begin{equation*}
        MZ(x)^{-1}F(x)Z(x)M^{-1}=\begin{pmatrix}if_{11}(x)&f_{12}(x)e^{-2\pi i n_{*}x}\\\overline{f_{12}(x)}e^{2\pi i n_{*}x}&-if_{11}(x)\end{pmatrix}.
    \end{equation*}

    We first estimate $G_{1}$. By the subadditivity of $\ell$, if $\ell(n)>\widetilde{N}$, then
    \begin{equation*}
        e^{\sigma_{+}\ell(n\pm n_{*})}\leqslant e^{\sigma_{+}\ell(n_{*})}e^{\sigma_{+}\ell(n)}\leqslant e^{\sigma_{+}\ell(n_{*})-(\sigma-\sigma_{+})\widetilde{N}}e^{\sigma\ell(n)}.
    \end{equation*}
    Therefore,
    \begin{equation*}
        \begin{split}
            \|G_{1}\|_{\sigma_{+}}&\leqslant e^{\sigma_{+}\ell(n_{*})-(\sigma-\sigma_{+})\widetilde{N}}\|\mathcal{R}_{\widetilde{N}}\widetilde{F}^{\mathrm{re}}\|_{\sigma}\\
            &\leqslant 2\widetilde{\varepsilon}e^{\sigma_{+}N-(\sigma-\sigma_{+})(\big(\frac{|\log \varepsilon|}{20\tau}\big)^{\frac{1}{\rho}}-N)}\\
            &=2\widetilde{\varepsilon}e^{\sigma N-(\sigma-\sigma_{+})\big(\frac{|\log \varepsilon|}{20\tau}\big)^{\frac{1}{\rho}}}\\
            &<\frac{1}{2}\varepsilon^{100}.
        \end{split}
    \end{equation*}

    We next estimate $G_{2}$. By \eqref{PF}, one may write 
    \begin{equation*}
        \mathbb{P}_{n_{*}} \widetilde{F}^{\mathrm{re}}(x)=M^{-1}\begin{pmatrix}it&\nu e^{2\pi i n_{*}x}\\\overline{\nu}e^{-2\pi i n_{*}x}&-it\end{pmatrix}M,
    \end{equation*}
    for some $t\in\mathbb{R}$ and $\nu\in\mathbb{C}$. Hence
    \begin{equation*}
        G_{2}=M^{-1}\begin{pmatrix}i t&\nu\\\overline{\nu}&-it\end{pmatrix}M\in\mathrm{sl}(2,\mathbb{R}).
    \end{equation*}

    Define $A_{+}=(-1)^{m_{*}}\chi e^{\widehat{A}}e^{G_{2}}$ and $F_{+}(x)=\log\big(e^{-G_{2}}e^{G_{2}+G_{1}(x)}\big)$. Then
    \begin{equation*}
        \begin{split}
            Z(x+\alpha)^{-1}\widetilde{A}e^{\widetilde{F}^{\mathrm{re}}(x)}Z(x)&=(-1)^{m_{*}} \chi e^{\widehat{A}}e^{G_{2}+G_{1}(x)}\\
            &=A_{+}e^{-G_{2}}e^{G_{2}+G_{1}(x)}\\
            &=A_{+}e^{F_{+}(x)}.
        \end{split}
    \end{equation*}
    Let $\chi_{+}=(-1)^{m_{*}}\chi\in \{\pm 1\}$. Since $\|\widehat{A}\|\leqslant\pi\varepsilon^{\frac{1}{10}}$ and $\|G_{2}\|\leqslant 2\varepsilon^{\frac{9}{10}}$, by BCH formula, one has
    \begin{equation*}
        \|A_{+}-\chi_{+}\mathrm{I}\|\leqslant 2(\|\widehat{A}\|+ \|{G_{2}}\|)\leqslant 10\varepsilon^{\frac{1}{10}}.
    \end{equation*}        
    Apply the BCH formula again,
    \begin{equation*}
        F_{+}(x)=G_{1}(x)+\frac{1}{2}[-G_{2},G_{1}(x)]+\cdots,
    \end{equation*}
    and therefore
    \begin{equation*}
        \|F_{+}\|_{\sigma_{+}}\leqslant 2\|G_{1}\|_{\sigma_{+}}\leqslant\varepsilon^{100}.
    \end{equation*}

    Finally, let $B(x)=Pe^{Y(x)}Z(x)\in\mathcal{A}_{\sigma_{+}}(2\mathbb{T},\mathrm{SL}(2,\mathbb{R}))$. We have
    \begin{equation*}
        B(x+\alpha)^{-1}Ae^{F(x)}B(x)=A_{+}e^{F_{+}(x)}.
    \end{equation*}
    Since $P$ and $e^{Y}$ are homotopic to the identity, we have $\operatorname{deg}B=\operatorname{deg}Z=n_{*}$. The direct calculation gives $B(x+1)=(-1)^{n_{*}} B(x)$.
    Moreover, since $Z\in \mathcal{A}_{\sigma_{+}}(2\mathbb{T},\mathrm{SO}(2,\mathbb{R}))$ with $\|Z\|_{C^{0}}=1$, we have 
    \begin{equation*}
        \begin{split}
            \|B\|_{C^{0}}&\leqslant 2\|P\| \leqslant\exp\big(|\log \varepsilon|^{\frac{1+\rho}{2}}\big).
        \end{split}
    \end{equation*}
    This completes the proof.
\end{proof}

\section{The KAM iteration}

Let $0<\varepsilon<1$ and $\varepsilon_{-1}=1$. For $j\geqslant 0$, define 
\begin{equation*}
    \varepsilon_{0}=\varepsilon,\quad \varepsilon_{j+1}=4\varepsilon_{j}^{3},\quad \sigma_{j}=\sigma'+\frac{\sigma-\sigma'}{j+1},\quad N_{j}=\frac{2|\log\varepsilon_{j}|}{\sigma_{j}-\sigma_{j+1}}.
\end{equation*}

\subsection{Reducibility and almost reducibility}

\begin{theorem}\label{KAM}
    Let $0<\rho<1$, $\alpha\in \mathrm{DC}(\gamma,\tau)$,  $0<\sigma'<\sigma$ and $A\in\mathrm{SL}(2,\mathbb{R})$. There exists $\varepsilon_{0}=\varepsilon_{0}(\rho, \|A\|,\sigma,\sigma',\gamma,\tau)$ such that if $F\in\mathcal{A}_{\sigma}(\mathbb{T},\mathrm{sl}(2,\mathbb{R}))$ satisfies $\|F\|_{\sigma}<\varepsilon_{0}$, then the following holds.
    
    \begin{enumerate}
        \item \label{item:AR} For every $j\geqslant 0$, there exist $B_{j}\in\mathcal{A}_{\sigma_{j}}(2\mathbb{T},\mathrm{SL}(2,\mathbb{R}))$, $A_{j}\in\mathrm{SL}(2,\mathbb{R})$, and $F_{j}\in\mathcal{A}_{\sigma_{j}}(\mathbb{T},\mathrm{sl}(2,\mathbb{R}))$ such that
        \begin{equation*}
            B_{j}(x+\alpha)^{-1}Ae^{F(x)}B_{j}(x)=A_{j}e^{F_{j}(x)},\quad \|F_{j}\|_{\sigma_{j}}\leqslant \varepsilon_{j}.
        \end{equation*}
        More precisely, if $A_{j}\in \mathcal{NR}_{\ell}(N_{j},\varepsilon_{j}^{\frac{1}{10}})$, then $B_{j+1}=B_{j}\widetilde{B}_{j}$ with
        \begin{equation*}
            \|\widetilde{B}_{j}-\mathrm{I}\|_{\sigma_{j+1}}\leqslant 2\varepsilon_{j}^{\frac{1}{2}},\quad \|A_{j+1}-A_{j}\|\leqslant 4\|A_{j}\|\varepsilon_{j}.
        \end{equation*}
        If $A_{j}\in \mathcal{RS}_{\ell}(N_{j},\varepsilon_{j}^{\frac{1}{10}})$, then $B_{j+1}=B_{j}\widetilde{B}_{j}$ with $\deg \widetilde{B}_{j}=n_{j}$ for some $0<\ell(n_{j})\leqslant N_{j}$, and there is $\chi_{j+1}\in \{\pm 1\}$ such that
        \begin{equation*}
            \|\widetilde{B}_{j}\|_{C^{0}}\leqslant\exp\big(|\log\varepsilon_{j}|^{\frac{1+\rho}{2}}\big),\quad \|A_{j+1}-\chi_{j+1}\mathrm{I}\|\leqslant 10\varepsilon_{j}^{\frac{1}{10}}.
        \end{equation*}
        Consequently, 
        \begin{equation*}
            \|B_{j}\|_{C^{0}}\leqslant \exp\big(C_{\rho}|\log \varepsilon_{j-1}|^{\frac{1+\rho}{2}}\big),\quad \|A_{j}\|\leqslant 2\|A\|.
        \end{equation*}

        \item \label{item:deg} For every $j\geqslant 1$, 
        \begin{equation*}
            B_{j}(x+1)=(-1)^{\deg B_{j}} B_{j}(x),\quad \ell(\deg B_{j}) \leqslant 2N_{j-1}.
        \end{equation*}

        \item \label{item:UT} For every $j\geqslant 0$, there exist $U_{j}\in \mathrm{SU}(2)$,  $\varrho_{j}\in i\mathbb{R}\cup\mathbb{R}$, and $\zeta_{j}\in\mathbb{C}$ such that
        \begin{equation*}
            U_{j}^{-1} A_{j}U_{j}=\chi_{j}\begin{pmatrix}e^{2\pi i\varrho_{j}}&\zeta_{j}\\0&e^{-2\pi i\varrho_{j}}\end{pmatrix} \quad \text{and} \quad \|B_{j}\|_{C^{0}}^{2} \, |\zeta_{j}|\leqslant 8\|A\|.
        \end{equation*}
        
        \item \label{item:RE} If there are only finitely many resonant steps, namely,
        \begin{equation*}
            \#\{j\geqslant 0: A_{j}\in \mathcal{RS}_{\ell}(N_{j}, \varepsilon_{j}^{\frac{1}{10}})\}<\infty,
        \end{equation*}
        then there exist $B\in\mathcal{A}_{\sigma'}(2\mathbb{T},\mathrm{SL}(2,\mathbb{R}))$ and $\bar{A}\in\mathrm{SL}(2,\mathbb{R})$ such that
        \begin{equation*}
            B(x+\alpha)^{-1}Ae^{F(x)}B(x)=\bar{A}.
        \end{equation*}
    \end{enumerate}
\end{theorem}

\begin{proof}
    We first prove the iterative conjugacy and degree estimates in \hyperref[KAM]{Theorem~\ref*{KAM}} by induction. Let $C_{\rho}>1$ be a large constant such that
    \begin{equation}\label{Crho}
        C_{\rho }\Big(\frac{1}{2}\Big)^{\frac{1+\rho}{2}}\leqslant C_{\rho}-1.
    \end{equation}

    For $j=0$, let
    \begin{equation*}
        B_{0}=\mathrm{I},\qquad A_{0}=A,\qquad F_{0}=F.
    \end{equation*}
    Thus all the estimates hold naturally.

    Assume that $B_{i}$, $A_{i}$, and $F_{i}$ have been constructed for every $0\leqslant i\leqslant j$. Applying \hyperref[onestep]{Theorem~\ref*{onestep}} to $A_{j}e^{F_{j}}$, we obtain $\widetilde{B}_{j}\in\mathcal{A}_{\sigma_{j+1}}(2\mathbb{T},\mathrm{SL}(2,\mathbb{R}))$, $A_{j+1}\in\mathrm{SL}(2,\mathbb{R})$, and $F_{j+1}\in\mathcal{A}_{\sigma_{j+1}}(\mathbb{T},\mathrm{sl}(2,\mathbb{R}))$ such that
    \begin{equation*}
        \widetilde{B}_{j}(x+\alpha)^{-1} A_{j}e^{F_{j}(x)}\widetilde{B}_{j}(x)=A_{j+1}e^{F_{j+1}(x)}.
    \end{equation*}
    Let $B_{j+1}=B_{j}\widetilde{B}_{j}$. If $A_{j}\in \mathcal{NR}_{\ell}(N_{j},\varepsilon_{j}^{\frac{1}{10}})$, then
    \begin{equation}\label{NRiteration}
        \|\widetilde{B}_{j}-\mathrm{I}\|_{\sigma_{j+1}} \leqslant 2\varepsilon_{j}^{\frac{1}{2}},\quad \|F_{j+1}\|_{\sigma_{j+1}}\leqslant 4\varepsilon_{j}^{3}=\varepsilon_{j+1},
    \end{equation}
    together with
    \begin{equation}\label{Aincrement}
        \|A_{j+1}-A_{j}\|\leqslant 4\|A_{j}\|\varepsilon_{j}.
    \end{equation}
    Since $\widetilde{B}_{j}$ is close to identity,  $\deg \widetilde{B}_{j}=0$. Thus $\deg B_{j+1}=\deg B_{j}$ and
    \begin{equation*}
        \ell(\deg B_{j+1})=\ell(\deg B_{j})\leqslant 2N_{j-1}\leqslant 2N_{j}.
    \end{equation*}
    By the inductive assumption $B_{j}(x+1)=(-1)^{\deg B_{j}}B_{j}(x)$, we have
    \begin{equation*}
        B_{j+1}(x+1)=B_{j}(x+1)\widetilde{B}_{j}(x+1)=(-1)^{\deg B_{j+1}} B_{j+1}(x).
    \end{equation*}
    If $A_{j}\in \mathcal{RS}_{\ell}(N_{j},\varepsilon_{j}^{\frac{1}{10}})$, then 
    \begin{equation*}
        \|\widetilde{B}_{j}\|_{C^{0}}\leqslant\exp\big(|\log\varepsilon_{j}|^{\frac{1+\rho}{2}}\big),\quad \|F_{j+1}\|_{\sigma_{j+1}}\leqslant \varepsilon_{j}^{100}\leqslant \varepsilon_{j+1},
    \end{equation*}
    and there exists $\chi_{j+1}\in\{\pm 1\}$ such that
    \begin{equation}\label{Aresonant}
        \|A_{j+1}-\chi_{j+1}\mathrm{I}\|\leqslant 10\varepsilon_{j}^{\frac{1}{10}}.
    \end{equation}
    Since $\deg \widetilde{B}_{j}=n_{j}$, one has $\deg B_{j+1}=\deg B_{j}+n_{j}$ and thus
    \begin{equation*}
        B_{j+1}(x+1)=(-1)^{\deg B_{j}}(-1)^{n_{j}} B_{j}(x)\widetilde{B}_{j}(x)=(-1)^{\deg B_{j+1}} B_{j+1}(x).
    \end{equation*}
    By the subadditivity of $\ell$, 
    \begin{equation*}
        \ell(\deg B_{j+1})\leqslant \ell(\deg B_{j})+\ell(n_{j}) \leqslant 2N_{j-1}+N_{j}\leqslant 2N_{j}.
    \end{equation*}
    This proves the degree estimate.
    
    Combining the previous conjugacy with the one-step conjugacy gives
    \begin{equation*}
        B_{j+1}(x+\alpha)^{-1}Ae^{F(x)}B_{j+1}(x)=A_{j+1}e^{F_{j+1}(x)}.
    \end{equation*}
    Then by \eqref{Crho} and $|\log \varepsilon_{j-1}|\leqslant \frac{1}{2}|\log \varepsilon_{j}|$,
    \begin{equation*}
        \|B_{j+1}\|_{C^{0}} \leqslant \exp\big(C_{\rho}|\log \varepsilon_{j-1}|^{\frac{1+\rho}{2}}+|\log \varepsilon_{j}|^{\frac{1+\rho}{2}}\big)\leqslant \exp\big(C_{\rho}|\log\varepsilon_{j}|^{\frac{1+\rho}{2}}\big).
    \end{equation*}

    We next estimate $A_{j+1}$. Suppose first that no resonant step occurs among the first $j+1$ steps. By \eqref{Aincrement},
    \begin{equation*}
        \|A_{j+1}\|\leqslant\|A\|\prod_{i=0}^{j}(1+4\varepsilon_{i})\leqslant 2\|A\|.
    \end{equation*}
    Otherwise, let $k\leqslant j$ be the last resonant step. By \eqref{Aresonant} and \eqref{Aincrement},
    \begin{equation*}
        \|A_{j+1}\|\leqslant\|A_{k+1}\|\prod_{i=k+1}^{j}(1+4\varepsilon_{i})\leqslant\big(1+10\varepsilon_{k}^{\frac{1}{10}}\big)\prod_{i=k+1}^{j}(1+4\varepsilon_{i})\leqslant 2\|A\|.
    \end{equation*}
    This proves the iterative almost reducibility statement.
  
    We next study the structure of $A_{j}$. Since $A_{j}\in\mathrm{SL}(2,\mathbb{R})$, there exist $U_{j}\in\mathrm{SU}(2)$, $\chi_{j}\in\{\pm 1\}$, $\varrho_{j}\in i\mathbb{R}\cup\mathbb{R}$, and $\zeta_{j}\in\mathbb{C}$ such that
    \begin{equation*}
        U_{j}^{-1}A_{j}U_{j}=\chi_{j}\begin{pmatrix}e^{2\pi i\varrho_{j}}&\zeta_{j}\\0&e^{-2\pi i\varrho_{j}}\end{pmatrix}.
    \end{equation*}
    For $j=0$, one has
    \begin{equation*}
        |\zeta_{0}|\leqslant\|A\|,
    \end{equation*}
    and hence the desired estimate holds. We now estimate $\zeta_{j+1}$.

    Suppose first that $A_{j}\in \mathcal{RS}_{\ell}(N_{j},\varepsilon_{j}^{\frac{1}{10}})$. By \eqref{Aresonant},
    \begin{equation*}
        |\zeta_{j+1}|\leqslant\|A_{j+1}-\chi_{j+1}\mathrm{I}\|\leqslant 10\varepsilon_{j}^{\frac{1}{10}}.
    \end{equation*}
    Therefore,
    \begin{equation*}
        \|B_{j+1}\|_{C^{0}}^{2}|\zeta_{j+1}|\leqslant 10\exp\big(2C_{\rho}|\log\varepsilon_{j}|^{\frac{1+\rho}{2}}\big)\varepsilon_{j}^{\frac{1}{10}}\leqslant 1\leqslant 8\|A\|.
    \end{equation*}

    Suppose next that $A_{j}\in \mathcal{NR}_{\ell}(N_{j},\varepsilon_{j}^{\frac{1}{10}})$ and no resonant step occurs before it. By \eqref{NRiteration},
    \begin{equation*}
        \|B_{j+1}\|_{C^{0}}\leqslant\prod_{i=0}^{j}\big(1+2\varepsilon_{i}^{\frac{1}{2}}\big)\leqslant 2.
    \end{equation*}
    Since $|\zeta_{j+1}|\leqslant\|A_{j+1}\|\leqslant 2\|A\|$, it follows that
    \begin{equation*}
        \|B_{j+1}\|_{C^{0}}^{2}|\zeta_{j+1}|\leqslant 8\|A\|.
    \end{equation*}

    Finally, suppose that $A_{j}\in \mathcal{NR}_{\ell}(N_{j},\varepsilon_{j}^{\frac{1}{10}})$ and let $k<j$ be the last resonant step. By \eqref{NRiteration},
    \begin{equation*}
        \|B_{j+1}\|_{C^{0}}\leqslant\|B_{k+1}\|_{C^{0}}\prod_{i=k+1}^{j}\big(1+2\varepsilon_{i}^{\frac{1}{2}}\big)\leqslant 2\|B_{k+1}\|_{C^{0}}.
    \end{equation*}
    Moreover, by \eqref{Aincrement} and \eqref{Aresonant},
    \begin{equation*}
        \begin{split}
            |\zeta_{j+1}|&\leqslant\|A_{j+1}-\chi_{k+1}\mathrm{I}\|\\
            &\leqslant\|A_{j+1}-A_{k+1}\|+\|A_{k+1}-\chi_{k+1}\mathrm{I}\|\\
            &\leqslant 8\|A\|\sum_{i=k+1}^{j}\varepsilon_{i}+10\varepsilon_{k}^{\frac{1}{10}}\\
            &\leqslant 20\varepsilon_{k}^{\frac{1}{10}}.
        \end{split}
    \end{equation*}
    Hence,
    \begin{equation*}
        \|B_{j+1}\|_{C^{0}}^{2}|\zeta_{j+1}|\leqslant 80\exp\big(2C_{\rho}|\log\varepsilon_{k}|^{\frac{1+\rho}{2}}\big)\varepsilon_{k}^{\frac{1}{10}}\leqslant 1\leqslant 8\|A\|.
    \end{equation*}
    This proves the triangularization estimate.
    
    It remains to prove the final reducibility statement. Assume that there are only finitely many resonant steps. Let $j'\geqslant 0$ be such that $A_{j}\in\mathcal{NR}_{\ell}(N_{j},\varepsilon_{j}^{\frac{1}{10}})$ for every $j\geqslant j'$. By \eqref{NRiteration},
    \begin{equation*}
        \sum_{j\geqslant j'}\|\widetilde B_j-\mathrm I\|_{\sigma'}<\infty,
        \qquad
        \sum_{j\geqslant j'}\|A_{j+1}-A_j\|<\infty.
    \end{equation*}
    Let $B=\lim_{j\to \infty} B_{j} \in \mathcal{A}_{\sigma'}(2\mathbb{T},\mathrm{SL}(2,\mathbb{R}))$ and $\bar{A}=\lim_{j\to \infty} A_{j}\in \mathrm{SL}(2,\mathbb{R})$. 
    This proves \hyperref[KAM]{Theorem~\ref*{KAM}}.
\end{proof}

\subsection{Parabolic and elliptic cases}

\begin{corollary}\label{PE}
    Under the assumptions of \hyperref[KAM]{Theorem~\ref*{KAM}}, let $\mathcal{UH}$ denote the set of uniformly hyperbolic cocycles. Then the following holds:
    \begin{enumerate}
        \item \label{item:rational} If $2\varrho(\alpha, Ae^{F})=k\alpha\bmod\mathbb{Z}$ for some $k\in\mathbb{Z}$ and $(\alpha,Ae^{F})\notin \mathcal{UH}$, there exist $B\in\mathcal{A}_{\sigma'}(2\mathbb{T},\mathrm{SL}(2,\mathbb{R}))$ and $\zeta\in\mathbb{R}$ such that
        \begin{equation*}
            B(x+\alpha)^{-1}Ae^{F(x)}B(x)=\pm \begin{pmatrix}
                1&\zeta\\
                0&1
            \end{pmatrix}.
        \end{equation*}

        \item \label{item:diophantine} If $\|2\varrho(\alpha,Ae^{F})-n\alpha\|_{\mathbb{T}}\geqslant \kappa e^{-\tau\ell(n)^{\rho}}$ for every $n\in\mathbb{Z}$, there exist $B\in\mathcal{A}_{\sigma'}(2\mathbb{T},\mathrm{SL}(2,\mathbb{R}))$ and $\phi\in\mathbb{T}$ such that 
        \begin{equation*}
            B(x+\alpha)^{-1}Ae^{F(x)}B(x)=R_{\phi}.
        \end{equation*}
    \end{enumerate}
\end{corollary}

\begin{proof}
    We claim that in both cases, there are only finitely many resonant steps. Otherwise, suppose there are infinitely many resonances. Then there exist arbitrarily large $j$ such that $A_{j}\in \mathcal{RS}_{\ell}(N_{j},\varepsilon_{j}^{\frac{1}{10}})$. Consequently, there exists $n_{j}\in\mathbb{Z}$ satisfying
    \begin{equation*}
        0<\ell(n_{j})\leqslant N_{j},\qquad \|2\varrho(A_{j})-n_{j}\alpha\|_{\mathbb{T}}\leqslant\varepsilon_{j}^{\frac{1}{10}}.
    \end{equation*}
    Since \hyperref[amorholder]{Lemma~\ref*{amorholder}} gives $\|\varrho(\alpha,A_{j}e^{F_{j}})-\varrho(A_{j})\|\leqslant C\|F_{j}\|_{C^{0}}^{\frac{1}{2}}\leqslant C\varepsilon_{j}^{\frac{1}{2}}$, one has
    \begin{equation*}
        \|2\varrho(\alpha, A_{j}e^{F_{j}})-n_{j}\alpha\|_{\mathbb{T}}\leqslant 2\varepsilon_{j}^{\frac{1}{10}}.
    \end{equation*}
    By $\deg B_{j+1}=\deg B_{j}+n_{j}$  and $\varrho(\alpha,A_{j}e^{F_{j}})=\varrho(\alpha,Ae^{F})-\frac{(\operatorname{deg}B_{j})\alpha }{2}\bmod\mathbb{Z}$, 
    \begin{equation}\label{labelapproximation}
        \|2\varrho(\alpha,Ae^{F})-(\operatorname{deg}B_{j+1})\alpha\|_{\mathbb{T}}\leqslant 2\varepsilon_{j}^{\frac{1}{10}}.
    \end{equation}
    The degree estimate in \hyperref[KAM]{Theorem~\ref*{KAM}} gives
    \begin{equation}\label{degreeestimate}
        \ell(\operatorname{deg} B_{j+1}) \leqslant 2N_{j}.
    \end{equation} 

    We first consider \hyperref[item:rational]{Corollary~\ref*{PE}\,\textup{(\ref*{item:rational})}}. If $\operatorname{deg}B_{j+1}\neq k$, one may choose $j$ sufficiently large such that $\ell(k)< N_{j}$. Then, by the subadditivity of $\ell$, \eqref{degreeestimate} and $\alpha\in\mathrm{DC}(\gamma,\tau)$,
    \begin{equation*}
        \begin{split}
            \|2\varrho(\alpha,Ae^{F})-(\operatorname{deg}B_{j+1})\alpha\|_{\mathbb{T}}&=\|(k-\operatorname{deg}B_{j+1})\alpha\|_{\mathbb{T}}\\
            &\geqslant\gamma e^{-\tau(3N_{j})^{\rho}}\gg 2\varepsilon_{j}^{\frac{1}{10}},
        \end{split}
    \end{equation*}
    contradicting \eqref{labelapproximation}. Hence, every sufficiently large resonant step satisfies $\operatorname{deg}B_{j+1}=k$. Choose two consecutive sufficiently large resonant steps $j<j'$. Since all the intermediate steps are non-resonant, one has $\operatorname{deg}B_{j'}=\operatorname{deg}B_{j+1}=k$. However, the $j'$-th resonant step gives
    \begin{equation*}
        \operatorname{deg}B_{j'+1}=\operatorname{deg}B_{j'}+n_{j'}=k+n_{j'}\neq k,
    \end{equation*}
    yielding a contradiction. Thus there are only finitely many resonant steps.
    
    We next consider \hyperref[item:diophantine]{Corollary~\ref*{PE}\,\textup{(\ref*{item:diophantine})}}. By the assumed arithmetic condition and \eqref{degreeestimate}, for sufficiently large $j$,
    \begin{equation*}
        \|2\varrho(\alpha,Ae^{F})-(\operatorname{deg}B_{j+1})\alpha\|_{\mathbb{T}}\geqslant\kappa e^{-\tau(2N_{j})^{\rho}}\gg 2\varepsilon_{j}^{\frac{1}{10}},
    \end{equation*}
    which contradicts \eqref{labelapproximation}. Hence there are only finitely many resonant steps as well.

    By the reducibility conclusion of \hyperref[KAM]{Theorem~\ref*{KAM}}, in both cases there exist $B\in\mathcal{A}_{\sigma'}(2\mathbb{T},\mathrm{SL}(2,\mathbb{R}))$ and $\bar{A}\in\mathrm{SL}(2,\mathbb{R})$ such that
    \begin{equation}\label{constantreduction}
        B(x+\alpha)^{-1}Ae^{F(x)}B(x)=\bar{A}.
    \end{equation}

    We now identify $\bar{A}$. If $2\varrho(\alpha,Ae^{F})=k\alpha \bmod \mathbb{Z}$ for some $k\in \mathbb{Z}$, we have $2\varrho(\bar{A})=(k-\operatorname{deg}B)\alpha\bmod\mathbb{Z}$. We claim that $k=\operatorname{deg}B$. Otherwise, if $k \neq \operatorname{deg}B$, then for all sufficiently large $j$,
    \begin{equation*}
        \begin{split}
            \|2\varrho(A_{j})-(k-\operatorname{deg}B)\alpha\|_{\mathbb{T}}&\leqslant 2|\varrho(A_{j})-\varrho(\bar{A})|\\
            &\leqslant C \bigg(\sum_{i=j}^{\infty}\|A_{i+1}-A_{i}\|\bigg)^{\frac{1}{2}}\leqslant C\varepsilon_{j}^{\frac{1}{2}}\ll \varepsilon_{j}^{\frac{1}{10}},
        \end{split}
    \end{equation*}
    which contradicts the assumption that there are only finitely many resonant steps. Since uniform hyperbolicity is invariant under conjugacy and $(\alpha,Ae^{F})\notin\mathcal{UH}$, the constant matrix $\bar{A}$ is not hyperbolic. Consequently, there exist $P\in\mathrm{SL}(2,\mathbb{R})$ and $\zeta\in\mathbb{R}$ such that
    \begin{equation*}
        P^{-1}\bar{A}P=\pm\begin{pmatrix}
            1&\zeta\\
            0&1
        \end{pmatrix}.
    \end{equation*}
    Replacing $B$ by $BP$ in \eqref{constantreduction} proves \hyperref[item:rational]{Corollary~\ref*{PE}\,\textup{(\ref*{item:rational})}}.

    Finally, under the assumption of \hyperref[item:diophantine]{Corollary~\ref*{PE}\,\textup{(\ref*{item:diophantine})}}, one has
    \begin{equation*}
        \|2\varrho(\bar{A})\|_{\mathbb{T}}=\|2\varrho(\alpha,Ae^{F})-(\operatorname{deg}B)\alpha\|_{\mathbb{T}}\geqslant \kappa e^{-\tau\ell(\operatorname{deg}B)^{\rho}}>0.
    \end{equation*}
    Therefore, $\bar{A}$ is elliptic. Hence there exist $P\in\mathrm{SL}(2,\mathbb{R})$ and $\phi\in\mathbb{T}$ such that
    \begin{equation*}
        P^{-1}\bar{A}P=R_{\phi}.
    \end{equation*}
    Replacing $B$ by $BP$ in \eqref{constantreduction} proves \hyperref[item:diophantine]{Corollary~\ref*{PE}\,\textup{(\ref*{item:diophantine})}}.
\end{proof}

\subsection{Triangularization and growth of cocycles}
\begin{corollary}\label{NUHKAM}
    Under the assumptions of \hyperref[KAM]{Theorem~\ref*{KAM}}, suppose  $(\alpha,Ae^{F})\notin\mathcal{UH}$. Then, for every $j\geqslant 0$, there exist $\widehat{B}_{j}\in\mathcal{A}_{\sigma_{j}}(2\mathbb{T},\mathrm{SL}(2,\mathbb{C}))$, $\widehat{A}_{j}\in\mathrm{SL}(2,\mathbb{C})$, and $\widehat{F}_{j}\in\mathcal{A}_{\sigma_{j}}(\mathbb{T},\mathrm{sl}(2,\mathbb{C}))$ such that
    \begin{equation}\label{nuh}
        \widehat{B}_{j}(x+\alpha)^{-1}Ae^{F(x)}\widehat{B}_{j}(x)=\widehat{A}_{j}e^{\widehat{F}_{j}(x)},
    \end{equation}
    where $\widehat{A}_{j}=
        \begin{pmatrix}
            e^{2\pi i\widehat{\varrho}_{j}}&\widehat{\zeta}_{j}\\
            0&e^{-2\pi i\widehat{\varrho}_{j}}
        \end{pmatrix}$ with $\widehat{\varrho}_{j}\in\mathbb{R}$ and  $\widehat{\zeta}_{j}\in\mathbb{C}$.
    Moreover,
    \begin{equation*}
        \begin{split}
        &\|\widehat{B}_{j}\|_{C^{0}}\leqslant\exp\big(C_{\rho}|\log\varepsilon_{j-1}|^{\frac{1+\rho}{2}}\big),\qquad \|\widehat{F}_{j}\|_{\sigma_{j}}\leqslant 4\varepsilon_{j}^{\frac{1}{4}},\\
            &\|\widehat{A}_{j}\|\leqslant 4\|A\|,
        \qquad \|\widehat{B}_{j}\|_{C^{0}}^{2}\, |\widehat{\zeta}_{j}|\leqslant 8\|A\|.
        \end{split}
    \end{equation*}
\end{corollary}

\begin{proof}
    By the almost reducibility and triangularization estimates in \hyperref[KAM]{Theorem~\ref*{KAM}}, there exist $U_{j}\in\mathrm{SU}(2)$, $\chi_{j}\in\{\pm1\}$, $\varrho_{j}\in i\mathbb{R}\cup\mathbb{R}$, and $\zeta_{j}\in\mathbb{C}$ such that
    \begin{equation*}
        T_{j}\coloneqq U_{j}^{-1}A_{j}U_{j}
        =
        \chi_{j}
        \begin{pmatrix}
            e^{2\pi i\varrho_{j}}&\zeta_{j}\\
            0&e^{-2\pi i\varrho_{j}}
        \end{pmatrix}.
    \end{equation*}
    Let $\widehat{B}_{j}=B_{j}U_{j}$ and $G_{j}=U_{j}^{-1}F_{j}U_{j}$.
    Then
    \begin{equation*}
        \widehat{B}_{j}(x+\alpha)^{-1}Ae^{F(x)}\widehat{B}_{j}(x)=T_{j}e^{G_{j}(x)},
    \end{equation*}
    with
    \begin{equation}\label{BGest}
        \begin{split}
            &\|\widehat{B}_{j}\|_{C^{0}}=\|B_{j}\|_{C^{0}}\leqslant \exp\big(C_{\rho}|\log\varepsilon_{j-1}|^{\frac{1+\rho}{2}}\big),\\
            &\|G_{j}\|_{\sigma_{j}}=\|F_{j}\|_{\sigma_{j}}\leqslant\varepsilon_{j}.
        \end{split}
    \end{equation}

    If $\varrho_{j}\in\mathbb{R}$, we set
    \begin{equation*}
        \widehat{A}_{j}=T_{j},
        \qquad
        \widehat{F}_{j}=G_{j},
        \qquad
        \widehat{\varrho}_{j}=\varrho_{j}+\frac{1-\chi_{j}}{4},
        \qquad
        \widehat{\zeta}_{j}=\chi_{j}\zeta_{j}.
    \end{equation*}
    The desired estimates follow immediately from \eqref{BGest} and the triangularization estimate in \hyperref[KAM]{Theorem~\ref*{KAM}}.

    It remains to consider the case $\rho_{j}\in i\mathbb{R}\setminus\{0\}$. After interchanging the two eigenvalues if necessary, we may write $T_{j}=
        \chi_{j}\begin{pmatrix}
            e^{-\lambda_{j}}&\zeta_{j}\\
            0&e^{\lambda_{j}}
        \end{pmatrix}$ where $\lambda_{j}>0$. 

    We first show that $\lambda_{j}<\varepsilon_{j}^{\frac{1}{4}}$. Suppose otherwise that $\lambda_{j}\geqslant\varepsilon_{j}^{\frac{1}{4}}$. Define
    \begin{equation*}
        Q_{j}=
        \begin{pmatrix}
            1&\frac{\zeta_{j}}{e^{\lambda_{j}}-e^{-\lambda_{j}}}\\
            0&1
        \end{pmatrix}\in\mathrm{SL}(2,\mathbb{C}).
    \end{equation*}
    Since $|\zeta_{j}|\leqslant\|A_{j}\|\leqslant2\|A\|$, we have $\|Q_{j}\|\leqslant 1+4\|A\|\varepsilon_{j}^{-1/4}$. Then
    \begin{equation*}
        Q_{j}^{-1}T_{j}e^{G_{j}(x)}Q_{j}=
        \chi_{j}\begin{pmatrix}
            e^{-\lambda_{j}}&0\\
            0&e^{\lambda_{j}}
        \end{pmatrix}e^{Q_{j}^{-1} G_{j}(x)Q_{j}}\eqqcolon \chi_{j} D_{j}e^{H_{j}(x)},
    \end{equation*}
    where $\|H_{j}\|_{\sigma_{j}}\leqslant 20\|A\|^{2}\varepsilon_{j}^{\frac{1}{2}}$.
    Therefore, 
    \begin{equation}\label{cone}
        \frac{1+e^{-\lambda_{j}}}{1-e^{-\lambda_{j}}}\|e^{H_{j}}-\mathrm{I}\|_{C^{0}}
            \leqslant 100\|A\|^{2}\varepsilon_{j}^{\frac{1}{4}}<1.
    \end{equation}
    Since $(\alpha,\chi_{j}D_{j})\in\mathcal{UH}$  with exponent $\lambda_{j}$, by \cite[Theorem 5.2]{MR4348670} and \eqref{cone}, one has $(\alpha,\chi_{j}D_{j}e^{H_{j}})\in\mathcal{UH}$. Consequently, $(\alpha,Ae^{F})\in\mathcal{UH}$, which contradicts the assumption. Thus $\lambda_{j}<\varepsilon_{j}^{\frac{1}{4}}$.

    We now absorb $\lambda_{j}$ into the perturbation. Let
    \begin{equation*}
        D_{j}=
        \begin{pmatrix}
            e^{-\lambda_{j}}&0\\
            0&e^{\lambda_{j}}
        \end{pmatrix},
        \qquad
        \widehat{\varrho}_{j}=\frac{1-\chi_{j}}{4},
        \qquad
        \widehat{\zeta}_{j}=\chi_{j}e^{-\lambda_{j}}\zeta_{j},
    \end{equation*}
    and define
    \begin{equation*}
        \widehat{A}_{j}=
        \begin{pmatrix}
            e^{2\pi i\widehat{\varrho}_{j}}&\widehat{\zeta}_{j}\\
            0&e^{-2\pi i\widehat{\varrho}_{j}}
        \end{pmatrix}
        =
        \chi_{j}
        \begin{pmatrix}
            1&e^{-\lambda_{j}}\zeta_{j}\\
            0&1
        \end{pmatrix}.
    \end{equation*}
    Then one can check that $T_{j}=\widehat{A}_{j}D_{j}$. Define $\widehat{F}_{j}=\log (D_{j}e^{G_{j}(x)})$.
    Combine $\lambda_{j}<\varepsilon_{j}^{\frac{1}{4}}$ with the BCH formula and \eqref{BGest}, one has
    \begin{equation*}
        \|\widehat{F}_{j}\|_{\sigma_{j}}
        \leqslant 2\big(\lambda_{j}+\|G_{j}\|_{\sigma_{j}}\big)
        \leqslant 4\varepsilon_{j}^{\frac{1}{4}}.
    \end{equation*}
    Thus \eqref{nuh} holds. Moreover, 
    \begin{equation*}
        \|\widehat{A}_{j}\|\leqslant 2\|T_{j}\|\leqslant 4\|A\|,\quad \|\widehat{B}_{j}\|_{C^{0}}^{2}|\widehat{\zeta}_{j}|\leqslant\|B_{j}\|_{C^{0}}^{2}|\zeta_{j}|
        \leqslant8\|A\|.
    \end{equation*}
    This finishes the proof.
\end{proof}

Recall the following product estimate.
    \begin{lemma}[\cite{MR2846380}]\label{prod}
        We have
        \begin{equation*}
            M_{s} (\mathrm{I}+\xi_{s})\cdots M_{0}(\mathrm{I}+\xi_{0})=M^{(s)} (\mathrm{I}+\xi^{(s)}),
        \end{equation*}
        where $M^{(s)}=M_{s}\cdots M_{0}$, and 
        \begin{equation*}
            \|\xi^{(s)}\|\leqslant \exp\bigg(\sum_{k=0}^{s}\|M^{(k)}\|^{2}\|\xi_{k}\|\bigg)-1.
        \end{equation*}
    \end{lemma}

\begin{corollary}\label{growthA}
    Under the assumptions of \hyperref[NUHKAM]{Corollary~\ref*{NUHKAM}}, we have
    \begin{equation*}
        \|\mathcal{A}_{s}\|_{C^{0}}\leqslant C(s+1),
    \end{equation*}
    where $(s\alpha, \mathcal{A}_{s})\coloneqq (\alpha, Ae^{F})^{s}$.
\end{corollary}

\begin{proof}

    By \hyperref[NUHKAM]{Corollary~\ref*{NUHKAM}}, we can express the iterated cocycle as
    \begin{equation*}
        \mathcal{A}_{s}(x)=\widehat{B}_{j}(x+s\alpha)\bigg(\prod_{k=s-1}^{0}\widehat{A}_{j}e^{\widehat{F}_{j}(x+k\alpha)}\bigg)\widehat{B}_{j}(x)^{-1}.
    \end{equation*}
    Using \hyperref[prod]{Lemma~\ref*{prod}} and $\|\widehat{A}_{j}^{k}\|\leqslant 1+k|\widehat{\zeta}_{j}|$ for all $k\geqslant 0$, we deduce
    \begin{equation*}
        \begin{split}
            \|\mathcal{A}_{s}\|_{C^{0}}&\leqslant \|\widehat{B}_{j}\|_{C^{0}}^{2}\|\widehat{A}_{j}^{s}\| \exp\bigg(\|\widehat{F}_{j}\|_{C^{0}}\sum_{k=0}^{s} \|\widehat{A}_{j}^{k}\|^{2}\bigg)\\
            &\leqslant \|\widehat{B}_{j}\|_{C^{0}}^{2} (1+s|\widehat{\zeta}_{j}|) \exp\bigg(4\varepsilon_{j}^{\frac{1}{4}}\sum_{k=0}^{s}(1+k|\widehat{\zeta}_{j}|)^{2}\bigg).
        \end{split}
    \end{equation*}
    For each $j\geqslant 0$, define the interval
    \begin{equation*}
        I_{j}=\Big(\varepsilon_{j}^{-1/100}, \varepsilon_{j}^{-1/20}\Big).
    \end{equation*}
    Since $\varepsilon_{j+1}=4\varepsilon_{j}^{3}$, we have $\varepsilon_{j+1}^{-1/100}<\varepsilon_{j}^{-1/20}$, which implies that $I_{j}\cap I_{j+1}\neq \emptyset$. Therefore, the intervals $\{I_{j}\}_{j\geqslant 0}$ cover all sufficiently large integers $s$.

    Fix any $s\in I_{j}$. Using the bounds $\|\widehat{B}_{j}\|_{C^{0}}^{2}\, |\widehat{\zeta}_{j}|\leqslant 8\|A\|$ and $\|\widehat{B}_{j}\|_{C^{0}}^{2}\leqslant \exp\big(C_{\rho}|\log \varepsilon_{j-1}|^{\frac{1+\rho}{2}}\big)\leqslant s$, we obtain
    \begin{equation*}
        \|\mathcal{A}_{s}\|_{C^{0}}\leqslant 2\|\widehat{B}_{j}\|_{C^{0}}^{2} (1+s|\widehat{\zeta}_{j}|) \leqslant 10 \|A\|(s+1).
    \end{equation*}
\end{proof}

\subsection{Lyapunov exponent}
The Lyapunov exponent of the cocycle $(\alpha,A)$ is defined as
\begin{equation*}
    L(\alpha,A)=\lim_{n\to \infty} \frac{1}{n}\int_{\mathbb{T}} \log \|A(x+(n-1)\alpha)\cdots A(x)\|\,\mathrm{d}x.
\end{equation*}

\begin{theorem}\label{zeroLE}
    Under the assumptions of \hyperref[KAM]{Theorem~\ref*{KAM}}, suppose that $(\alpha,Ae^{F})\notin\mathcal{UH}$. Then $L(\alpha,Ae^{F})=0$.
\end{theorem}
\begin{proof}
    Since $\det Ae^{F}\equiv 1$.  \hyperref[growthA]{Corollary~\ref*{growthA}} gives that 
    \begin{equation*}
        0\leqslant L(\alpha,Ae^{F}) \leqslant \frac{1}{s}\log \|\mathcal{A}_{s}\|_{C^{0}}\to 0.
    \end{equation*}
    This finishes the proof.
\end{proof}

\begin{theorem}\label{LEholder}
    Under the assumptions of \hyperref[KAM]{Theorem~\ref*{KAM}}, denote $\mathcal{A}(x)=Ae^{F(x)}$. Then there exists $C>0$ such that, for every $\widetilde{\mathcal{A}}\in C^{0}(\mathbb{T},\mathrm{SL}(2,\mathbb{C}))$,
    \begin{equation*}
        |L(\alpha,\mathcal{A})-L(\alpha,\widetilde{\mathcal{A}})|\leqslant C\|\mathcal{A}-\widetilde{\mathcal{A}}\|_{C^{0}}^{\frac{1}{2}}.
    \end{equation*}
\end{theorem}

\begin{proof}
    If $(\alpha,\mathcal{A})\in\mathcal{UH}$, the Lyapunov exponent is locally Lipschitz continuous in the $C^{0}$ topology; see \cite{MR534172}. It remains to consider the case $(\alpha,\mathcal{A})\notin\mathcal{UH}$. By \hyperref[zeroLE]{Theorem~\ref*{zeroLE}}, we have $L(\alpha,\mathcal{A})=0$. 

    Let 
    \begin{equation*}
        \delta=\|\mathcal{A}-\widetilde{\mathcal{A}}\|_{C^{0}}.
    \end{equation*}
    By \hyperref[NUHKAM]{Corollary~\ref*{NUHKAM}}, for every $j\geqslant 0$, there exist $\widehat{B}_{j}$, $\widehat{A}_{j}$, and $\widehat{F}_{j}$ such that
    \begin{equation}\label{ARholder}
        \widehat{B}_{j}(x+\alpha)^{-1}\mathcal{A}(x)\widehat{B}_{j}(x)=\widehat{A}_{j}e^{\widehat{F}_{j}(x)}=\begin{pmatrix}e^{2\pi i\widehat{\varrho}_{j}}&\widehat{\zeta}_{j}\\0&e^{-2\pi i\widehat{\varrho}_{j}}\end{pmatrix}e^{\widehat{F}_{j}(x)},
    \end{equation}
    where $\widehat{\varrho}_{j}\in\mathbb{R}$ and
    \begin{equation}\label{quantARholder}
        \begin{split}
            &\|\widehat{B}_{j}\|_{C^{0}}\leqslant\exp\big(C_{\rho}|\log\varepsilon_{j-1}|^{\frac{1+\rho}{2}}\big),\qquad \|\widehat{F}_{j}\|_{C^{0}}\leqslant 4\varepsilon_{j}^{\frac{1}{4}},\\
            &\|\widehat{A}_{j}\|\leqslant 4\|A\|,\qquad \|\widehat{B}_{j}\|_{C^{0}}^{2}|\widehat{\zeta}_{j}|\leqslant 8\|A\|.
        \end{split}
    \end{equation}

    Write
    \begin{equation}\label{qdecomp}
        \widehat{A}_{j}e^{\widehat{F}_{j}(x)}=\begin{pmatrix}e^{2\pi i\widehat{\varrho}_{j}}&0\\0&e^{-2\pi i\widehat{\varrho}_{j}}\end{pmatrix}+\begin{pmatrix}q_{11}^{(j)}(x)&q_{12}^{(j)}(x)\\q_{21}^{(j)}(x)&q_{22}^{(j)}(x)\end{pmatrix}.
    \end{equation}
    By \eqref{quantARholder}, there exists $C_{0}=C_{0}(\|A\|)\geqslant 1$ such that
    \begin{equation}\label{qest}
        \begin{split}
            &\|q_{11}^{(j)}\|_{C^{0}}+\|q_{21}^{(j)}\|_{C^{0}}+\|q_{22}^{(j)}\|_{C^{0}}\leqslant C_{0}\varepsilon_{j}^{\frac{1}{4}},\\
            &\|q_{12}^{(j)}\|_{C^{0}}\leqslant |\widehat{\zeta}_{j}|+C_{0}\varepsilon_{j}^{\frac{1}{4}}.
        \end{split}
    \end{equation}

    For every $j\geqslant 0$, define the intervals
    \begin{equation*}
        I_{j}=\Big(C_{0}\varepsilon_{j}^{\frac{1}{4}}, \exp\big(-4C_{\rho}|\log\varepsilon_{j-1}|^{\frac{1+\rho}{2}}\big)\Big).
    \end{equation*}
    Since $\varepsilon_{j+1}=4\varepsilon_{j}^{3}$, one has $\exp\big(-4C_{\rho}|\log\varepsilon_{j}|^{\frac{1+\rho}{2}}\big)\geqslant C_{0}\varepsilon_{j}^{\frac{1}{4}}$. Hence $I_{j}\cap I_{j+1}\neq \emptyset$. Therefore $\{I_{j}\}_{j\geqslant 0}$ covers all sufficiently small $\delta>0$. 

    Fix $j\geqslant 0$ such that $\delta\in I_{j}$. Let $D_{j}=\operatorname{diag}\{d_{j}, d_{j}^{-1}\}$ where $d_{j}=\|\widehat{B}_{j}\|_{C^{0}}\delta^{\frac{1}{4}}$. By \eqref{quantARholder} and $\delta\in I_{j}$, one has
    \begin{equation}\label{dj}
        d_{j}\leqslant\exp\big(C_{\rho}|\log\varepsilon_{j-1}|^{\frac{1+\rho}{2}}\big)\exp\big(-C_{\rho}|\log\varepsilon_{j-1}|^{\frac{1+\rho}{2}}\big)\leqslant 1.
    \end{equation}
    Define $W_{j}(x)=\widehat{B}_{j}(x)D_{j}^{-1}\in\mathrm{SL}(2,\mathbb{C})$. By \eqref{dj}, we have
    \begin{equation}\label{West}
        \|W_{j}\|_{C^{0}}\leqslant 2d_{j}^{-1} \|\widehat{B}_{j}\|_{C^{0}}\leqslant 2\delta^{-\frac{1}{4}}.
    \end{equation}

    By \eqref{ARholder} and \eqref{qdecomp},
    \begin{equation*}
        \begin{split}
            W_{j}(x+\alpha)^{-1}\mathcal{A}(x)W_{j}(x)&=\mathcal{Z}_{j}(x)\\
            &\coloneqq\begin{pmatrix}e^{2\pi i\widehat{\varrho}_{j}}&0\\0&e^{-2\pi i\widehat{\varrho}_{j}}\end{pmatrix}+\begin{pmatrix}q_{11}^{(j)}(x)&d_{j}^{2}q_{12}^{(j)}(x)\\d_{j}^{-2}q_{21}^{(j)}(x)&q_{22}^{(j)}(x)\end{pmatrix}.
        \end{split}
    \end{equation*}
    Since $\delta\in I_{j}$, \eqref{qest} gives
    \begin{equation*}
        \begin{split}
            &\|q_{11}^{(j)}\|_{C^{0}}+\|q_{22}^{(j)}\|_{C^{0}}\leqslant \delta\leqslant \delta^{\frac{1}{2}},\\
            &d_{j}^{-2}\|q_{21}^{(j)}\|_{C^{0}}=\|\widehat{B}_{j}\|_{C^{0}}^{-2}\delta^{-\frac{1}{2}}\|q_{21}^{(j)}\|_{C^{0}}\leqslant \delta^{\frac{1}{2}}.
        \end{split}
    \end{equation*}
    
    Moreover, by \eqref{quantARholder}, \eqref{qest}, and \eqref{dj},
    \begin{equation*}
        \begin{split}
            d_{j}^{2}\|q_{12}^{(j)}\|_{C^{0}}&\leqslant\|\widehat{B}_{j}\|_{C^{0}}^{2}\delta^{\frac{1}{2}}|\widehat{\zeta}_{j}|+d_{j}^{2}C_{0}\varepsilon_{j}^{\frac{1}{4}}\\
            &\leqslant 8\|A\|\delta^{\frac{1}{2}}+\delta\\
            &\leqslant (8\|A\|+1)\delta^{\frac{1}{2}}.
        \end{split}
    \end{equation*}
    
    Consequently,
    \begin{equation}\label{Zest}
        \|\mathcal{Z}_{j}\|_{C^{0}}\leqslant 1+C_{A}\delta^{\frac{1}{2}}.
    \end{equation}

    Define $\widetilde{\mathcal{Z}}_{j}(x)=W_{j}(x+\alpha)^{-1}\widetilde{\mathcal{A}}(x)W_{j}(x)$. By \eqref{West},
    \begin{equation*}
        \|\widetilde{\mathcal{Z}}_{j}-\mathcal{Z}_{j}\|_{C^{0}} \leqslant\|W_{j}^{-1}\|_{C^{0}}\|\widetilde{\mathcal{A}}-\mathcal{A}\|_{C^{0}}\|W_{j}\|_{C^{0}}\leqslant 4\delta^{\frac{1}{2}}.
    \end{equation*}
    Combining this with \eqref{Zest}, we obtain
    \begin{equation*}
        \|\widetilde{\mathcal{Z}}_{j}\|_{C^{0}}\leqslant 1+C_{A}\delta^{\frac{1}{2}}.
    \end{equation*}

    Since the Lyapunov exponent is invariant under continuous conjugacies,
    \begin{equation*}
        L(\alpha,\widetilde{\mathcal{A}})=L(\alpha,\widetilde{\mathcal{Z}}_{j})\leqslant \log\|\widetilde{\mathcal{Z}}_{j}\|_{C^{0}}\leqslant C_{A}\delta^{\frac{1}{2}}.
    \end{equation*}
    The conclusion then follows from $L(\alpha,\mathcal{A})=0$.
\end{proof}

\begin{proof}[Proof of {\hyperref[thm:intro-kam]{Theorem~\ref*{thm:intro-kam}}}]
    Fix any admissible Fourier length $\ell$. Fix any compact subset $\mathcal{K}\subseteq \mathrm{SL}(2,\mathbb{R})$, one may choose $\varepsilon_{0}$ in \hyperref[KAM]{Theorem~\ref*{KAM}} uniform with respect to $A\in\mathcal{K}$. Consequently, the almost reducibility follows from \hyperref[KAM]{Theorem~\ref*{KAM}}.

    If there are only finitely many resonant steps, reducibility follows from the final conclusion of \hyperref[KAM]{Theorem~\ref*{KAM}}.

    Under the stated non-uniform-hyperbolicity assumption, the vanishing of $L(\alpha,\mathcal{A})$ follows from \hyperref[zeroLE]{Theorem~\ref*{zeroLE}}.

    The $1/2$-H\"older continuity of $L(\alpha,\mathcal{A})$ follows from \hyperref[LEholder]{Theorem~\ref*{LEholder}}.
\end{proof}

\section{Schr\"odinger cocycles and spectral results}

\subsection{Reducibility and almost reducibility}
Let $\Sigma_{\alpha,V}$ denote the phase-independent spectrum. 
For the Schr\"odinger cocycle $(\alpha, S_{E}^{V})$, denote $\varrho(E)=\varrho(\alpha,S_{E}^{V})$. The following holds.

\begin{theorem}\label{SOKAM}
    Let $0<\rho<1$, $\alpha\in\mathrm{DC}(\gamma,\tau)$, and $\sigma_{0}>0$. There exists $\varepsilon_{*}=\varepsilon_{*}(\gamma,\tau,\sigma_{0},\rho)>0$ such that for any $V\in \mathcal{A}_{\sigma_{0}}(\mathbb{T},\mathbb{R})$ satisfying $\|V\|_{\sigma_{0}}<\varepsilon_{*}$, the following statements hold.
    \begin{enumerate}
        \item \label{item:SOAR} For every $E\in\Sigma_{\alpha,V}$, there exist $B_{j}$, $A_{j}(E)$, and $F_{j}$ such that
        \begin{equation*}
            B_{j}(x+\alpha)^{-1} S_{E}^{V}(x)B_{j}(x)=A_{j}(E)e^{F_{j}(x)}.
        \end{equation*}
        Moreover, the conjugacy, degree, and triangularization estimates in \hyperref[KAM]{Theorem~\ref*{KAM}} hold.
        
        \item For every $E\in\Sigma_{\alpha,V}$, if $2\varrho(E)=k\alpha \bmod{\mathbb{Z}}$ for some $k\in\mathbb{Z}$, then $(\alpha, S_{E}^{V})$ is reducible to a parabolic constant cocycle.
    \end{enumerate}
\end{theorem}

\begin{proof}
    Let $\sigma_{*}=\frac{\sigma_{0}}{2}$. For every $E\in\mathbb{R}$, we define
    \begin{equation*}
        A(E)=\begin{pmatrix}E&-1\\1&0\end{pmatrix}, \qquad F_{V}(x)=\begin{pmatrix}0&0\\V(x)&0\end{pmatrix}.
    \end{equation*}
    Observe that $\|F_{V}\|_{\sigma_{0}}\leqslant \|V\|_{\sigma_{0}}$ and
    \begin{equation*}
        A(E)e^{F_{V}(x)}=\begin{pmatrix}E-V(x)&-1\\1&0\end{pmatrix}=S_{E}^{V}(x).
    \end{equation*}

    For $\varepsilon_{*}>0$ sufficiently small, \hyperref[thm:intro-kam]{Theorem~\ref*{thm:intro-kam}} applies with widths $\sigma_{0}$ and $\sigma_{*}$ to the compact set
    \begin{equation*}
        \mathcal{K}=\bigg\{\begin{pmatrix}E&-1\\1&0\end{pmatrix} : |E|\leqslant 3\bigg\}.
    \end{equation*}
    For $E\in\Sigma_{\alpha,V}$, the spectrum bound yields
    \begin{equation*}
        |E|\leqslant 2+\|V\|_{C^{0}}\leqslant 2+\|V\|_{\sigma_{0}}<3.
    \end{equation*}
    Hence $A(E)\in\mathcal{K}$ for all $E\in\Sigma_{\alpha,V}$. Since $\Sigma_{\alpha,V}$ is compact, $\varepsilon_{*}$ can be chosen uniformly with respect to $E\in\Sigma_{\alpha,V}$. This verifies all the assumptions of \hyperref[KAM]{Theorem~\ref*{KAM}}. Consequently, \hyperref[KAM]{Theorem~\ref*{KAM}} implies that $(\alpha,S_{E}^{V})$ is quantitatively almost reducible for every $E\in\Sigma_{\alpha,V}$, providing the sequence of transformations $B_{j}$, $A_{j}(E)$, and $F_{j}$ satisfying the desired estimates.

    Finally, if $2\varrho(E)=k\alpha \bmod{\mathbb{Z}}$ for some $k\in\mathbb{Z}$, the reducibility to a parabolic cocycle follows directly from \hyperref[item:rational]{Corollary~\ref*{PE} (\ref*{item:rational})}.
\end{proof}

\subsection{Integrated density of states}
Recall that the spectral measure $\mu_{\delta_{i}}, i=0,1$ of $H_{x,\alpha,V}$ is given by
\begin{equation*}
    \langle \delta_{i},(H_{x,\alpha,V}-z)^{-1}\delta_{i}\rangle= \int_{\mathbb{R}} \frac{1}{E-z}\,\mathrm{d}\mu_{\delta_{i}}(E).
\end{equation*}
Let $\mu_{x,\alpha,V}=\mu_{\delta_{0}}+\mu_{\delta_{1}}$ be the universal spectral measure.
The density of states  measure $\mathrm{d}N_{\alpha,V}$ is the half average of the spectral measure,
\begin{equation*}
    \mathrm{d}N_{\alpha,V}=\frac{1}{2}\int_{\mathbb{T}} \mathrm{d}\mu_{x,\alpha,V} \mathrm{d}x.
\end{equation*}
The integrated density of states (IDS) is defined as
\begin{equation*}
    N_{\alpha,V}(E)\coloneqq \int_{-\infty}^{E} \mathrm{d}N_{\alpha,V}.
\end{equation*}
Moreover, IDS is related to the fibered rotation number via
\begin{equation}\label{idsrotat}
    N_{\alpha,V}(E)=1-2\varrho(E).
\end{equation}

\begin{corollary}\label{idshalfholder}
    Under the assumptions of \hyperref[SOKAM]{Theorem~\ref*{SOKAM}}, for any $\eta>0$,
    \begin{equation*}
        N_{\alpha,V}(E+\eta)-N_{\alpha,V}(E-\eta)\leqslant C \eta^{\frac{1}{2}}.
    \end{equation*}
\end{corollary}
\begin{proof}
    Let $E\in\mathbb{R}$ and $\eta>0$. We apply \hyperref[LEholder]{Theorem~\ref*{LEholder}} to the Schr\"odinger cocycle. Denote $L(E)\coloneqq L(\alpha, S_{E}^{V})$ for short. Let $\mathcal{A}(x)=S_{E}^{V}(x)$ and $\widetilde{\mathcal{A}}(x)=S_{E+i\eta}^{V}$. Thus
    \begin{equation}\label{complexLE}
        0\leqslant L(E+i\eta)-L(E)\leqslant C\eta^{1/2}.
    \end{equation}
    Thouless formula gives that
    \begin{equation*}
        L(E)=\int_{\mathbb{R}}\log |E'-E|\,\mathrm{d}N_{\alpha,V}(E').
    \end{equation*}
    Thus
    \begin{equation*}
        L(E+i\eta)-L(E)=\frac{1}{2} \int_{\mathbb{R}} \log\bigg( 1+\frac{\eta^{2}}{(E-E')^{2}}\bigg) \,\mathrm{d}N_{\alpha,V}(E').
    \end{equation*}
Therefore,
    \begin{equation*}
        L(E+i\eta)-L(E)
        \geqslant
        \frac{\log 2}{2}
        \big(
            N_{\alpha,V}(E+\eta)
            -
            N_{\alpha,V}(E-\eta)
        \big).
    \end{equation*}
    Combining this estimate with \eqref{complexLE}, we complete the proof.
\end{proof}

\subsection{Absolutely continuous spectrum}

\begin{theorem}\label{AC}
    Under the assumptions of \hyperref[SOKAM]{Theorem~\ref*{SOKAM}}, the operator $H_{x,\alpha,V}$ has purely absolutely continuous spectrum for every $x\in\mathbb{T}$.
\end{theorem}

\begin{proof}
    Let $\Sigma=\Sigma_{\alpha,V}$. Define 
    \begin{equation*}
        \begin{split}
            \mathscr{B}&=\big\{E\in\Sigma: \sup_{s\in \mathbb{Z}}\|(S_{E}^{V})_{s}\|_{C^{0}}<\infty\big\},\\
            \mathscr{R}&=\{E\in\Sigma: (\alpha, S_{E}^{V}) \ \text{is reducible}\}.
        \end{split}
    \end{equation*}
    By the Gilbert and Pearson subordinacy theory \cite{MR915965}, the restricted measure $\mu_{x,\alpha,V}|_{\mathscr{B}}$ is absolutely continuous for every $x\in\mathbb{T}$. Thus, it remains to show that $\mu_{x,\alpha,V}(\Sigma\setminus\mathscr{B})=0$. 

    For any $E\in \mathscr{R}\setminus\mathscr{B}$, the cocycle $(\alpha,S_{E}^{V})$ is reducible to a parabolic one, and according to the gap-labeling theorem \cite{MR667409}, there exists a unique $n\in\mathbb{Z}$ such that $2\varrho(E)=n\alpha\bmod\mathbb{Z}$. Consequently, the set $\mathscr{R}\setminus \mathscr{B}$ is countable. Moreover, there are no eigenvalues in $\mathscr{R}$ because any non-zero solution of $H_{x,\alpha,V}u=Eu$ satisfies $\inf_{n\in\mathbb{Z}} \big(|u(n+1)|^{2}+|u(n)|^{2}\big)>0$, which further implies that $\mu_{x,\alpha,V}(\mathscr{R}\setminus\mathscr{B})=0$. Hence, we only need to show that $\mu_{x,\alpha,V}(\Sigma\setminus \mathscr{R})=0$. 

    For every $j\geqslant 0$, define
    \begin{equation*}
        K_{j}\coloneqq \big\{E\in\Sigma: A_{j}(E)\in \mathcal{RS}_{\ell}(N_{j},\varepsilon_{j}^{\frac{1}{10}})\big\}.
    \end{equation*}
    By the reducibility conclusion of \hyperref[KAM]{Theorem~\ref*{KAM}}, we have $\Sigma\setminus \mathscr{R}\subseteq \limsup_{j\to\infty}K_{j}$. Thus, by the Borel--Cantelli lemma, it suffices to show that 
    \begin{equation}\label{BCA}
        \sum_{j\geqslant 1}\mu_{x,\alpha,V}(K_{j})<\infty.
    \end{equation}

    Recall the following standard estimate.
    \begin{lemma}[{\cite[Lemma 2.5]{avila2008absolutelycontinuousspectrummathieu}}]\label{Avilamu}
        There exists a universal constant $C>0$ such that
        \begin{equation*}
            \mu_{x,\alpha,V}(E-\epsilon,E+\epsilon)\leqslant C\epsilon \sup_{0\leqslant s\leqslant C\epsilon^{-1}} \|(S_{E}^{V})_{s}\|_{C^{0}}^{2}.
        \end{equation*}
    \end{lemma}

    \begin{lemma}\label{rsgrowth}
        Let $C>0$ be the constant in \hyperref[Avilamu]{Lemma~\ref*{Avilamu}}. For any $E\in K_{j}$,
        \begin{equation*}
            \sup_{0\leqslant s\leqslant C\varepsilon_{j}^{-\frac{1}{20}}}\|(S_{E}^{V})_{s}\|_{C^{0}}\leqslant 4\exp\big(2C_{\rho}|\log\varepsilon_{j}|^{\frac{1+\rho}{2}}\big).
        \end{equation*}
    \end{lemma}
    \begin{proof}
        Since $E\in K_{j}$, the KAM step at scale $j$ is resonant. By \hyperref[onestep]{Theorem~\ref*{onestep}} and the iterative estimates in \hyperref[KAM]{Theorem~\ref*{KAM}}, there exists $\chi_{j+1}\in\{\pm 1\}$ such that
        \begin{equation*}
            B_{j+1}(x+\alpha)^{-1}S_{E}^{V}(x)B_{j+1}(x)=A_{j+1}e^{F_{j+1}(x)},
        \end{equation*}
        where
        \begin{equation*}
            \begin{split}
                &\|B_{j+1}\|_{C^{0}}\leqslant\exp\big(C_{\rho}|\log\varepsilon_{j}|^{\frac{1+\rho}{2}}\big),\\
                &\|F_{j+1}\|_{\sigma_{j+1}}\leqslant\varepsilon_{j}^{100},\qquad \|A_{j+1}-\chi_{j+1}\mathrm{I}\|\leqslant 10\varepsilon_{j}^{\frac{1}{10}}.
            \end{split}
        \end{equation*}
        For any $0\leqslant s\leqslant C\varepsilon_{j}^{-\frac{1}{20}}$, we have $\|A_{j+1}^{s}\|\leqslant \exp(10s \varepsilon_{j}^{\frac{1}{10}})\leqslant 2$. Thus by \hyperref[prod]{Lemma~\ref*{prod}}, we write
        \begin{equation*}
            \prod_{k=s-1}^{0}A_{j+1}e^{F_{j+1}(x+k\alpha)}=A_{j+1}^{s} (\mathrm{I}+\Xi_{s}(x)),
        \end{equation*}
        where
        \begin{equation*}
            \|\Xi_{s}(x)\|_{C^{0}}\leqslant \exp(Cs\varepsilon_{j}^{100})-1\leqslant 1.
        \end{equation*}
        Thus
        \begin{equation*}
            \sup_{0\leqslant s\leqslant C\varepsilon_{j}^{-\frac{1}{20}}}\|(S_{E}^{V})_{s}\|_{C^{0}} \leqslant 4 \|B_{j+1}\|_{C^{0}}^{2} \leqslant 4\exp\big(2C_{\rho}|\log\varepsilon_{j}|^{\frac{1+\rho}{2}}\big).
        \end{equation*}
    \end{proof}

    By \hyperref[Avilamu]{Lemma~\ref*{Avilamu}} and \hyperref[rsgrowth]{Lemma~\ref*{rsgrowth}}, we have for any $E\in K_{j}$,
    \begin{equation}\label{single}
        \mu_{x,\alpha,V}\big(E-\varepsilon_{j}^{\frac{1}{20}},E+\varepsilon_{j}^{\frac{1}{20}}\big)\leqslant C\varepsilon_{j}^{\frac{1}{20}}\exp\big(4C_{\rho}|\log\varepsilon_{j}|^{\frac{1+\rho}{2}}\big).
    \end{equation}

    As a consequence of the Thouless formula and the $1/2$-H\"older continuity of the IDS (see \hyperref[idshalfholder]{Corollary~\ref*{idshalfholder}}), we obtain the following lower bound on $N(E)\coloneqq N_{\alpha,V}(E)$. We omit the proof because it closely follows the one in \cite{avila2008absolutelycontinuousspectrummathieu}.

    \begin{lemma}[{\cite[Lemma 3.11]{avila2008absolutelycontinuousspectrummathieu}}]\label{idslower}
        There exists a constant $c=c(\alpha,V)>0$ such that for any $E\in\Sigma$ and $0<\epsilon<1$,
        \begin{equation*}
            N(E+\epsilon)-N(E-\epsilon)\geqslant c\epsilon^{\frac{3}{2}}|\log \epsilon|^{-1}.
        \end{equation*}
    \end{lemma}

    \begin{lemma}\label{coverK}
        There exists a constant $C>0$ such that, for every sufficiently large $j$, there are $\{E_{r}\}_{r=1}^{R}\subseteq K_{j}$ satisfying
        \begin{equation*}
            R\leqslant 12Q_{\ell}(2N_{j})\leqslant 12C_{\ell}\exp\big(c_{\ell}(2N_{j})^{\rho}\big)
        \end{equation*}
        and
        \begin{equation*}
            \overline{K_{j}}\subseteq\bigcup_{r=1}^{R}\big(E_{r}-\varepsilon_{j}^{\frac{1}{20}},E_{r}+\varepsilon_{j}^{\frac{1}{20}}\big).
        \end{equation*}
    \end{lemma}

    \begin{proof}
        Fix $E\in K_{j}$. Then there exists $n_{j}\in\mathbb{Z}$ satisfying
        \begin{equation*}
            0<\ell(n_{j})\leqslant N_{j},\qquad \|2\varrho(A_{j}(E))-n_{j}\alpha\|_{\mathbb{T}}\leqslant\varepsilon_{j}^{\frac{1}{10}}.
        \end{equation*}
        Moreover,
        \begin{equation*}
            B_{j}(x+\alpha)^{-1}S_{E}^{V}(x)B_{j}(x)=A_{j}(E)e^{F_{j}(E,x)}.
        \end{equation*}
        By the inductive construction, one has $\ell(\operatorname{deg}B_{j})\leqslant 2N_{j-1}$.
        Let
        \begin{equation*}
            \widetilde{n}_{j}=n_{j}+\operatorname{deg}B_{j}.
        \end{equation*}
        By the subadditivity of $\ell$,
        \begin{equation*}
            \ell(\widetilde{n}_{j})\leqslant\ell(n_{j})+\ell(\operatorname{deg}B_{j})\leqslant 2N_{j-1}+N_{j}\leqslant 2N_{j}.
        \end{equation*}
        Since $2\varrho(\alpha,A_{j}e^{F_{j}})=2\varrho(E)-(\operatorname{deg}B_{j})\alpha\bmod\mathbb{Z}$, and \hyperref[amorholder]{Lemma~\ref*{amorholder}} gives 
        \begin{equation*}
            |\varrho(\alpha,A_{j}e^{F_{j}})-\varrho(A_{j})|\leqslant C\|F_{j}\|_{C^{0}}^{\frac{1}{2}}\leqslant C\varepsilon_{j}^{\frac{1}{2}},
        \end{equation*}
        we obtain
        \begin{equation}\label{interval}
            \|2\varrho(E)-\widetilde{n}_{j}\alpha\|_{\mathbb{T}}\leqslant 2\varepsilon_{j}^{\frac{1}{10}}.
        \end{equation}
        Since \eqref{idsrotat} gives $N(E)=1-2\varrho(E)$, it follows from \eqref{interval} that there exist intervals $J_{n}\subseteq\mathbb{T}$ of length $|J_{n}|=4\varepsilon_{j}^{\frac{1}{10}}$ such that
        \begin{equation*}
            N(K_{j})\coloneqq \{N(E): E\in K_{j}\} \subseteq \bigcup_{\ell(n)\leqslant 2N_{j}}J_{n},
        \end{equation*}
        whose total number is at most $2Q_{\ell}(2N_{j})$.

        Since $\overline{K_{j}}$ is compact, there exists a finite family $\{E_{r}\}_{r=1}^{R}\subseteq K_{j}$ such that
        \begin{equation*}
            \overline{K_{j}}\subseteq\bigcup_{r=1}^{R}I_{r},\qquad I_{r}=\big(E_{r}-\varepsilon_{j}^{\frac{1}{20}},E_{r}+\varepsilon_{j}^{\frac{1}{20}}\big).
        \end{equation*}
        By refining the cover, we may assume that every point of $\mathbb{R}$ belongs to at most two intervals $I_{r}$.

        By \hyperref[idslower]{Lemma~\ref*{idslower}}, for all sufficiently large $j$, we have
        \begin{equation*}
            |N(I_{r})|=N\big(E_{r}+\varepsilon_{j}^{\frac{1}{20}}\big)-N\big(E_{r}-\varepsilon_{j}^{\frac{1}{20}}\big)\geqslant 8\varepsilon_{j}^{\frac{1}{10}}=2|J_{n}|.
        \end{equation*}
        Therefore, for each $J_{n}$, there are at most $6$ intervals $N(I_{r})$ that can intersect $J_{n}$. Consequently, by \eqref{eq:abstract-ball},
        \begin{equation*}
            R\leqslant 12Q_{\ell}(2N_{j})\leqslant 12C_{\ell}\exp\big(c_{\ell}(2N_{j})^{\rho}\big),
        \end{equation*}
        which proves the lemma.
    \end{proof}

    Combining \hyperref[coverK]{Lemma~\ref*{coverK}} with \eqref{single}, we obtain
    \begin{equation*}
        \begin{split}
            \mu_{x,\alpha,V}(K_{j})&\leqslant CQ_{\ell}(2N_{j})\varepsilon_{j}^{\frac{1}{20}}\exp\big(4C_{\rho}|\log\varepsilon_{j}|^{\frac{1+\rho}{2}}\big)\\
            &\leqslant C\varepsilon_{j}^{\frac{1}{20}}\exp\big(4C_{\rho}|\log\varepsilon_{j}|^{\frac{1+\rho}{2}}+c_{\ell}(2N_{j})^{\rho}\big).
        \end{split}
    \end{equation*}
    By the definition of $N_{j}$ and the fact that $0<\rho<1$,
    \begin{equation*}
        4C_{\rho}|\log\varepsilon_{j}|^{\frac{1+\rho}{2}}+c_{\ell}(2N_{j})^{\rho}\leqslant\frac{1}{40}|\log\varepsilon_{j}|.
    \end{equation*}
    Combining this with $\varepsilon_{j+1}=4\varepsilon_{j}^{3}$ gives
    \begin{equation*}
        \sum_{j\geqslant 1}\mu_{x,\alpha,V}(K_{j})\leqslant C\sum_{j\geqslant 1}\varepsilon_{j}^{\frac{1}{40}}<\infty.
    \end{equation*}
    Thus \eqref{BCA} holds. This finishes the proof.
\end{proof}
\begin{proof}[Proof of {\hyperref[thm:intro-schrodinger]{Theorem~\ref*{thm:intro-schrodinger}}}]
    The almost reducibility for every $E\in\Sigma_{\alpha,V}$ follows from \hyperref[SOKAM]{Theorem~\ref*{SOKAM}}.

    (2) The purely absolutely continuous spectrum follows from \hyperref[AC]{Theorem~\ref*{AC}}.

    (3) Let $\mathscr{P}\subseteq\mathscr{R}$ be the set of energies for which the reducing constant is not elliptic. Thus,
    \begin{equation*}
        \mathscr{P}\subseteq \bigcup_{k\in\mathbb{Z}}\big\{E\in\Sigma_{\alpha,V}: 2\varrho(\alpha,S_{E}^{V})=k\alpha\bmod\mathbb{Z}\big\}.
    \end{equation*}
    Consequently, $\mathscr{P}$ is countable. Since $\mu_{x}$ is absolutely continuous, one has
    \begin{equation*}
        \mu_{x}(\mathscr{P})=0.
    \end{equation*}
    Therefore, $(\alpha,S_{E}^{V})$ is reducible to an elliptic constant for $\mu_{x}$-almost every $E$.

    (4) The $1/2$-H\"older continuity of the IDS follows from \hyperref[idshalfholder]{Corollary~\ref*{idshalfholder}}.
\end{proof}

\subsection{Aubry duality and adapted localization}

For 
\begin{equation*}
    V(x)=\sum_{k\in\mathbb{Z}}\widehat{V}(k)e^{2\pi i kx}\in\mathcal{A}_{\sigma}^{\ell}(\mathbb{T},\mathbb{R}),
\end{equation*}
define the dual operator on $\ell^{2}(\mathbb{Z})$ by
\begin{equation}\label{eq:dual}
    [\widehat{H}_{\alpha,\theta,V}u]_{n} = \sum_{k\in\mathbb{Z}}\widehat{V}(k)u_{n-k} + 2\cos 2\pi(\theta+n\alpha)u_{n}.
\end{equation}
We first prove that the regularity of the potential cannot exceed that of the Bloch waves.

\begin{proposition}
\label{prop:bloch-regularity}
    Let $V\in C^{0}(\mathbb{T},\mathbb{R})$, $\alpha\in\mathbb{R}\setminus\mathbb{Q}$, $E\in\mathbb{R}$, $0<r<1$, and $z\in\mathbb{C}$ with $|z|=1$. Suppose that there exist $\psi_{+},\psi_{-}\in C^{r}(\mathbb{T},\mathbb{C})$ satisfying
    \begin{equation}\label{psi+}
        z\psi_{+}(x+\alpha)+z^{-1}\psi_{+}(x-\alpha)+V(x)\psi_{+}(x)=E\psi_{+}(x)
    \end{equation}
    and
    \begin{equation}\label{psi-}
        z^{-1}\psi_{-}(x+\alpha)+z\psi_{-}(x-\alpha)+V(x)\psi_{-}(x)=E\psi_{-}(x).
    \end{equation}
    Assume that the matrix
    \begin{equation*}
        U(x)=
        \begin{pmatrix}
            \psi_{+}(x)&\psi_{-}(x)\\
            z^{-1}\psi_{+}(x-\alpha)&z\psi_{-}(x-\alpha)
        \end{pmatrix}
    \end{equation*}
    is invertible for some $x\in\mathbb{T}$. Then $V\in C^{r}(\mathbb{T},\mathbb{R})$.
\end{proposition}

\begin{proof}
    Let 
    \begin{equation*}
        D=
        \begin{pmatrix}
            z&0\\
            0&z^{-1}
        \end{pmatrix}.
    \end{equation*}
    Equations \eqref{psi+} and \eqref{psi-} imply that the Schr\"odinger cocycle satisfies
    \begin{equation*}
        S_{E}^{V}(x)U(x)=U(x+\alpha)D.
    \end{equation*}
    Since $\det S_{E}^{V}(x)=\det D=1$, one has $\det U(x+\alpha)=\det U(x)$. Thus $\det U$ is constant. The invertibility assumption guarantees that this constant is nonzero, yielding $U^{-1}\in C^{r}(\mathbb{T},\mathrm{GL}(2,\mathbb{C}))$. Therefore,
    \begin{equation*}
        S_{E}^{V}(x)=U(x+\alpha)DU(x)^{-1}\in C^{r}(\mathbb{T},\mathrm{SL}(2,\mathbb{C})).
    \end{equation*}
    Since $V(x)=E-\big(S_{E}^{V}(x)\big)_{11}$, it follows that $V\in C^{r}(\mathbb{T},\mathbb{R})$.
\end{proof}

\subsubsection{Pure point spectrum of Aubry dual operator}

\begin{proof}[Proof of {\hyperref[cor:intro-dual]{Corollary~\ref*{cor:intro-dual}}}]
    Set $\sigma_{*}=\frac{\sigma_{0}}{2}$ and let $\mathrm{d}N_{\alpha,V}$ be the density of states measure. Let $\mathscr{E}\subseteq\Sigma_{\alpha,V}$ be the set of energies for which $(\alpha,S_{E}^{V})$ is reducible to an elliptic constant by a conjugacy in $\mathcal{A}_{\sigma_{*}}^{\ell}$. By \hyperref[thm:intro-schrodinger]{Theorem~\ref*{thm:intro-schrodinger}} and the construction in \hyperref[KAM]{Theorem~\ref*{KAM}}, one has $\mu_{x,\alpha,V}(\Sigma_{\alpha,V}\setminus\mathscr{E})=0$ for every $x\in\mathbb{T}$. It follows that
    \begin{equation}\label{DSzero}
        N_{\alpha,V}(\Sigma_{\alpha,V}\setminus\mathscr{E})=\frac{1}{2}\int_{\mathbb{T}}\mu_{x,\alpha,V}(\Sigma_{\alpha,V}\setminus\mathscr{E})\,\mathrm{d}x=0.
    \end{equation}

    Fix any $E\in\mathscr{E}$. By \hyperref[thm:intro-schrodinger]{Theorem~\ref*{thm:intro-schrodinger}}, there exist $B_{E}\in \mathcal{A}_{\sigma_{*}}(2\mathbb{T}, \mathrm{SL}(2,\mathbb{R}))$ and $\phi_{E}\in \mathbb{T}$ such that 
    \begin{equation*}
        B_{E}(x+\alpha)^{-1}S_{E}^{V}(x)B_{E}(x)=R_{\phi_{E}}\coloneqq M^{-1}
        \begin{pmatrix}
            e^{2\pi i \phi_{E}}&0\\
            0&e^{-2\pi i \phi_{E}}
        \end{pmatrix}
        M.
    \end{equation*}
    Let $m_{E}=\deg B_{E}$. Note that $B_{E}(x+1)=(-1)^{m_{E}}B_{E}(x)$, one has $C_{E}(x)\coloneqq B_{E}(x)R_{-m_{E}x/2}\in \mathcal{A}_{\sigma_{*}}(\mathbb{T},\mathrm{SL}(2,\mathbb{R}))$ and $\deg C_{E}=0$. Moreover, $\phi_{E}= \varrho(E)-m_{E}\alpha/2 \bmod{\mathbb{Z}}$. Consequently,
    \begin{equation*}
        C_{E}(x+\alpha)^{-1}S_{E}^{V}(x)C_{E}(x) =R_{\varrho(E)}.
    \end{equation*}
    Let $\widetilde{C}_{E}=C_{E}M^{-1}\in\mathcal{A}_{\sigma_{*}}^{\ell}(\mathbb{T},\mathrm{SL}(2,\mathbb{C}))$. Then
    \begin{equation}\label{eq:degree-zero-reduction}
        \widetilde{C}_{E}(x+\alpha)^{-1}S_{E}^{V}(x)\widetilde{C}_{E}(x)=
        \begin{pmatrix}
            e^{2\pi i\varrho(E)}&0\\
            0&e^{-2\pi i\varrho(E)}
        \end{pmatrix},
    \end{equation}
    and
    \begin{equation*}
        \|\widetilde{C}_{E}\|_{L^{2}}\leqslant \|\widetilde{C}_{E}\|_{C^{0}}\leqslant \|\widetilde{C}_{E}\|_{\sigma_{*}}<\infty.
    \end{equation*}

    Let $\psi_{E}(x)=\big(\widetilde{C}_{E}(x)\big)_{11}$. Expanding the first column of \eqref{eq:degree-zero-reduction} gives
    \begin{equation}\label{eq:dual-bloch}
        e^{2\pi i\varrho(E)}\psi_{E}(x+\alpha)+e^{-2\pi i\varrho(E)}\psi_{E}(x-\alpha)+V(x)\psi_{E}(x)=E\psi_{E}(x).
    \end{equation}
    Write
    \begin{equation*}
        \psi_{E}(x)=\sum_{n\in\mathbb{Z}}u_{E}(n)e^{2\pi i nx}.
    \end{equation*}
    Since $\psi_{E}\in\mathcal{A}_{\sigma_{*}}^{\ell}(\mathbb{T},\mathbb{C})$, for every $0<\sigma<\sigma_{*}$,
    \begin{equation}\label{eq:adapted-eigenvector}
        \sum_{n\in\mathbb{Z}}|u_{E}(n)|e^{\sigma\ell(n)}<\infty.
    \end{equation}
    Comparing the Fourier coefficients in \eqref{eq:dual-bloch} yields
    \begin{equation*}
        \widehat{H}_{\alpha,\varrho(E),V}u_{E}=Eu_{E}.
    \end{equation*}
    Thus, by \eqref{DSzero}, \eqref{eq:degree-zero-reduction}, and \eqref{eq:adapted-eigenvector}, for $N_{\alpha,V}$-almost every $E\in\Sigma_{\alpha,V}$, the cocycle $(\alpha,S_{E}^{V})$ admits an $L^{2}$ degree zero reduction preserving its fibered rotation number. Therefore, \cite[Theorem 3.1]{MR3512893} implies that there exists a full Lebesgue measure set $\Theta\subseteq\mathbb{T}$ such that $\widehat{H}_{\alpha,\theta,V}$ has pure point spectrum for every $\theta\in\Theta$. 

    Hence, for every $\theta\in\Theta$, there exists an orthonormal eigenbasis $\{u^{(j)}\}_{j\geqslant 1}$ such that, by \eqref{eq:adapted-eigenvector},
    \begin{equation*}
        \sum_{n\in\mathbb{Z}}|u_{n}^{(j)}|e^{\sigma\ell(n)}<\infty
    \end{equation*}
    for every $j\geqslant 1$ and every $0<\sigma<\sigma_{*}$.

    We now prove the last assertion. Let $V\notin C^{r}(\mathbb{T},\mathbb{R})$ for some $0<r<1$. Define
    \begin{equation*}
        \Theta_{0}=\Theta\cap \{\theta\in\mathbb{T}: 2\theta\notin \alpha\mathbb{Z}\oplus\mathbb{Z}\},
    \end{equation*}
    which also has full Lebesgue measure. Fix any $\theta\in\Theta_{0}$ and an eigenfunction $u=u^{(j)}$. Suppose that for some $s>r$,
    \begin{equation*}
        \sum_{n\in\mathbb{Z}} (1+|n|)^{s}|u_{n}| < \infty.
    \end{equation*}
    Then
    \begin{equation*}
        \psi(x)\coloneqq \sum_{n\in\mathbb{Z}}u_{n}e^{2\pi i nx}\in C^{r}(\mathbb{T},\mathbb{C}).
    \end{equation*}
    Aubry duality implies
    \begin{equation*}
        e^{2\pi i\theta}\psi(x+\alpha)+e^{-2\pi i\theta}\psi(x-\alpha)+V(x)\psi(x)=E\psi(x),
    \end{equation*}
    and
    \begin{equation*}
        e^{-2\pi i\theta}\overline{\psi(x+\alpha)}+e^{2\pi i\theta}\overline{\psi(x-\alpha)}+V(x)\overline{\psi(x)}=E\overline{\psi(x)}.
    \end{equation*}
    
    Define $U(x)= \big(U_{+}(x),U_{-}(x)\big)$, where
    \begin{equation*}
        U_{+}(x)=
        \begin{pmatrix}
            \psi(x)\\
            e^{-2\pi i\theta} \psi(x-\alpha)
        \end{pmatrix},
        \qquad 
        U_{-}(x)=
        \begin{pmatrix}
            \overline{\psi(x)}\\
            e^{2\pi i\theta} \overline{\psi(x-\alpha)}
        \end{pmatrix}.
    \end{equation*}
    We claim that $\det U\not\equiv 0$. Otherwise, since $U_{+},U_{-}\not\equiv 0$, there exists a nonzero $c\in C^{0}(\mathbb{T},\mathbb{C})$ such that $U_{-}(x)=c(x)U_{+}(x)$. From
    \begin{equation*}
        S_{E}^{V}(x)U_{-}(x)=c(x)e^{2\pi i\theta} U_{+}(x+\alpha)
    \end{equation*}
    and
    \begin{equation*}
        S_{E}^{V}(x)U_{-}(x)=e^{-2\pi i\theta} U_{-}(x+\alpha)=e^{-2\pi i\theta} c(x+\alpha) U_{+}(x+\alpha),
    \end{equation*}
    we obtain
    \begin{equation*}
        c(x+\alpha)=e^{4\pi i\theta} c(x).
    \end{equation*}
    Comparing Fourier coefficients yields $e^{2\pi i k\alpha}\hat{c}(k)=e^{4\pi i\theta} \hat{c}(k)$ for all $k\in\mathbb{Z}$, which implies $2\theta\equiv k\alpha \bmod{\mathbb{Z}}$. This contradicts $\theta\in \Theta_{0}$. Thus $\det U\not\equiv 0$. Therefore, by \hyperref[prop:bloch-regularity]{Proposition~\ref*{prop:bloch-regularity}}, this implies $V\in C^{r}(\mathbb{T},\mathbb{R})$, contradicting the  assumption $V\notin C^{r}(\mathbb{T},\mathbb{R})$.

    Finally, if $V$ belongs to no positive H\"older class, then for any $s>0$ one may choose $0<r<\min\{s,1\}$ and apply the preceding conclusion.
\end{proof}

\subsubsection{Weierstrass potentials}

For the Weierstrass potential $V(x)=\lambda W_{a,b}(x)$, the dual operator in \eqref{eq:dual} takes the form
\begin{equation}\label{eq:dual-weierstrass}
    [\widehat{H}_{\alpha,\theta,\lambda W_{a,b}}u]_{n}=2\cos 2\pi(\theta+n\alpha)u_{n}+\frac{\lambda}{2}\sum_{m\geqslant 0}a^{m}(u_{n-b^{m}}+u_{n+b^{m}}).
\end{equation}
If $\lambda\neq 0$, the hopping coefficients satisfy
\begin{equation*}
    \widehat{V}(\pm b^{m})=\frac{\lambda}{2}a^{m}, \qquad m\geqslant 0.
\end{equation*}
Consequently, for every $r\geqslant 0$,
\begin{equation*}
    \sum_{k\in\mathbb{Z}}(1+|k|)^{r}|\widehat{V}(k)| \asymp |\lambda|\sum_{m\geqslant 0}(ab^{r})^{m},
\end{equation*}
and hence
\begin{equation*}
    \sum_{k\in\mathbb{Z}}(1+|k|)^{r}|\widehat{V}(k)|<\infty \quad\Longleftrightarrow\quad r<\beta,
\end{equation*}
where $\beta=-\frac{\log a}{\log b}\in(0,1)$ is the H\"older exponent. In particular, the hopping is absolutely summable but has no finite first moment. The adapted localization conclusion should not be confused with localization in the Euclidean lattice distance. Since $\ell_{b}(b^{m})\leqslant m+1$,
\begin{equation*}
    e^{\sigma\ell_{b}(b^{m})}\leqslant e^{\sigma}|b^{m}|^{\sigma/\log b}.
\end{equation*}
Thus the adapted exponential weight grows only polynomially along the hopping sites $b^{m}$. The divergence statement in \hyperref[cor:intro-weierstrass]{Corollary~\ref*{cor:intro-weierstrass}} shows that this distinction is genuine.

\begin{proof}[Proof of {\hyperref[cor:intro-weierstrass]{Corollary~\ref*{cor:intro-weierstrass}}}]
    Choose $\gamma>0$ such that $\alpha\in\mathrm{DC}_{b}(\gamma,\tau)$ and fix $\sigma_{0}=-\frac{1}{2}\log a$.
    By \hyperref[prop:W]{Proposition~\ref*{prop:W}}, one has $W_{a,b}\in\mathcal{A}_{\sigma_{0}}^{(b)}(\mathbb{T},\mathbb{R})$.
    Let $\varepsilon_{*}$ be the smallness constant in \hyperref[thm:intro-schrodinger]{Theorem~\ref*{thm:intro-schrodinger}} and choose
    \begin{equation*}
        \lambda_{0} = \frac{\varepsilon_{*}}{\|W_{a,b}\|_{\sigma_{0}}}.
    \end{equation*}
    Then $\|\lambda W_{a,b}\|_{\sigma_{0}}<\varepsilon_{*}$
    whenever $|\lambda|<\lambda_{0}$. The spectral conclusions of $H_{\alpha,x,\lambda W_{a,b}}$ follow from \hyperref[thm:intro-schrodinger]{Theorem~\ref*{thm:intro-schrodinger}}, while the adapted localization conclusion for the dual operator $\widehat{H}_{\alpha,\theta,\lambda W_{a,b}}$ follows from \hyperref[cor:intro-dual]{Corollary~\ref*{cor:intro-dual}}.

    Assume now that $\lambda\neq 0$ and let $s>\beta$. Choose  $\beta<r<\min\{s,1\}$.
    Since
    \begin{equation*}
        (b^{m})^{r} \big|\lambda\widehat{ W}_{a,b}(b^{m})\big| = \frac{|\lambda|}{2} (ab^{r})^{m} \to \infty,\quad \text{as } m\to\infty,
    \end{equation*}
    one has $\lambda W_{a,b}\notin C^{r}(\mathbb{T},\mathbb{R})$.
    Therefore, by \hyperref[cor:intro-dual]{Corollary~\ref*{cor:intro-dual}} again, no eigenfunction  can satisfy
    \begin{equation*}
        \sum_{n\in\mathbb{Z}} (1+|n|)^{s}|u_{n}|<\infty.
    \end{equation*}
    This proves the last assertion.
\end{proof}

\section{Cantor spectrum via the Moser--P\"oschel construction}

\subsection{Moser--P\"oschel construction}
We first provide an expansion for perturbation of parabolic constant.
\begin{lemma}\label{paraperturbation}
    Let $\bar{A}=\begin{pmatrix}
        1&\zeta\\
        0&1
    \end{pmatrix}$ and $D_{0}=\begin{pmatrix}
        0&\zeta\\
        0&0
    \end{pmatrix}$ with $\zeta\in \mathbb{R}$. Let $\alpha\in\mathrm{DC}(\gamma,\tau)$ and $P\in\mathcal{A}_{\sigma}(\mathbb{T},\mathrm{sl}(2,\mathbb{R}))$. For any $0<\sigma'<\sigma$, there exist $Y\in\mathcal{A}_{\sigma'}(\mathbb{T},\mathrm{sl}(2,\mathbb{R}))$, $D_{1}\in\mathrm{sl}(2,\mathbb{R})$, and $R_{t}\in\mathcal{A}_{\sigma'}(\mathbb{T},\mathrm{sl}(2,\mathbb{R}))$ such that, for every sufficiently small $|t|\leqslant t_{0}$,
    \begin{equation*}
        e^{-t Y(x+\alpha)} (\bar{A}e^{tP(x)})e^{tY(x)}=\exp \big(D_{0}+t D_{1}+t^{2} R_{t}(x)\big),
    \end{equation*}
    where 
    \begin{equation*}
        \begin{split}
            &\|Y\|_{\sigma'}
        \leqslant
        C(\zeta)\gamma^{-3}
        \exp\Big(
            C(\rho,\tau)(\sigma-\sigma')^{-\frac{\rho}{1-\rho}}
        \Big)
        \|P\|_{\sigma},\\
        &\sup_{|t|\leqslant t_{0}}\|R_{t}\|_{\sigma'}
        \leqslant
        C(\zeta,\gamma,\rho,\tau,\sigma,\sigma')\|P\|_{\sigma}^{2}.
        \end{split}
    \end{equation*}
\end{lemma}

\begin{proof}
    We first solve the linearized equation
    \begin{equation}\label{linear}
        \bar{A}^{-1}Y(x+\alpha)\bar{A}-Y(x)
        =
        P(x)-\widehat{P}(0),
    \end{equation}
    with $\widehat{Y}(0)=0$. Comparing the Fourier coefficients on both sides, a direct calculation gives, for every $n\neq0$,
    \begin{equation*}
        \|\widehat{Y}(n)\|
        \leqslant
        C(\zeta)\|n\alpha\|_{\mathbb{T}}^{-3}\|\widehat{P}(n)\|.
    \end{equation*}
    Since $\alpha\in\mathrm{DC}(\gamma,\tau)$,
    \begin{equation*}
        \|\widehat{Y}(n)\|
        \leqslant
        C(\zeta)\gamma^{-3}
        \exp\big(3\tau\ell(n)^{\rho}\big)
        \|\widehat{P}(n)\|.
    \end{equation*}
    Therefore,
    \begin{equation*}
        \begin{split}
            \|Y\|_{\sigma'}
            &\leqslant
            C(\zeta)\gamma^{-3}
            \sum_{n\neq0}
            \|\widehat{P}(n)\|
            \exp\big(\sigma'\ell(n)+3\tau\ell(n)^{\rho}\big)\\
            &\leqslant
            C(\zeta)\gamma^{-3}
            \|P\|_{\sigma}
            \sup_{r\geqslant0}
            \exp\big(-(\sigma-\sigma')r+3\tau r^{\rho}\big).
        \end{split}
    \end{equation*}
    Since
    \begin{equation*}
        \sup_{r\geqslant0}
        \big(-(\sigma-\sigma')r+3\tau r^{\rho}\big)
        \leqslant
        C(\rho,\tau)(\sigma-\sigma')^{-\frac{\rho}{1-\rho}},
    \end{equation*}
    we obtain
    \begin{equation}\label{Yestimate}
        \|Y\|_{\sigma'}
        \leqslant
        C(\zeta)\gamma^{-3}
        \exp\Big(
            C(\rho,\tau)(\sigma-\sigma')^{-\frac{\rho}{1-\rho}}
        \Big)
        \|P\|_{\sigma}.
    \end{equation}

    Since $e^{D_{0}}=\bar{A}$, equation \eqref{linear} gives
    \begin{equation*}
        \begin{split}
            e^{-tY(x+\alpha)}\bar{A}e^{tP(x)}e^{tY(x)}
            &=
            \bar{A}
            e^{-t\bar{A}^{-1}Y(x+\alpha)\bar{A}}
            e^{tP(x)}
            e^{tY(x)}\\
            &=
            e^{D_{0}}
            e^{-t(Y(x)+P(x)-\widehat{P}(0))}
            e^{tP(x)}
            e^{tY(x)}.
        \end{split}
    \end{equation*}
    By the BCH formula, there exists a uniformly bounded family $Z_{t}\in\mathcal{A}_{\sigma'}(\mathbb{T},\mathrm{sl}(2,\mathbb{R}))$ such that
    \begin{equation*}
        \log\big(
            e^{-t(Y(x)+P(x)-\widehat{P}(0))}
            e^{tP(x)}
            e^{tY(x)}
        \big)
        =
        t\widehat{P}(0)+t^{2}Z_{t}(x),
    \end{equation*}
    with
    \begin{equation*}
        \sup_{|t|\leqslant t_{0}}\|Z_{t}\|_{\sigma'}
        \leqslant
        C\big(\|Y\|_{\sigma'}+\|P\|_{\sigma'}\big)^{2}.
    \end{equation*}
    Applying the local BCH formula at $D_{0}$ and using $[D_{0},[D_{0},[D_{0},D]]]=0$ for every $D\in\mathrm{sl}(2,\mathbb{R})$, we obtain
    \begin{equation*}
        \log\big(
            e^{D_{0}}e^{t\widehat{P}(0)+t^{2}Z_{t}(x)}
        \big)
        =
        D_{0}+tD_{1}+t^{2}R_{t}(x),
    \end{equation*}
    where
    \begin{equation*}
        D_{1}
        =
        \widehat{P}(0)
        +
        \frac{1}{2}[D_{0},\widehat{P}(0)]
        +
        \frac{1}{12}[D_{0},[D_{0},\widehat{P}(0)]]
        \in\mathrm{sl}(2,\mathbb{R}),
    \end{equation*}
    and
    \begin{equation*}
        \sup_{|t|\leqslant t_{0}}\|R_{t}\|_{\sigma'}
        \leqslant
        C(\zeta)\big(\|Y\|_{\sigma'}+\|P\|_{\sigma'}\big)^{2}.
    \end{equation*}
    The estimate for $R_{t}$ now follows from \eqref{Yestimate}.
\end{proof}

We give a criterion for uniform hyperbolicity. Denote $[f]\coloneqq\widehat{f}(0)$ for simplicity. 

\begin{proposition}\label{MP}
    Under the assumptions of \hyperref[SOKAM]{Theorem~\ref*{SOKAM}}, let $W\in\mathcal{A}_{\sigma}(\mathbb{T},\mathbb{R})$, $0<\sigma'<\sigma$, and assume that there exist
    \begin{equation*}
        B(x)=
        \begin{pmatrix}
            b_{11}(x)&b_{12}(x)\\
            b_{21}(x)&b_{22}(x)
        \end{pmatrix}
        \in
        \mathcal{A}_{\sigma'}(2\mathbb{T},\mathrm{SL}(2,\mathbb{R})),
    \end{equation*}
    $\chi\in\{\pm1\}$, and $\zeta\in\mathbb{R}$ such that
    \begin{equation*}
        B(x+\alpha)^{-1}S_{E}^{V}(x)B(x)=\chi \begin{pmatrix}
            1&\zeta\\
            0&1
        \end{pmatrix}.
    \end{equation*}
    Then for sufficiently small $|t|>0$, the following holds.
    \begin{enumerate}
        \item \label{item:eq0} If $\zeta\neq 0$ and $[Wb_{11}^{2}]\neq 0$, then $(\alpha, S_{E}^{V+tW})\in\mathcal{UH}$ if  $t\zeta[Wb_{11}^{2}]>0$;

        \item \label{item:neq0}If $\zeta =0$ and $-[W b_{11}b_{12}]^{2}+[Wb_{12}^{2}] [Wb_{11}^{2}]<0$, then $(\alpha, S_{E}^{V+tW})\in\mathcal{UH}$.
    \end{enumerate}
\end{proposition}
\begin{proof}
Write
\begin{equation*}
    S_{E}^{V+tW}(x)=S_{E}^{V}(x) \exp \begin{pmatrix}
        0&0\\
        tW(x)&0
    \end{pmatrix},\quad B(x)=\begin{pmatrix}
    b_{11}(x)&b_{12}(x)\\
    b_{21}(x)&b_{22}(x)
    \end{pmatrix}.
\end{equation*}
    By \hyperref[item:rational]{Corollary~\ref*{PE}\,\textup{(\ref*{item:rational})}}, there exists $\chi\in \{\pm 1\}$, such that
    \begin{equation*}
        B(x+\alpha)^{-1}S_{E}^{V+tW}(x)B(x)=\chi \begin{pmatrix}
        1&\zeta\\
        0&1
    \end{pmatrix}e^{t P(x)} =\chi \bar{A}e^{t P(x)},
    \end{equation*}
    where $P\in\mathcal{A}_{\sigma'}(\mathbb{T},\mathrm{sl}(2,\mathbb{R}))$ satisfies
    \begin{equation*}
        P(x)=B(x)^{-1}\begin{pmatrix}
            0&0\\
            W(x)&0
        \end{pmatrix} B(x)=W(x)\begin{pmatrix}
            -b_{11}(x)b_{12}(x)&-b_{12}(x)^{2}\\
            b_{11}(x)^{2}&b_{11}(x)b_{12}(x)
        \end{pmatrix}.
    \end{equation*}
    By \hyperref[paraperturbation]{Lemma~\ref*{paraperturbation}}, for any $0<\sigma''<\sigma'$, there exist $Y\in\mathcal{A}_{\sigma''}(\mathbb{T},\mathrm{sl}(2,\mathbb{R}))$, $D_{1}\in\mathrm{sl}(2,\mathbb{R})$, and uniformly bounded $R_{t}\in\mathcal{A}_{\sigma''}(\mathbb{T},\mathrm{sl}(2,\mathbb{R}))$ such that
    \begin{equation*}
        e^{-tY(x+\alpha)}\bar{A}e^{t P(x)} e^{tY(x)} =\exp(D_{0}+tD_{1}+t^{2}R_{t}(x)),
    \end{equation*}
    where 
    \begin{equation}\label{D1formula}
        D_{1}=\widehat{P}(0)+\frac{1}{2}[D_{0},\widehat{P}(0)]+\frac{1}{12}[D_{0},[D_{0},\widehat{P}(0)]].
    \end{equation}
    Let $B_{t}(x)=B(x)e^{tY(x)}$. Then we have
    \begin{equation*}
        B_{t}(x+\alpha)^{-1}S_{E}^{V+tW}(x)B_{t}(x)=\chi \exp(D_{0}+tD_{1}+t^{2}R_{t}(x)).
    \end{equation*}
    Denote $a=[Wb_{11}b_{12}]$, $b=[W b_{12}^{2}]$, $c=[W b_{11}^{2}]$. Then $\widehat{P}(0)=\begin{pmatrix}
        -a&-b\\
        c&a
    \end{pmatrix}$. 
    By \eqref{D1formula}, the direct calculation shows
    \begin{equation*}
        D_{1}=\begin{pmatrix}
            -a+\zeta c/2&-b+\zeta a-\zeta^{2}c/6\\
            c&a-\zeta c/2
        \end{pmatrix}.
    \end{equation*}

Now we consider two cases.

If $\zeta\neq 0$ and $c\neq 0$, then for sufficiently small $|t|>0$ and $t \zeta c>0$, 
    \begin{equation*}
        \det (D_{0}+tD_{1})=-t \zeta c+O(|t|^{2})<0,
    \end{equation*}
    which means that $D_{0}+tD_{1}$ has two real eigenvalues $\{\pm \lambda_{t}\}$ with $\lambda_{t}\asymp |t|^{1/2}$. By the hyperbolic diagonalization in \cite[Proposition 18]{MR2199393}, there exists $Q_{t}\in\mathrm{SL}(2,\mathbb{R})$ with $\|Q_{t}\|\leqslant C(\zeta)\lambda_{t}^{-1/2}$ such that
    \begin{equation*}
        Q_{t}^{-1} (D_{0}+tD_{1}+t^{2}R_{t})Q_{t}=\begin{pmatrix}
            \lambda_{t}&0\\
            0&-\lambda_{t}
        \end{pmatrix}+O(|t|^{3/2}).
    \end{equation*}
    Thus $(\alpha, S_{E}^{V+tW})\in\mathcal{UH}$. 

    If $\zeta=0$ and $-a^{2}+bc<0$, then $D_{0}=0$ and $D_{1}=\widehat{P}(0)=\begin{pmatrix}
        -a&-b\\
        c&a
    \end{pmatrix}$. The direct calculation shows
    \begin{equation*}
        \det(D_{0}+tD_{1})=t^{2}\det D_{1}= t^{2}(-a^{2}+bc)<0.
    \end{equation*}
    There exist $\lambda=\sqrt{a^{2}-bc}>0$ and $Q\in \mathrm{SL}(2,\mathbb{R})$ such that
    \begin{equation*}
        Q^{-1}(D_{0}+tD_{1}+t^{2}R_{t})Q=t\begin{pmatrix}
            \lambda&0\\
            0&-\lambda
        \end{pmatrix}+ O(|t|^{2}).
    \end{equation*}
    Thus $(\alpha, S_{E}^{V+tW})\in\mathcal{UH}$.
\end{proof}

\subsection{Cantor spectrum for H\"older continuous potential}
\begin{proof}[Proof of {\hyperref[cor:intro-cantor]{Corollary~\ref*{cor:intro-cantor}}}]
    Define a Banach space as
    \begin{equation*}
        \mathcal{X}_{\sigma_{0}}^{1}=C^{1}(\mathbb{T},\mathbb{R})\cap\mathcal{A}_{\sigma_{0}}^{(b)}(\mathbb{T},\mathbb{R}),
    \end{equation*}
    with norm
    \begin{equation*}
        \|g\|_{\mathcal{X}_{\sigma_{0}}^{1}}=\|g\|_{\sigma_{0}}+\|g\|_{C^{1}}\quad \text{for every }g\in \mathcal{X}_{\sigma_{0}}^{1}.
    \end{equation*}
 
    Let $\varepsilon_{*}$ be the smallness constant in \hyperref[SOKAM]{Theorem~\ref*{SOKAM}}. Since $\lambda$ is sufficiently small, choose $\delta_{*}=\frac{1}{2}(\varepsilon_{*}-\|\lambda W_{a,b}\|_{\sigma_{0}})$. For any $0<\delta<\delta_{*}$, define
    \begin{equation*}
        \mathcal{U}_{\delta}=\big\{g\in\mathcal{X}_{\sigma_{0}}^{1}: \|g\|_{\mathcal{X}_{\sigma_{0}}^{1}}<\delta\big\}.
    \end{equation*}
    Note that $\mathcal{U}_{\delta}$ is an open subset of the Banach space $\mathcal{X}_{\sigma_{0}}^{1}$. For every gap label $k\in\mathbb{Z}\setminus\{0\}$, denote by $G_{k}(V)$ the gap for label $k$ for the potential $V$. For every $k\in\mathbb{Z}\setminus\{0\}$, define
    \begin{equation*}
        \mathcal{O}_{k,\delta}=\big\{g\in\mathcal{U}_{\delta}: |G_{k}(\lambda W_{a,b}+g)|>0\big\}.
    \end{equation*}

    We first show that $\mathcal{O}_{k,\delta}$ is open in $\mathcal{U}_{\delta}$. Let $g_{0}\in \mathcal{O}_{k,\delta}$ and $V_{0}=\lambda W_{a,b}+g_{0}$ and fix any $E_{0}\in G_{k}(V_{0})$. Johnson's theorem implies $(\alpha,S_{E_{0}}^{V_{0}})\in\mathcal{UH}$. Since uniform hyperbolicity is an open condition in the $C^{0}$ topology, there exists $\eta>0$ such that if $\|g_{1}-g_{0}\|_{\mathcal{X}_{\sigma_{0}}^{1}}<\eta$, then $(\alpha, S_{E_{0}}^{V_{1}})\in\mathcal{UH}$ where $V_{1}=\lambda W_{a,b}+g_{1}$. Moreover, for $s\in [0,1]$, let $V_{s}=\lambda W_{a,b}+(1-s)g_{0}+sg_{1}$. By \hyperref[banachalgebra]{Proposition~\ref*{banachalgebra}}, we have
    \begin{equation*}
        \|V_{s}-V_{0}\|_{\mathcal{X}_{\sigma_{0}}^{1}}\leqslant s\|g_{1}-g_{0}\|_{\mathcal{X}_{\sigma_{0}}^{1}}<\eta,
    \end{equation*}
    which implies $(\alpha,S_{E_{0}}^{V_{s}})\in\mathcal{UH}$ for every $s\in [0,1]$. Thus, for every $s\in [0,1]$, there exists $m\in\mathbb{Z}$ such that $E_{0}\in G_{m}(V_{s})$ with $|G_{m}(V_{s})|>0$. By the gap-labeling theorem, $2\varrho(\alpha,S_{E_{0}}^{V_{s}}) \in \{m\alpha \bmod \mathbb{Z}: m\in\mathbb{Z}\}$. On the other hand, since the rotation number is continuous with respect to the cocycle, the map $s\mapsto 2\varrho(\alpha,S_{E_{0}}^{V_{s}})$ is continuous. This implies $2\varrho(\alpha,S_{E_{0}}^{V_{1}})=2\varrho(\alpha,S_{E_{0}}^{V_{0}})=k\alpha\bmod\mathbb{Z}$. Therefore, $E_{0}\in G_{k}(V_{1})$, and thus $g_{1}\in\mathcal{O}_{k,\delta}$. Hence, $\mathcal{O}_{k,\delta}$ is open.

    We next show that $\mathcal{O}_{k,\delta}$ is dense in $\mathcal{U}_{\delta}$. Fix $g_{0}\in\mathcal{U}_{\delta}$ and $\eta>0$, and let
    \begin{equation*}
        V_{0}=\lambda W_{a,b}+g_{0}.
    \end{equation*}
    If $g_{0}\in\mathcal{O}_{k,\delta}$, there is nothing to prove. Assume that $|G_{k}(V_{0})|=0$, and let $E$ be the collapsed gap energy satisfying
    \begin{equation*}
        2\varrho(\alpha,S_{E}^{V_{0}})=k\alpha\bmod\mathbb{Z}.
    \end{equation*}
    By \hyperref[SOKAM]{Theorem~\ref*{SOKAM}} and \hyperref[item:eq0]{Proposition~\ref*{MP}\,\textup{(\ref*{item:eq0})}}, there exists $B(x)=\begin{pmatrix}
            b_{11}(x)&b_{12}(x)\\
            b_{21}(x)&b_{22}(x)
        \end{pmatrix}$
    such that
    \begin{equation*}
        B(x+\alpha)^{-1}S_{E}^{V_{0}}(x)B(x)=\pm\mathrm{I}.
    \end{equation*}
    Denote
    \begin{equation*}
        y_{1}=\frac{b_{11}^{2}+b_{12}^{2}}{2},\qquad y_{2}=\frac{b_{11}^{2}-b_{12}^{2}}{2},\qquad y_{3}=b_{11}b_{12}.
    \end{equation*}
    Since $\det B=1$, we have $[y_{1}]>0$.  Since the real trigonometric polynomials are dense in $\mathcal{A}_{\sigma_{0}}^{(b)}$, there exists a real trigonometric polynomial $W_{0}$ such that
    \begin{equation}\label{C1genericor}
        -[W_{0}y_{1}][y_{2}]+[W_{0}y_{2}][y_{1}]\neq 0 \quad\text{or}\quad -[W_{0}y_{1}][y_{3}]+[W_{0}y_{3}][y_{1}]\neq 0.
    \end{equation}
    In particular, $W_{0}\in\mathcal{X}_{\sigma_{0}}^{1}$. Let
    \begin{equation*}
        e_{0}=\frac{[W_{0}y_{1}]}{[y_{1}]},\qquad W=W_{0}-e_{0}.
    \end{equation*}
    Then $[Wy_{1}]=0$, and \eqref{C1genericor} implies that
    \begin{equation*}
        [Wy_{2}]\neq 0 \quad\text{or}\quad [Wy_{3}]\neq 0.
    \end{equation*}
    Consequently,
    \begin{equation*}
        \begin{split}
            -[Wb_{11}b_{12}]^{2}+[Wb_{12}^{2}][Wb_{11}^{2}]&=[Wy_{1}]^{2}-[Wy_{2}]^{2}-[Wy_{3}]^{2}\\
            &=-[Wy_{2}]^{2}-[Wy_{3}]^{2}<0.
        \end{split}
    \end{equation*}
    By \hyperref[item:neq0]{Proposition~\ref*{MP}\,\textup{(\ref*{item:neq0})}}, one has $(\alpha,S_{E}^{V_{0}+tW})\in\mathcal{UH}$ for every sufficiently small $|t|>0$. Since $S_{E}^{V_{0}+tW}=S_{E+te_{0}}^{V_{0}+tW_{0}}$, it follows that $(\alpha,S_{E+te_{0}}^{V_{0}+tW_{0}})\in\mathcal{UH}$. By the continuity of the rotation number and the gap-labeling theorem, for every sufficiently small $|t|>0$, $2\varrho(\alpha,S_{E+te_{0}}^{V_{0}+tW_{0}})=k\alpha\bmod\mathbb{Z}$. Hence the gap with label $k$ is open for $V_{0}+tW_{0}$. Choosing $|t|>0$ sufficiently small such that
    \begin{equation*}
        |t|\|W_{0}\|_{\mathcal{X}_{\sigma_{0}}^{1}}<\min\{\eta,\,
        \delta-\|g_{0}\|_{\mathcal{X}_{\sigma_{0}}^{1}}\},\qquad \|g_{0}+tW_{0}\|_{\sigma_{0}}<\delta,
    \end{equation*}
    we obtain $g_{0}+tW_{0}\in\mathcal{O}_{k,\delta}$ and $\|g_{0}+tW_{0}-g_{0}\|_{\mathcal{X}_{\sigma_{0}}^{1}}<\eta$. Thus $\mathcal{O}_{k,\delta}$ is dense in $\mathcal{U}_{\delta}$.

    Define
    \begin{equation*}
        \mathcal{G}_{\delta}=\bigcap_{k\in\mathbb{Z}\setminus\{0\}}\mathcal{O}_{k,\delta}.
    \end{equation*}
    By the Baire category theorem, $\mathcal{G}_{\delta}$ is dense in $\mathcal{U}_{\delta}$. Choose $g\in\mathcal{G}_{\delta}$ and set $V=\lambda W_{a,b}+g$.
    Then $V\in\mathcal{A}_{\sigma_{0}}^{(b)}(\mathbb{T},\mathbb{R})$ and
    \begin{equation*}
        \|V-\lambda W_{a,b}\|_{\sigma_{0}}=\|g\|_{\sigma_{0}}<\delta,
    \end{equation*}
    and every bounded gap is open for $V$. Moreover, since $\Sigma_{\alpha,V}$ is compact and has empty interior, it is a Cantor set.

    Since $g\in C^{1}(\mathbb{T},\mathbb{R})$ and $\lambda\neq 0$, the function $V$ is nowhere differentiable. Since the Weierstrass function is H\"older continuous, $V$ is H\"older continuous with exponent $\beta=-\frac{\log a}{\log b}$. Finally, since
    \begin{equation*}
        \|V\|_{\sigma_{0}}\leqslant \|\lambda W_{a,b}\|_{\sigma_{0}}+\|g\|_{\sigma_{0}}<\varepsilon_{*},
    \end{equation*}
    by \hyperref[thm:intro-schrodinger]{Theorem~\ref*{thm:intro-schrodinger}}, $H_{x,\alpha,V}$ has purely absolutely continuous spectrum for every $x\in\mathbb{T}$.
\end{proof}

\section{Extensions and an all-frequency problem}\label{sec:scope}

The adapted Fourier framework is developed here for one-frequency $\mathrm{SL}(2,\mathbb{R})$ cocycles. We conclude with three natural directions and questions suggested by the present theory.

\subsection{Multi-frequencies and higher-dimensional cocycles}

At the level of the weighted Fourier estimates, the passage to several frequencies is formal, but a complete KAM theorem is not automatic. One replaces $\mathbb{Z}$ by $\mathbb{Z}^{d}$, chooses a proper symmetric subadditive length $\ell$ with power-subexponential ball growth, and assumes the matched arithmetic condition
\begin{equation*}
    \|\langle n,\alpha\rangle\|_{\mathbb{T}}
    \geqslant
    \gamma\exp\big(-\tau\ell(n)^{\rho}\big),
    \qquad
    n\in\mathbb{Z}^{d}\setminus\{0\}.
\end{equation*}
The Banach algebra estimates, Fourier truncation bounds, and linear homological estimates retain the same form. The main additional issue is the organization of lattice resonances in the nonlinear iteration. We do not pursue this extension here, since it would require substantial extra notation without changing the basic adapted Fourier mechanism.

The weighted estimates extend to matrix-valued functions in any fixed finite dimension. For cocycles taking values in a fixed finite-dimensional linear Lie group, one expects the same framework to combine with the usual normal form or block-resonance decompositions from perturbative analytic theory; see \cite{MR1168974,MR1858550,MR4942811}. The new work would lie in the finite-dimensional resonance and normal-form analysis. We restrict the paper to $\mathrm{SL}(2,\mathbb{R})$ in order to keep the main mechanism transparent.

\subsection{Quasiperiodic linear differential equations}
For quasiperiodic linear differential equations, the weighted algebra remains compatible with the nonlinear matrix operations, while the discrete homological operator is replaced by
\begin{equation*}
\omega\cdot\partial_{x}X-[A,X].
\end{equation*}
Accordingly, the small divisors are of the form $i\langle n,\omega\rangle+\lambda_{k}-\lambda_{j}$. The nonlinear algebraic operations cause no additional loss of Fourier radius. At the matrix level they are products and commutators, which preserve the adapted Fourier algebra at the same radius. Equivalently, in the Hamiltonian formulation of a linear system, the relevant Hamiltonians are quadratic in the fiber variables, and their Poisson bracket reduces to the corresponding matrix commutator. Thus
\begin{equation*}
\|\{F,G\}\|_{\sigma}\leqslant C\|F\|_{\sigma}\|G\|_{\sigma}
\end{equation*}
with no reduction of $\sigma$. The passage from $\sigma$ to a smaller radius $\sigma'<\sigma$ is required only when inverting the homological operator, in order to compensate for the adapted small divisors. These estimates show that the adapted Fourier scale is compatible with the continuous-time problem. A complete differential-equation theorem would still require its own normal-form and resonance analysis, as well as a separate iteration statement; it is not proved here. The absence of derivative loss in the nonlinear algebraic operations is specific to linear systems, or equivalently quadratic Hamiltonian systems. For general nonlinear Hamiltonians, the derivatives in the Poisson bracket normally cause an additional loss of radius.

\subsection{The all-frequency problem}

The main theorem is perturbative for each fixed adapted Diophantine class, and its smallness threshold depends on the corresponding arithmetic constant $\gamma$, as is natural in the multifrequency setting. In one-frequency case, however, a natural problem is to remove this dependence and obtain a threshold that is uniform over irrational frequencies. This would be the adapted Fourier-algebraic analogue of nonperturbative almost reducibility in the analytic one-frequency setting \cite{MR2969277,MR2846380,AvilaARC}.

\begin{problem}[All-frequency almost reducibility]
Fix an adapted Fourier algebra scale $\mathcal{A}^{\ell}$ satisfying the assumptions of the main theorem. Given
$A\in\mathrm{SL}(2,\mathbb{R})$ and $\sigma>0$, does there exist
\begin{equation*}
    \varepsilon_{*}=\varepsilon_{*}(A,\sigma)>0
\end{equation*}
such that, for every irrational $\alpha$ and every $F\in\mathcal{A}_{\sigma}^{\ell}(\mathbb{T},\mathrm{sl}(2,\mathbb{R}))$ satisfying $\|F\|_{\sigma,\ell}<\varepsilon_{*}$, there exists $\sigma_{*}=\sigma_{*}(\alpha,A,F)>0$ for which the cocycle $(\alpha,Ae^{F})$ is almost reducible in $\mathcal{A}_{\sigma_{*}}^{\ell}$?
\end{problem}

\section*{Acknowledgments}
Jiangong You is supported by NSFC grant (12531006, 12526201) and Nankai Zhide Foundation.

\bibliographystyle{alpha}
\bibliography{main}
\end{document}